\documentclass[12pt]{imsart}    %for arXiv
\usepackage{a4wide}

\RequirePackage{amsthm,amsmath,amsfonts,amssymb}
\RequirePackage[numbers]{natbib}
\RequirePackage[colorlinks,citecolor=blue,urlcolor=blue]{hyperref}
\RequirePackage{graphicx}

\startlocaldefs
\newtheorem{theorem}{Theorem}[section]
\newtheorem{lemma}[theorem]{Lemma}
\newtheorem{corollary}[theorem]{Corollary}
\newtheorem{assumption}[theorem]{Assumption}
\newtheorem{proposition}[theorem]{Proposition}
\theoremstyle{remark}
\newtheorem{definition}[theorem]{Definition}
\newtheorem*{example}{Example}

\newtheorem*{remark}{Remark}
\usepackage{wasysym}
\usepackage{enumerate}
\usepackage{graphicx}
\usepackage{subcaption}
\newcommand{\ul}{\underline}

\newcommand{\E}{\mathbb{E}}
\newcommand{\indep}{\rotatebox[origin=c]{90}{$\models$}}
\usepackage{array}

\endlocaldefs

\begin{document}

\begin{frontmatter}
\title{Central limit theorem for linear spectral statistics of block-Wigner-type matrices}
%\title{A sample article title with some additional note\thanksref{t1}}
\runtitle{CLT for LSS of block-Wigner-type matrices}
%\thankstext{T1}{A sample additional note to the title.}

\begin{aug}
%%%%%%%%%%%%%%%%%%%%%%%%%%%%%%%%%%%%%%%%%%%%%%
%%Only one address is permitted per author. %%
%%Only division, organization and e-mail is %%
%%included in the address.                  %%
%%Additional information can be included in %%
%%the Acknowledgments section if necessary. %%
%%%%%%%%%%%%%%%%%%%%%%%%%%%%%%%%%%%%%%%%%%%%%%
\author[A]{\fnms{Zhenggang} \snm{Wang}\ead[label=e1]{u3005312@connect.hku.hk}},
\and \author[B]{\fnms{Jianfeng} \snm{Yao}\ead[label=e2]{jeffyao@hku.hk}}
%\and
%\author[B]{\fnms{Third} \snm{Author}\ead[label=e3,mark]{third@somewhere.com}}
%%%%%%%%%%%%%%%%%%%%%%%%%%%%%%%%%%%%%%%%%%%%%%
%% Addresses                                %%
%%%%%%%%%%%%%%%%%%%%%%%%%%%%%%%%%%%%%%%%%%%%%%
\address[A]{Department of Statistics and Actuarial Science,
The University  of Hong Kong,
\printead{e1}}

\address[B]{Department of Statistics and Actuarial Science,
	The University  of Hong Kong,
\printead{e2}}
\end{aug}

\begin{abstract}
Motivated by the stochastic block model, we investigate a class of Wigner-type matrices with certain block structures, and establish a CLT for the corresponding linear spectral statistics via the large-deviation bounds from local law and the cumulant expansion formula. We apply the results to the  stochastic block model. Specifically,  a class of renormalized adjacency matrices will be  block-Wigner-type matrices.  Further, we show that for certain estimator of such renormalized adjacency matrices, which will be no longer Wigner-type but share long-range non-decaying weak correlations among the entries,  the linear spectral statistics of such estimators will still share the same limiting behavior as those of the block-Wigner-type matrices, thus enabling hypothesis testing about stochastic block model. 
\end{abstract}



\end{frontmatter}
%%%%%%%%%%%%%%%%%%%%%%%%%%%%%%%%%%%%%%%%%%%%%%
%% Please use \tableofcontents for articles %%
%% with 50 pages and more                   %%
%%%%%%%%%%%%%%%%%%%%%%%%%%%%%%%%%%%%%%%%%%%%%%
\tableofcontents

%
%This template helps you to create a properly formatted \LaTeXe\ manuscript.
%%%%%%%%%%%%%%%%%%%%%%%%%%%%%%%%%%%%%%%%%%%%%%%
%%% `\ ' is used here because TeX ignores    %%
%%% spaces after text commands.              %%
%%%%%%%%%%%%%%%%%%%%%%%%%%%%%%%%%%%%%%%%%%%%%%%
%Prepare your paper in the same style as used in this sample .pdf file.
%Try to avoid excessive use of italics and bold face.
%Please do not use any \LaTeXe\ or \TeX\ commands that affect the layout
%or formatting of your document (i.e., commands like \verb|\textheight|,
%\verb|\textwidth|, etc.).

\section{Introduction}

%part 1 introduce the problem we are concerned with, wigner, wigner type
%part 1.1  
The investigation into the limiting properties of large random matrices has been popular for over two decades. Many techniques \cite{Bai1999}\cite{Bai2009}\cite{Benaych-Georges2016}\cite{Erdoes2011} are developed to solve problems in this area. There are plenty of objects of interest, namely the empirical spectral distribution (ESD), the limiting spectral distribution(LSD), the largest eigenvalue, the linear spectral statistics (LSS), the eigenvector statistics, etc. Particularly, 
the linear spectral statistics have attracted lots of attention ever since the 90s \cite{Sinai}. Various methods are explored to study the behavior of the LSS,  such as moment method \cite{Anderson2006}, martingale difference method \cite{bai2005}\cite{Bai2004},  cumulant expansion method \cite{Khorunzhy1996}. Also there is progress from the stochastic calculus \cite{Costin1995}\cite{Guionnet2002a}  and free probability \cite{Mingo2006}. Further, \cite{Chatterjee2006}\cite{Chatterjee2009} generalize the Stein method and use second order Poincar\'e inequalities to prove a CLT for the LSS. Specifically, in recent years,  a more in-depth understanding of the behavior of the LSS of  Wigner and Wigner-type matrices has been achieved by researchers from various perspectives.  \cite{Lytova2009a} introduces an interpolation method for more general Wigner matrices than the ones that share the same cumulants with GOE/GUE. \cite{Male2013} extends the CLT to certain heavy-tailed random matrices. More recently, \cite{He2017a} studies the mesoscopic eigenvalue statistics of the Wigner matrices via the Green function and the local law, \cite{Cipolloni2020} yields a thorough analysis of fluctuations of regular functions of Wigner matrices and \cite{bao2021quantitative} establishes a near-optimal convergence rate for the CLT of LSS of Wigner matrices.

%part 1.3 sbm(motivation and recent progress)
\par
In the meantime, motivations are drawn from social networks and other associated random graph models, which brings the researchers' attention to more involved matrix models.  One of the most classic models in this field is the stochastic block model (SBM).  In contrast to the Erd\H{o}s-Renyi model in which all nodes are exchangeable, the SBM introduces inhomogeneity by dividing the nodes into different communities. In the SBM with nodes $V$ and edges $E$, all edges are undirected, and different edges are independent, in the meantime, the probability that two nodes $v_i,v_j\in V$ connect with each other is only determined by which communities $v_i$ and $v_j$ belong to. In other words, the adjacency matrix of the SBM can be viewed as a random 0-1 matrix whose entries have block-wise constant expectations.  Thus, the centered adjacency matrices of the SBMs are  Wigner-type matrices with inhomogeneous variance profiles. 
\par
One of the most important questions in the SBM is community detection, which is to recover the community structure underneath via one single observation of the adjacency matrix. Further, an induced problem is to determine the number of communities. For most community detection algorithms, the number of communities needs to be given a priori as a hyperparameter. This motivates hypothesis testing for this parameter via the distributional information of certain test statistics of the model.  \cite{Lei2016} proposes a sequential test for  the renormalized adjacency matrix  $\left(\frac{A_{ij}-{p}_{ij}}{\sqrt{np_{ij}(1-p_{ij})}}\right)_{ij}$ and $\left(\frac{A_{ij}-\hat{p}_{ij}}{\sqrt{n\hat{p}_{ij}(1-\hat{p}_{ij})}}\right)_{ij}$ based on the Tracy-Widom fluctuation of the largest eigenvalue. In the same spirit, more recently in \cite{Banerjee2017}, Banerjee and Ma propose a hypothesis testing for the community structure via the LSS of the renormalized adjacency matrix $\left(\frac{A_{ij}-\hat{p}_{ij}}{\sqrt{n\hat{p}_{ij}(1-\hat{p}_{ij})}}\right)_{ij}$ with the method of moments approach \cite{Anderson2006} in the cases where the SBM has only one community or two asymptotically equal-sized communities. Towards another end, \cite{Adhikari2019a} proves a CLT for the LSS of general Wigner-type matrices via the second order Poincar\'e inequality without providing the explicit formulas for the asymptotic  mean and covariance function. 

In this paper, We establish CLTs for the class of block-Wigner-type matrices which is motivated by the renormalization $\left(\frac{A_{ij}-p_{ij}}{\sqrt{n}}\right)$ as well as a correlated matrix model induced from the renormalization $\left(\frac{A_{ij}-\hat{p}_{ij}}{\sqrt{n}}\right)$. We derive the explicit formulas for the asymptotic mean functions and covariance functions with the help of  precise large deviation estimates of the Green function by \cite{Ajanki2017a} and the application of cumulant expansion formula \cite{Khorunzhy1996}\cite{Erds2019MDE}.

%part 2:

%define the problem

%explain what you attempt to do

%summariza the results achieved, progress made and problem unsolved

%outline the plan of attack
\paragraph*{Our contributions.}
We strengthen the existing results in the following ways:
\begin{enumerate}
	\item  Our block-Wigner-type matrices may have not only  inhomogeneous fourth moments but also inhomogeneous second moments. This greatly extends the potential range of application of the theorem. We show that the approximately low-rank structure of the entries would reproduce itself in terms of repetitive patterns in the system of equations for key moments of $Tr(G(z))$
	and other related higher-order structures.
	\item Further, we establish a CLT for the LSS of the data-driven variation of the above matrices, which is no longer Wigner-type and shares long-range weak correlations among the entries. This yields  a direct application in the SBM.
\end{enumerate}

%part 3 Organization of the paper, use different descriptions to summarize each subsection
\subsection*{Organization}

We first introduce a few prerequisites about our main tools and ingredients
in Section \ref{sec:preliminary}. Then we introduce the block-Wigner-type matrices and establish the CLT for LSS of such matrices in Section \ref{sec:blockWignerclt}.  In Section \ref{sec:applicationtosbm}, we consider the application to the SBM and establish a new CLT for LSS of a data-driven variation of the block-Wigner-type matrices. In Section \ref{sec:outlineofpfblockWigner}, the outlines of the proofs of the main results are shown. In Section \ref{sec:simulation}, we apply the above 2 CLTs to the synthetic data of the SBM to show the efficiency of the theorems. Details of proofs are shown in Section \ref{sec:detailedcalculation} and \ref{sec:pfofhatversion}.

\section{Preliminary} \label{sec:preliminary}

\subsection{Notation}
For simplicity of presentation, we will use 
$\underline{M}$ for normalized trace $\frac{1}{N} Tr(M)$ of a  $n\times n$ square matrix $M$,   
$\langle R\rangle$ to denote centered version $R-\mathbb{E} R$ of a random variable $R$,
and $[K]=\{1,2,\cdots,K\}$ to represent the set of positive integers from 1 to $K$. Further, we introduce  two operations related to diagonal terms: for a column vector $\mathbf{v}=[v_1,\cdots,v_n]^{\top}$,  
$Diag[\mathbf{v}]$ denotes the diagonal matrix  whose diagonal elements are the  entries of  $\mathbf{v}$, and for a $n\times n$ matrix $M$, $diag(M)$  denotes the column vector whose entries are the diagonal element of $M$. In particular, $diag\o\circ Diag=Id$.

\subsection{Large deviation bounds from local law for  Wigner-type matrices}

This section gives a quick review of the large deviation bounds from the local law for general Wigner-type matrices by Ajanki et al. \cite{Ajanki}\cite{Ajanki2017a}, which will serve as one of the main ingredients for proving our central limit theorem.

The main object of interest is the resolvent $G(z)=(H-z)^{-1}$, where $H$ is the so-called Wigner-type matrix such that  $H$ is real symmetric and
the entries $H_{i j}$ are independent for $i \leq j$ and centered $\E H_{i j}=0, \ \forall i,j\in[n]$.

Let  $S=(s_{ij})_{i,j=1}^n=(\E |H_{ij}|^2)_{i,j=1}^n$, then  the system of the quadratic vector equations (QVE) is
\begin{equation}\label{eq:qve0}
	-\frac{1}{m_{i}(z)}=z+\sum_{j=1}^{n} s_{i j} m_{j}(z), \quad \text { for all } \quad i\in[n],\quad z \in \mathbb{C}_+.
\end{equation}

There exists a unique solution $\mathbf{m}=(m_1,\ldots,m_n): \mathbb{C}_+\rightarrow \mathbb{C}_+^n$ of the above system on the complex upper half-plane. We refer the readers to \cite{Ajanki} for properties of the QVE system. It is proved by Ajanki et al. \cite{Ajanki2017a} that under certain regularity conditions, the solution $\mathbf{m}=(m_1,\ldots,m_n)$ of the above system of equations may serve as a good approximation for the diagonal terms $(G_{11},\ldots,G_{nn})$ of the resolvent.

\begin{definition}
	[Stochastic domination] Suppose $N_{0}:(0, \infty)^{2} \rightarrow \mathbb{N}$ is a given function, depending only on certain model parameters. For two sequences, $ \varphi=\left(\varphi^{(N)}\right)_{N}$  and  $\psi=\left({\psi}^{(N)}\right)_{N},$ of non-negative random variables we say that $ \varphi$ is stochastically dominated by $\psi$  if for all $ \varepsilon>0 $ and 
	$D>0 , $
	$$
	\mathbb{P}\left(\varphi^{(N)}>N^{\varepsilon} \psi^{(N)}\right) \leq N^{-D}, \quad N \geq N_{0}(\varepsilon, D).
	$$
	In this case we write $\varphi \prec \psi $ or $\varphi =O_{\prec} (\psi)$.  
\end{definition}

\begin{lemma} [Theorem 1.7 of \cite{Ajanki2017a}, reformulated to a macroscopic version]\label{lem:QVE}
	
	Let $H$ be a Wigner-type matrix and $\mathbf{m}$ be defined in \eqref{eq:qve0}.	Suppose that the following assumptions hold: 
	\begin{enumerate}
		\item [A] For all $n$ the matrix $\mathrm{S}$ is flat, i.e.,
		$$
		s_{i j} \leq \frac{C}{n}, \quad C>0,\quad i, j\in [n];
		$$
		\item [B] For all $n$ the matrix $S$ is uniformly primitive, i.e.;
		$$
		\left(S^{L}\right)_{i j} \geq \frac{p}{n},\ p>0, \quad i, j\in [n] ,
		$$
		%			\item [C] For all $n$ the matrix $S$ induces a bounded solution of the QVE, i.e., the unique
		%			solution $m$ of ( 1.7) corresponding to $S$ is bounded,
		%			$$
		%			\left|m_{i}(z)\right| \leq P, \quad i\in[n], \quad z \in \mathbb{C}_+ ;
		%			$$
		\item [C] For all $n$ the entries $H_{i j}$ of the random matrix $H$ have bounded moments
		$$
		\mathbb{E}\left|H_{i j}\right|^{k} \leq \mu_{k} s_{i j}^{k / 2}, \quad k \in \mathbb{N}, i, j\in[n].
		$$
	\end{enumerate}
	are satisfied. Then uniformly for all $z=\tau+i\eta\in\mathbb{C}_+$ with constant order imaginary part $\eta$ or real  part $\tau$ that is bounded away from the edge of the spectrum of $H$, the resolvents of the random matrices $H=H^{(n)}$ satisfy
	\begin{equation}
		\max_{i,j}|G_{ij}(z)-m_i(z)\delta_{ij}|=O_{\prec}( \frac{1}{\sqrt{n}}).
	\end{equation}
	Furthermore, for any sequence of deterministic vectors $\boldsymbol{w}=\boldsymbol{w}^{(n)}=[w_1,\cdots,w_n]\in\mathbb{C}^n$ with $\max_i|w_i|\le 1$, the averaged resolvent diagonal has an improved convergence rate.
	\begin{equation}\label{eq:isotropic}
		|\frac{1}{n}\sum_{i=1}^{n}\bar{w}_i(G_{ii}(z)-m_i(z))|=O_{\prec}( \frac{1}{n}) .	\end{equation}
	
\end{lemma}
A direct application of Lemma~\ref{lem:QVE} together with the trivial bound $|G_{ij}(z)|\le \frac{1}{|\Im z|}$ leads to the following corollary, whose proof is omitted.

\begin{corollary}\label{Cor:QVE}
	$\forall \varepsilon>0, K_0\in\mathbb{N}$, there exists a $N_{\varepsilon,K_0}$ s.t. when $n\ge N_{\varepsilon,K_0}$,
	\begin{equation}
		\begin{aligned}
				&\E|G_{ij}(z)-\delta_{ij}m_i(z)|^k\le \frac{n^{\varepsilon}}{n^{k/2}},\\ 
	  &\E|\frac{1}{n}\sum_{i=1}^{n}\bar{w}_i(G_{ii}(z)-m_i(z))|\le\frac{n^{\varepsilon}}{n^{k}},
		\end{aligned}
	 	\end{equation}
 	for $k\in[K_0]$, for any fixed $z\in \mathbb{C}\backslash\mathbb{R}$, where $\boldsymbol{w}=\boldsymbol{w}^{(n)}=[w_1,\cdots,w_n]\in\mathbb{C}^n$ is deterministic with $\max_i|w_i|\le 1$.\\ 
\end{corollary}

\subsection{Cumulant expansion}

The cumulant expansion formula was first introduced to the random matrices literature by Pastur et al. \cite{Khorunzhy1996}.
\begin{lemma}
	\begin{equation}
		\E[\xi f(\xi)]=\sum_{a=0}^p\frac{\kappa^{a+1}}{a!}\E[f^{(a)}(\xi)]+\varepsilon,
	\end{equation}
	where $|\varepsilon|\le C\sup_t|f^{(p+1)}(t)|E[|\xi|^{p+2}]$ and $C$ depends on $p$ only.
\end{lemma}

The cumulant expansion formula will serve as another important tool in our analysis. In some literature, it is also known as the generalized Stein's method.

\section{Main results}
\subsection{CLT for LSS of block-Wigner-type matrices}
\label{sec:blockWignerclt}

We first define the random matrix model of concern. Note that the initial motivation comes from the stochastic block model. Intuitively, the block-Wigner-type matrix to be defined should be close to a symmetric block-wise i.i.d. matrix. Further, for simplicity and consistency with the SBM, we require that all the diagonal terms $H_{ii}= 0,\forall i\in[n]$. 

First, we introduce the community and the membership operator.

\begin{definition}
	[Community and membership operator]
	Let  $\{C_k\}_{k\in [K]}$ be any partition of $[n]$ with $K$  components, i.e. 
	\begin{equation*}
		\begin{aligned}
			&C_{k_1}\cap C_{k_2}=\emptyset,\text{ when }k_1\ne k_2,\quad k_1,k_2\in[K].\\ 
			&\cup_{k=1}^KC_k=[n].
		\end{aligned}
	\end{equation*}  We call $C_k$ the $k$-th community and define the community membership operator $\sigma$ s.t.
	$$\sigma(i)=k \text{ iff } i\in C_k, \quad i\in[n],\ k\in[K].$$
\end{definition}
For simplicity, we will use $1_{C_k}$  to denote  the (column) indicator vector of $C_k$, $\forall k\in[K]$.

Further, we assume the community number $K$ is fixed and the sizes of the communities are comparable.
\begin{assumption} \label{as:blockratio}
	There exists $\mathbf{\alpha}=[\alpha_1,\cdots, \alpha_K]$, s.t.	$$n_k:=|C_k|=\alpha_k n, \forall k\in[K].$$
	$$\sum_{k=1}^K\alpha_k=1,\quad 0<\alpha_k<1,\quad\forall k\in[K].$$ 
\end{assumption}    
\begin{definition}\label{def:blockWignermatrix} [block-Wigner-type Matrix]
	Let $\kappa_{ij}^{(a)}$ be the $a$-th cumulant of $(\sqrt{n}H_{ij})$.  If there exists a sequence of  partitions $\{C_k\}_{k\in [K]}=\{C_k^{(n)}\}_{k\in [K]}$, s.t. 
	\begin{enumerate}
		\item [a]$H$ is a real symmetric matrix with mean zero and zero diagonal terms,\\ $\E H=0,$\\
		$H_{ii}=0,\forall i\in[n].$
		\item [b]Assumption \ref{as:blockratio} is satisfied.
		\item [c]The first 4 cumulants of $(\sqrt{n}H_{ij})$ will be fully determined by the partition $\{C_k\}_{k\in [K]}$ and $K\times K$ constant matrices $Q^{(2)},Q^{(3)},Q^{(4)}$, namely, let $\sigma$ be the membership operator induced by   $\{C_k\}_{k\in [K]}$, then 
		$$\kappa_{ij}^{(k)}:=\kappa^{(k)}(\sqrt{n}H_{ij})=\begin{cases} Q^{(k)}_{\sigma(i)\sigma(j)}, &\forall i\ne j\\
			0,& i=j \end{cases}, \quad \forall k=2,3,4,$$
		and $Q^{(2)}_{kl}>0,\forall k,l\in[K]$.
		\item [d]There exists a deterministic sequence $\{v_a\}_{a\ge 5}$, s.t. 
		$$\E |\sqrt{n}H_{ij}|^a\le v_a (\kappa_{ij}^{(2)})^{a/2},\quad a\ge 5.$$
	\end{enumerate}

	Then we say that $\{H_n\}$ are block-Wigner-type matrices with model parameters \\ 
	 $(K,n,\alpha,Q^{(2)},Q^{(3)},Q^{(4)},\{v_a\}_{a\ge 5})$. For simplicity,  we will use $H$ for short  in this paper when there is no confusion.

\end{definition}

With our $K$-block model, one can easily check that the quadratic vector equations (\ref{eq:qve0}) will degenerate into the following $K$-equations.
\begin{proposition}[Quadratic vector equation for the block-Wigner-type matrices]
	Given $H(K,n,\alpha,Q^{(2)},Q^{(3)},Q^{(4)},\{v_a\}_{a\ge 5})$, then for any fixed $z$, the diagonal terms of the resolvent $G=(H-z)^{-1}$ have the following approximation
	\begin{equation}
		|G_{ii}(z)-M_l(z)|=O_{\prec}(\frac{1}{\sqrt{n}}), \forall i\in C_l, \forall l\in[K],
	\end{equation} 
	where $\bf{M}$ $=(M_1,\ldots,M_k): \mathbb{C}_+\rightarrow\mathbb{C}_+^K$ is defined to be the unique solution on the complex upper half-plane of the system
	\begin{equation}\label{eq:QVE}
		-\frac{1}{M_{l}(z)}=z+\sum_{m=1}^{K} Q^{(2)}_{lm} \alpha_mM_{m}(z), \quad \text { for all } \quad l=1, \ldots, K,\quad z \in \mathbb{C}_+.
	\end{equation}
\end{proposition}
Thus, the Stieltjes transform of the ESD converges to
$$\sum_{l=1}^K\alpha_lM_l(z),$$
and the  corresponding measure $\mu_{\infty}$ is determined by

$$\int_{\mathbb{R}}\frac{1}{x-z}d\mu_{\infty}=\sum_{l=1}^K\alpha_lM_l(z).$$
\begin{remark}
	
	One may find the assumption $\alpha_k=\frac{n_k}{n}$ pretty strong. In general, due to the nature of the rational number, one may only expect that $\alpha_{k}^{(n)}:=\frac{n_k}{n}\rightarrow \alpha_k$. It then directly  follows from the fact  $|M_m(z)|\le |\frac{1}{\Im(z)}|$  that
	$$\sum_{m=1}^{K} Q^{(2)}_{lm} (\alpha_m^{(n)}-\alpha_m)M^{(n)}_{m}(z)\rightarrow 0.$$ Thus, when we consider only the leading order terms of the equations, we have
	\begin{align*}
		-\frac{1}{M^{(n)}_{l}(z)}=z+\sum_{m=1}^{K} Q^{(2)}_{lm} \alpha^{(n)}_mM^{(n)}_{m}(z)=z+\sum_{m=1}^{K} Q^{(2)}_{lm} \alpha_mM_{m}(z), \forall l\in[K],z \in \mathbb{C}_+.
	\end{align*}
	
	In other words, the leading term of $M_l(z)$ and $M^{(n)}_l(z)$ will follow the same QVE on the complex upper half-plane by the uniqueness of the solution of the QVE. Then w.l.o.g. we may simply treat the case as $\alpha_k=\frac{n_k}{n}$.
	
	One may argue that the above argument only implies that the limiting spectral distribution will be the same. We claim that it will not affect our CLT as well. Precisely, one may check the system of equations in the preceding sections and note that all the coefficients will count only up to order 1, and all the limiting functions will be fixed once the $|\alpha_{k}^{(n)}-\alpha_k|=o(1), \forall k$.
	
	The minor order terms in $|\alpha_{k}^{(n)}-\alpha_k|$ do matter, not in our CLT, but in the normalization term $-n\int fd\mu_{\infty}$.
	
\end{remark}

\begin{theorem}\label{thm:mainBlockWigner}
	Let the matrix $H:=H_n$ be a sequence of block-Wigner-type matrices with model parameter $(K,n,\alpha,Q^{(2)},Q^{(3)},Q^{(4)},\{v_{a}\}_{a\ge5})$. Let $Co_1(z)$ and $Co_2(z)$ be $K\times K$ matrices defined by\begin{equation}
		(Co_1(z))_{kl}:= \frac{Q^{(2)}_{kl}\alpha_kM_k(z)}{z}-\delta_{kl}\frac{1}{zM_k(z)},
	\end{equation}	
	and
	\begin{equation}
		(Co_2(z_1,z_2))_{kl}:= \frac{Q^{(2)}_{kl}\alpha_kM_k(z_2)}{z_1}-\delta_{kl}\frac{1}{z_1M_k(z_1)}
	\end{equation}
	respectively.	
	Then the  spectral empirical process $G_n=(G_n(f)):=\sum_{i=1}^nf(\lambda_i)-n\int fd\mu_{\infty}$ indexed by the set of analytic functions $\mathcal{A}$ converges weakly in finite dimension to a Gaussian process $G:=\{G(f):f\in\mathcal{A} \}$ with mean function $M(f)$ and the covariance function $V(f,g)$ to be defined below. The mean function is 
	$$M(f)=-\frac{1}{2\pi i}\int_{\Gamma}Mean(z)f(z)dz,$$
	where $\Gamma$ is a contour that encloses the support of spectrum of $H_n$
	and 
	\begin{align*}Mean(z)=&\sum_{k=1}^KY_{k}(z),
	\end{align*}
	where $\mathbf{Y}(z)=[Y_{1}(z),\cdots,Y_{K}(z)]^{\top}$ is the solution of 
	\begin{equation*}
		\begin{aligned}
			&	Co_1(z) \mathbf{Y}(z)\\ =&-\frac{1}{z}diag(Q^{(2)}X(z))+\frac{2}{z}[\alpha_1Q_{11}^{(2)}M_1^2,\cdots,\alpha_KQ_{KK}^{(2)}M_K^2]^{\top}\\ &-\frac{1}{z}diag\left[Q^{(4)}\left(\alpha_l\alpha_mM_l^2M_m^2\right)_{l,m=1}^K\right],
		\end{aligned}
	\end{equation*}
	and $X(z)=(X_{lm}(z))_{l,m=1}^K$ is defined by
	\begin{equation*}
		Co_1(z) X(z)=-\frac{1}{z}Diag([{\alpha_1M_1(z)},\ldots,{\alpha_KM_K(z)}]^{\top}).
	\end{equation*}

	The covariance function is
	\begin{equation*}
		V(f,g)=\frac{-1}{4\pi^2}\int_{\Gamma}\int_{\Gamma}f(z_1)g(z_2)Cov{(z_1,z_2)}dz_1dz_2,
	\end{equation*}
	where $Cov{(z_1,z_2)}=\sum_{l,m=1}^KZ_{lm}$, and $Z:=(Z_{lm}(z_1,z_2))_{l,m=1}^K$ satisfies the equation
	\begin{equation}\label{eq:covariancefunc}
		\begin{aligned}
			&z_1(Co_1Z)_{lm}\\ 
			=&-\sum_{k=1}^K2Q^{(2)}_{lk}W_{l,m,k}(z_1,z_2)+2Q^{(2)}_{ll}M_l(z_1)X_{lm}(z_2)\\ 
			&-\sum_{k=1}^KQ^{(4)}_{lk}\alpha_kM_l(z_1)M_k(z_1)M_k(z_2)X_{lm}(z_2)-\sum_{k=1}^KQ^{(4)}_{lk}\alpha_lM_l(z_1)M_k(z_1)M_l(z_2)X_{km}(z_2),
		\end{aligned}
	\end{equation}
	where 
	$W=(W_{l,m,r})$ is a $K\times K \times K$ tensor and the vector $\mathbf{\tilde{W}}^{(l,m)}= [W_{l,m,1},\cdots,W_{l,m,K}]^{\top}$ satisfies the equation
	\begin{equation}
		\begin{aligned}
			&(z_1Co_2(z_1,z_2)\mathbf{\tilde{W}}^{(l,m)}(z_1,z_2))_r=-\sum_{k=1}^KQ^{(2)}_{rk}\tilde{X}_{lk}(z_1,z_2)X_{mr}(z_2)-\delta_{rl}X_{lm}(z_2),
		\end{aligned}
	\end{equation}
	where $\tilde{X}=(\tilde{X}_{lm})_{l,m=1}^K$ satisfies
	\begin{equation*}
		\begin{aligned}
			Co_2(z_1,z_2)\tilde{X}(z_1,z_2)=-\frac{1}{z_1}Diag([{\alpha_1M_1(z_2)},\ldots,{\alpha_KM_K(z_2)}]^{\top}).
		\end{aligned}
	\end{equation*}

\end{theorem}

\subsection{Application to the stochastic block model: a step forward with the  data-driven renormalized adjacency matrices of SBM}
\label{sec:applicationtosbm}

As mentioned in the introduction, the stochastic block model serves as one of the primary motivations for the block-Wigner-type matrices. Recall  that a stochastic block model is a random graph with $n$ nodes which are divided into $K$ disjoint communities $\{C_k\}_{k=1}^K$, the size of $k$-th community $n_k=|C_k|$ satisfies assumption \ref{as:blockratio}. The upper-triangular entries of the symmetric adjacency matrix are independent Bernoulli random variables whose parameters are determined by the community membership of the nodes. In other words, we have a $K\times K$ deterministic symmetric matrix $(\tilde{P}_{ij})_{K\times K}$, such that the symmetric adjacency matrix of the network follows the rule:
\begin{equation}\label{eq:sbmsetting}
	\begin{aligned}
		A_{ij}&\in\{0,1\}, \forall i,j\in[n],\\ 
		A_{ii}&=0,\forall i\in [n],\\  
		p_{ij}:&=P(A_{ij}=1)=\tilde{P}_{\sigma(i)\sigma(j)},\\
		A_{ij}\indep A_{kl}, &\text{ for }(i,j)\ne (k,l),\ i<j,\ k<l.
	\end{aligned}
\end{equation}
where $\sigma(i)\in\{1,2,\ldots,K\}$ is the membership operator defined by this model and indicates which community node $i$ belongs to.  We can see that it fits the description of our block-Wigner-type matrices after the renormalization:
\begin{equation}\label{defi:nonhatversion}
	H_{ij}^{(A)}=\begin{cases}
		\frac{A_{ij}-p_{ij}}{\sqrt{n}}, & i\ne j\\ 
		0, & i=j.
	\end{cases}
\end{equation}

Further, in statistical application such as the hypothesis testing on a stochastic block model, the connection probabilities $p_{ij}$'s are not known a priori. Instead, they need to be directly estimated from the observed graph $(A_{ij})_{n\times n}$ as defined in \eqref{eq:sbmsetting}. Assume the membership operator $\sigma$ is known, we can define the empirical estimator $$\hat{p}_{ij}=\sum\limits_{\alpha\in C_{\sigma(i)},\beta\in C_{\sigma(j)}}\frac{A_{\alpha\beta}}{N_{\sigma(i)\sigma(j)}}$$
for $p_{ij}$, where $N_{kl}$ is the total number of non-diagonal entries whose first index falls in the $k$-th community, and the second index lies in the $l$-th community, $k,l\in[K]$. Namely
\begin{equation}\label{eq:pestimator}
	N_{kl}=\begin{cases}
		n_kn_l, &\text{if $k\ne l$},\\ 
		n_k(n_k-1), &\text{if $k=l$}.
	\end{cases}
\end{equation}
We then consider the data-driven renormalized adjacency matrix
\begin{equation}\label{defi:hatversion}
	\begin{aligned}
		\hat{H}^{(A)}_{ij}=\begin{cases}
			\frac{A_{ij}-\hat{p}_{ij}}{\sqrt{n}}, & i\ne j\\ 
			0, & i=j.
		\end{cases}\\ 	
	\end{aligned}
\end{equation}

It turns out that the LSS of $H$ and $\hat{H}$ will share similar asymptotic behavior. We have the following theorem. 

\begin{theorem}\label{thm:mainBlockWigneremp}
	Let the matrix $H^{(A)}$ be defined by \eqref{defi:nonhatversion}, which is  a  block-Wigner-type  matrix with model parameter $(K,n,\alpha,Q^{(2)},Q^{(3)},Q^{(4)},\{v_a\}_{a\ge 5})$, and $\hat{H}^{(A)}$ be defined via \eqref{defi:hatversion}, then the spectral empirical  process $\hat{G}_n(f):=\sum_{i=1}^nf(\lambda_i(\hat{H}))-n\int fd\mu_{\infty}$ will share the same limiting distributions with $G_n(f):=\sum_{i=1}^nf(\lambda_i(H))-n\int fd\mu_{\infty}$ established in Theorem \ref{thm:mainBlockWigner}. 	
\end{theorem}

\section{Outline of the proof}\label{sec:outlineofpfblockWigner}

\subsection{Outline of the proof of Theorem \ref{thm:mainBlockWigner}}
Recall the classic Cauchy integral trick
$f(x)=\frac{1}{2 \pi i} \oint_{\mathcal{C}} \frac{f(z)}{z-x} d z,$ which allows us to rewrite the sum  

\begin{equation}
	\sum_{j=1}^nf(\lambda_j)=-\frac{1}{2\pi i}\oint_{\mathcal{C}}f(z)Tr(G(z))dz,
\end{equation}
where $\mathcal{C}$ is a  contour that encloses the support of $H$ with high probability. Naturally one may expect that the behavior of the linear spectral statistics $\sum_{j=1}^nf(\lambda_j)$ will be governed by that of the quantity 
$TrG(z).$

Inspired by the previous works such as   \cite{Khorunzhy1996}\cite{landon2019applications}\cite{Bai2004}\cite{bai2005}, our proof first  combines the characteristic function method with the cumulant expansion to prove the finite-dimensional convergence of the process $\langle Tr(G(z))\rangle$, then with the tightness of the process  we proceed to the linear spectral statistics.  To be more specific, our tasks are divided into 4 steps mainly:
\begin{itemize}
	\item Expectation;
	\item Covariance;
	\item Normality;
	\item Tightness.
\end{itemize}

We use the resolvent identities $G = \frac { 1 } { z } ( H G - I )$ so that the cumulant expansion formula could be applied. Then we use the block structure to simplify the calculations.  Let $T_k:=Id|_{C_k}$ be the restriction of the identity matrix on the $k$-th community $C_k$, we have the following decomposition for $\E Tr(G(z))$:	
\begin{equation}\label{eq:meandecomposition}
	\begin{aligned}
		&z\E Tr(G(z))=\E Tr(HG-I)\\ 
		=&\E\Big\{ -n-\frac{1}{n}\sum_{i,j=1,i\ne j}^n\kappa_{ij}^{(2)}(G_{ij}^2+G_{ii}G_{jj}) \\ 
		&+\frac{1}{n^{3/2}}\sum_{i,j=1,i\ne j}^n\frac{\kappa_{ij}^{(3)}}{2!}(2G_{ij}^3+6G_{ij}G_{ii}G_{jj}) \\ 
		&-\frac{1}{n^2}\sum_{i,j=1,i\ne j}^n\frac{\kappa_{ij}^{(4)}}{3!}(6G_{ij}^4+36G_{ij}^2G_{ii}G_{jj}+6G_{ii}^2G_{jj}^2)\\ 
		=&-n-I_{1,1}-I_{1,2}-I_{1,3}+\varepsilon_{I_1},
	\end{aligned}
\end{equation}
where  

\begin{equation}
	\begin{aligned}
%	&=\frac{1}{n}\E\sum_{i,j=1,i\ne j}^n\kappa_{ij}^{(2)}	G_{ij}^2\\ 
		I_{1,1}	&=\frac{1}{n}\E\sum_{l,m=1}^K\sum_{i\in C_l,j\in C_m,i\ne j}^n\kappa_{ij}^{(2)}	G_{ij}^2\\ 
		&=\frac{1}{n}\E\sum_{l,m=1}^KQ^{(2)}_{lm}	\sum_{i\in C_l,j\in C_m}G_{ij}^2-\frac{1}{n}\E\sum_{l=1}^KQ^{(2)}_{ll}	\sum_{i\in C_l}G_{ii}^2\\ 
		&=\frac{1}{n}\E\sum_{l,m=1}^KQ^{(2)}_{lm}Tr(T_lGT_mG)-\frac{1}{n}\E\sum_{l=1}^K Q^{(2)}_{ll}	\sum_{i\in C_l}G_{ii}^2,
	\end{aligned}
\end{equation}

\begin{equation}\label{eq:meandecomposition2}
	\begin{aligned}
%	&=\frac{1}{n}\E\sum_{i,j=1,i\ne j}^n\kappa_{ij}^{(2)}	G_{ii}G_{jj}\\ 
		I_{1,2}	&=\frac{1}{n}\E\sum_{l,m=1}^K\sum_{i\in C_l,j\in C_m,i\ne j}^n\kappa_{ij}^{(2)}	G_{ii}G_{jj}\\ 
		&=\frac{1}{n}\E\sum_{l,m=1}^KQ^{(2)}_{lm}	\sum_{i\in C_l,j\in C_m}	G_{ii}G_{jj}-\frac{1}{n}\E\sum_{l=1}^nQ^{(2)}_{ll}	\sum_{i\in C_l}G_{ii}^2\\ 
		&=\frac{1}{n}\E\sum_{l,m=1}^KQ^{(2)}_{lm}Tr(T_lG)Tr(T_mG)-\frac{1}{n}\E\sum_{l=1}^nQ^{(2)}_{ll}	\sum_{i\in C_l}G_{ii}^2,
	\end{aligned}
\end{equation}
and 
\begin{equation}
	\begin{aligned}
		I_{1,3}=&\frac{1}{n^{2}}\E\sum_{i,j=1,i\ne j}^n\kappa_{ij}^{(4)}G_{ii}^2G_{jj}^2.
	\end{aligned}
\end{equation}

The remainder $\varepsilon_{I_1}$ will have a vanishing order \(O(\frac{1}{n^{1/2}})\).

Though $I_{1,3}$ can be directly estimated from the first-order approximation from the local law, approximations for $I_{1,1}$ and $I_{1,2}$ are not so straightforward. We need to derive new systems of equations for the quantities. To be more precise, we introduce the following lemmas.
\begin{lemma}\label{eq:trtgtg}
	The vector  $$\mathbf{X}_{\mathbf{GTGT}}^{(l)}{(z)}=[\frac{1}{n}\E Tr(G(z)T_lG(z)T_1),\cdots,\frac{1}{n}\E Tr(G(z)T_lG(z)T_K)]^{\top}$$ satisfies the following system of equations
	\begin{equation}\label{eq:GTGTl}
		Co_1(z)\mathbf{X}_{\mathbf{GTGT}}^{(l)}{(z)}=\mathbf{B}^{(l)}{(z)}
	\end{equation}up to order 1, 
	where 
	\begin{equation*}
		\begin{matrix}
			&\mathbf{B}^{(l)}{(z)}=  [0,\ldots,0,&-\frac{\alpha_lM_l(z)}{z}&,0,\ldots,0]^T.\\ 
			&    &\uparrow&\\
			&    &\text{l-th}&
		\end{matrix}	
	\end{equation*}   
	Further, the matrix $M_{GTGT}(z)=\big( \frac{1}{n}Tr(G(z)T_lG(z)T_m)\big)_{l,m=1}^K$ satisfies
	\begin{equation}\label{eq:MGTGT}
		M_{GTGT}(z)=-\big(Q^{(2)}-Diag ([\frac{1}{\alpha_1M_1^2(z)},\cdots,\frac{1}{\alpha_KM_K^2(z)}]^{\top})\big)^{-1}.
	\end{equation}

\end{lemma}

\begin{lemma}\label{eq:trtgsubleading}
	The vector $$\mathbf{Y}(z)=[\E Tr(T_1G(z))-\alpha_1nM_1(z),\cdots,\E Tr(T_KG(z))-\alpha_KnM_K(z)]^{\top}$$ satisfies the following equation
	\begin{equation}\label{eq:TrG-n}
		\begin{aligned}
			&	Co_1(z) \mathbf{Y}(z)\\ =&-\frac{1}{z}diag(Q^{(2)}X(z))+\frac{2}{z}[\alpha_1Q_{11}^{(2)}M_1^2,\cdots,\alpha_KQ_{KK}^{(2)}M_K^2]^{\top}\\
			&-\frac{1}{z}diag\left[Q^{(4)}\left(\alpha_l\alpha_mM_l^2M_m^2\right)_{l,m=1}^K\right],
		\end{aligned}
	\end{equation}
	
\end{lemma}
We refer the proofs to Sections \ref{sec:mean1} and \ref{sec:mean2}.

Similarly, we may use the same techniques to calculate  the covariance function  $Cov{(z_1,z_2)}:= Cov(Tr G(z_1),Tr G(z_2))$.

First we decompose the covariance function of $Tr(G)$ into the following block-wise forms
\begin{equation*}
	\begin{aligned}
		Cov{(z_1,z_2)}=& Cov(Tr G(z_1),Tr G(z_2))= Cov(\sum_{l=1}^K 	Tr(T_{l}G(z_1)),\sum_{m=1}^KTr(T_{m}G(z_2)))\\
		=&\sum_{l,m=1}^KCov_{lm}(z_1,z_2),
	\end{aligned}
\end{equation*}

where  $$Cov_{lm}(z_1,z_2):=Cov(T_{l}G(z_1),T_{m}G(z_2)).$$
%	
%	Here and later in Section \ref{sec:detailedcalculation}, for simplicity of writing, we will sometimes use $\underline{Mat}$ to denote the normalized trace $\frac{1}{n}Tr(Mat)$ of a $n$ by $n$ matrix $Mat$.

Our primary problem is to calculate $Cov_{lm}(z_1,z_2)$ to order 1,  $\forall l,m\in[K]$. Note that 	$Cov_{lm}(z_1,z_2)=n^2[\E\underline{T_lG(z_1)}\ \underline{T_mG(z_2)}-\E\underline{T_lG(z_1)}\E\underline{T_mG(z_2)}],$
then we need to calculate the following expansion to the order $\frac{1}{n^2}$,
\begin{equation}
	\begin{aligned}
		&\frac{1}{n^2}z_1Cov_{lm}(z_1,z_2)\\
		=&z_1[\E\underline{T_lG(z_1)}\ \underline{T_mG(z_2)}-\E\underline{T_lG(z_1)}\E\underline{T_mG(z_2)}] =z_1\E\underline{G(z_1)T_l}\langle\underline{T_mG(z_2)} \rangle\\ 
		=&\E\underline{HG(z_1)T_l}\langle\underline{T_mG(z_2)} \rangle
		=\frac{1}{n}\E\sum_{i\in C_l,j}H_{ij}G_{ij}(z_1)\langle\underline{T_mG(z_2)} \rangle\\ 
		=&\frac{1}{n}\E\sum_{i\in C_l,j}\sum_{a+b=0}^5\frac{\kappa_{ij}^{(a+b+1)}}{n^{(a+b+1)/2}a!b!}\frac{\partial^a{G_{ij}(z_1)}\partial^b\langle\underline{T_mG(z_2)}\rangle}{\partial H_{ij}^{a+b}}+\varepsilon_{I_2} \\ 
		=&\sum_{a+b=0}^5I_{2,(a,b)}+\varepsilon_{I_2}.
	\end{aligned}
\end{equation}

It turns out, only $I_{2,(0,1)}$,$I_{2,(1,0)}$,$I_{2,(1,2)}$ have $O(\frac{1}{n^2})$ contributions, which will lead to a set of systems of equations for $\{Cov_{lm}(z_1,z_2)\}_{l,m=1}^K$ as well. The key observation is that similar to the quantities calculated in the mean function, we will explore similar $K$-dimensional systems of equations via cumulant expansion due to the block structure in calculating the covariance function. We refer the details to Section \ref{sec:covlm}.

Section \ref{sec:gaussianitynonhat} will show the proof of the normality of the linear spectral statistics. The proof for normality is relatively routine. We will adopt the following technique originated from Tikhomirov \cite{Tikhomirov1981}. The core idea can be simplified as follows. To prove that a sequence of real random variable $R_n$ converges to a Gaussian random variable with mean zero and variance $\sigma^2$, it  suffices to prove that
$$\E e^{it R_n}\rightarrow  e^{-\frac{1}{2} \sigma^{2} t^{2}}. $$

We prove alternatively that  its derivative will behave similarly to that of  the derivative of a characteristic function of a Gaussian distribution
$$i\E R_ne^{it R_n}\rightarrow-\sigma^2 t\E e^{itR_n},$$
in which $R_n$ is a real function constructed from $Tr(G(z))$, from which by $HG=I+zG$ we can find the form $HG$, thus extract the form  $\E hf(h)$. Then the cumulant expansion formula can be applied. In Section \ref{sec:gaussianitynonhat}, we will apply the multivariate version of the above trick to establish the normality.

In Section \ref{sec:tightness}, we establish the tightness of the process $\langle Tr(G(z))\rangle$ via a similar approach, then it follows from \cite{Bai2004} that we can proceed from finite dimensional convergence of  $\{\langle Tr(G(z_s))\rangle\}_{s=1}^t$ to the weak convergence of the linear spectral statistics.

\subsection{Outline of the proof of Theorem \ref{thm:mainBlockWigneremp}}
Recall that in the SBM setting, given the adjacency matrix $(A_{ij})_{i,j\in [n]}$ of a SBM, we may consider two renormalized versions $H^{(A)}$ \eqref{defi:nonhatversion} and $\hat{H}^{(A)}$ \eqref{defi:hatversion}. For simplicity, we will use $H$ and $\hat{H}$ for short when there is no confusion.

Note that when $i\ne j$,
\begin{equation}
	\begin{aligned}
		(\hat{H}-H)_{ij}=&\frac{p_{ij}-\hat{p}_{ij}}{\sqrt{n}}=-\frac{1}{\sqrt{n}}\sum\limits_{\alpha\in C_{\sigma(i)},\beta\in C_{\sigma(j)}}\frac{A_{\alpha\beta}-p_{ij}}{N_{\sigma(i)\sigma(j)}}\\ 
%		=&-\frac{1}{\sqrt{n}}\sum\limits_{\alpha\in C_{\sigma(i)},\beta\in C_{\sigma(j)}}\frac{A_{\alpha\beta}-p_{\alpha\beta}}{N_{\sigma(i)\sigma(j)}}\\ 
		=&-\frac{1}{N_{\sigma(i)\sigma(j)}}\sum\limits_{\alpha\in C_{\sigma(i)},\beta\in C_{\sigma(j)}}\frac{A_{\alpha\beta}-p_{\alpha\beta}}{\sqrt{n}}=-\frac{1}{N_{\sigma(i)\sigma(j)}}\sum\limits_{\alpha\in C_{\sigma(i)},\beta\in C_{\sigma(j)}}H_{\alpha\beta}.
	\end{aligned}
\end{equation}
Then by concentration inequality we know instantly that $||\hat{H}-H||=o_p(\frac{\log(n)}{\sqrt{n}})$, which implies that the limiting spectral distribution of $\hat{H}$ will be the same as that of $H$. However, this stand-alone bound is not sufficient for identical CLTs. To study the LSS of $\hat{H}$, 	we need to follow a similar process  to the one we use to prove Theorem \ref{thm:mainBlockWigner}.

Further, we investigate on the resolvents $G(z)=(H-z)^{-1}$ and $\hat{G}(z)=(\hat{H}-z)^{-1}.$ Note also that by the resolvent identity, we have
$$\begin{aligned}			
	\hat{G}(z)=&\sum_{k=0}^{m}G(z)[-(\hat{H}-H)G(z)]^k+\hat{G}(z)[-(\hat{H}-H)G(z)]^{m+1}\\ 
	=&G(z)-G(z)(\hat{H}-H)G(z)+G(z)(\hat{H}-H)G(z)(\hat{H}-H)G(z)\\ 
	&-G(z)(\hat{H}-H)G(z)(\hat{H}-H)G(z)(\hat{H}-H)\hat{G}(z).	
\end{aligned}$$

Further, note that by $||\hat{H}-H||=o_p(\frac{\log(n)}{\sqrt{n}})$ and $||G(z)||\le \frac{1}{|\Im(z)|}$, we expect that the higher-order expansion terms would vanish. 

The proof of the Theorem \ref{thm:mainBlockWigneremp} will adopt the same approach as Theorem \ref{thm:mainBlockWigner} per se. However, we will mainly focus on the difference of the resolvents. The details of the proof can be found in Section \ref{sec:pfofhatversion}.

\section{Numerical results}\label{sec:simulation}
\subsection{Experiments on verifying Theorem \ref{thm:mainBlockWigner}}
We test our theorems under the setting of SBM \eqref{eq:sbmsetting} since the renormalized adjacency matrix \eqref{defi:nonhatversion} is naturally a block-Wigner-type matrix. Numerical experiments are conducted for the cases where $\tilde{P}$ in \eqref{eq:sbmsetting} is a matrix with identical diagonal terms $p$ and identical off-diagonal terms $q$. Under this framework, we may let both $p$ and $q$ run through the grid $\{0.1, 0.2,\cdots,0.9\}$ to obtain a total of  $9\times 9=81$ stochastic block models. Given a test function, we can calculate the theoretical values of asymptotic means and asymptotic variances  via Theorem \ref{thm:mainBlockWigner}. In the meantime,  we are also able to generate real empirical data  via Monte Carlo method with $N_r$ repetitions and get empirical means and variances for each model. Then we may compare the theoretical values and the empirical values via the 2D-mesh plots.

Note that for simplicity of presentation, we will compare $L_n(f)=\sum_{i=1}^nf(\lambda_i)$ instead of the truncated version $G_n(f)=\sum_{i=1}^nf(\lambda_i)-n\int fd\mu_{\infty}$.

\begin{example}\label{simu:ch2a}
	The following parameters are used:
	
	$K=3.$  $\alpha=[0.25,0.25,0.5]$. $N=800$. $N_r=800$. $\tilde{P}=(p_i-q_j)I+q_j11^T,$ where $p_i=\frac{i}{10}$, $q_j=\frac{j}{10}$, $i,j\in[9]$.
	$f=x^2.$
\end{example}	

%		
%		\begin{table}[h]
%			\caption{Comparison of the asymptotic mean, variance and their empirical values obtained by Monte Carlo for the LSS $L_n(x^2)$. Emipirical values use $800$ repetitions.}
%			\label{table:simuch2a}
%			\centering
%			\begin{tabular}{m{1.2cm} | m{3.5cm} | m{3.5cm}| m{3.5cm}}
%				& Asymptotic & Empirical & max abs diff\\
%				\hline
%				\tiny	Mean of $L_n(x^2)$ & \begin{subfigure}[c]{1\linewidth}
%					\label{Theoretical mean for x^2 l}
%					\includegraphics[width=\linewidth]{figures/TheoMean11.jpg}
%				\end{subfigure}& \begin{subfigure}[c]{1\linewidth}
%					\includegraphics[width=\linewidth]{figures/EmpMean11.jpg}
%					\label{Numerical mean for x^2 l}
%				\end{subfigure}	 &  0.0195\\
%				\hline
%				\tiny	Var of $L_n(x^2)$ & 	\begin{subfigure}[c]{1\linewidth}
%					
%					\includegraphics[width=\linewidth]{figures/TheoVar11.jpg}
%					\label{Theoretical var for x^2 l}
%				\end{subfigure} & 	\begin{subfigure}[c]{1\linewidth}
%					\includegraphics[width=\linewidth]{figures/EmpVar11.jpg}
%					\label{Numerical var for x^2 l}
%				\end{subfigure} &0.0112\\ 
%	  
%				\hline
%			\end{tabular}	
%			
%		\end{table}

%		\begin{example}\label{simu:ch2a}
%		The following parameters are used:
%		
%			$K=6$, $\alpha=[0.1,0.15,0.2,0.25,0.1,0.2]$, $n=1000$, $N_r=1600$, $\tilde{P}=(p_i-q_j)I+q_j11^T,$ where $p_i=\frac{2i-1}{10}$, $q_j=\frac{2j-1}{10}$.
%		$f=x^2.$
%	\end{example}	

\begin{table}[h]
	\centering
	\begin{tabular}{m{1.1cm} | m{4cm} | m{4cm}| m{1.6cm}}
		& Asymptotic & Empirical & maximal absolute difference\\
		\hline
		\small	Mean of $L_n(x^2)$ & \begin{subfigure}[c]{1\linewidth}
			\label{Theoretical mean for x^2 l}
			\includegraphics[width=\linewidth]{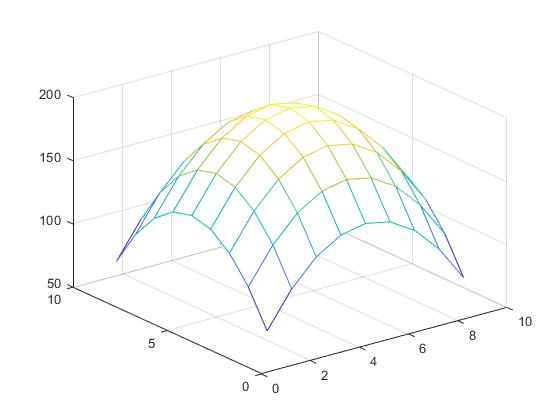}
		\end{subfigure}& \begin{subfigure}[c]{1\linewidth}
			\includegraphics[width=\linewidth]{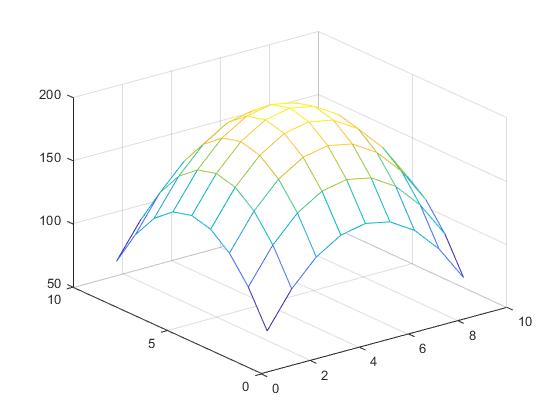}
			\label{Numerical mean for x^2 l}
		\end{subfigure}	 &  0.0195\\
		\hline
		\small	Variance of $L_n(x^2)$ & 	\begin{subfigure}[c]{1\linewidth}
			
			\includegraphics[width=\linewidth]{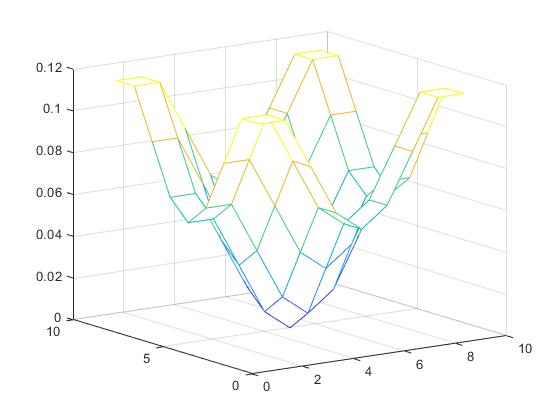}
			\label{Theoretical var for x^2 l}
		\end{subfigure} & 	\begin{subfigure}[c]{1\linewidth}
			\includegraphics[width=\linewidth]{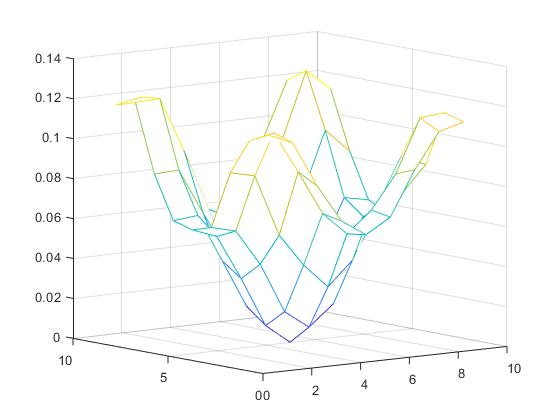}
			\label{Numerical var for x^2 l}
		\end{subfigure} &0.0112\\ 
		
		\hline
	\end{tabular}	
	\caption{Comparison of the asymptotic mean, variance and their empirical values obtained by Monte Carlo for the LSS $L_n(x^2)$. Empirical values use $800$ repetitions.}
	\label{table:simuch2a}
\end{table}

%		\begin{figure}
%			\begin{center}
%				\includegraphics[scale=0.55]{figures/EmpMean11.jpg}
%				\includegraphics[scale=0.55]{figures/TheoMean11.jpg}
%				\caption{Empirical Mean (L) vs Theoretical Mean (R) (max absolute error= 0.0195)}
%			\end{center}
%		\end{figure}
%		
%		
%		\begin{figure}
%			\begin{center}
%				\includegraphics[scale=0.55]{figures/EmpVar11.jpg}
%				\includegraphics[scale=0.55]{figures/TheoVar11.jpg}
%				\caption{Empirical Var(L) vs Theoretical Var (R) (max absolute error=0.0112)}
%			\end{center}
%		\end{figure}

One can see from Table \ref{table:simuch2a} that	we obtain a quite good match between theoretical and empirical means and variances.

Next, we consider 9 SBMs out of the 81 in Example \ref{simu:ch2a} and display in Figure \ref{qqplot:simuch2a} the normal qqplots of the empirical LSS $L_n(f)$ after normalization $\frac{L_n(f)-Mean(L_n(f))}{Std(L_n(f))}$. These qqplots empirically confirm the asymptotic normality of the LSS.

\begin{figure}
	\begin{center}
		\includegraphics[scale=0.23]{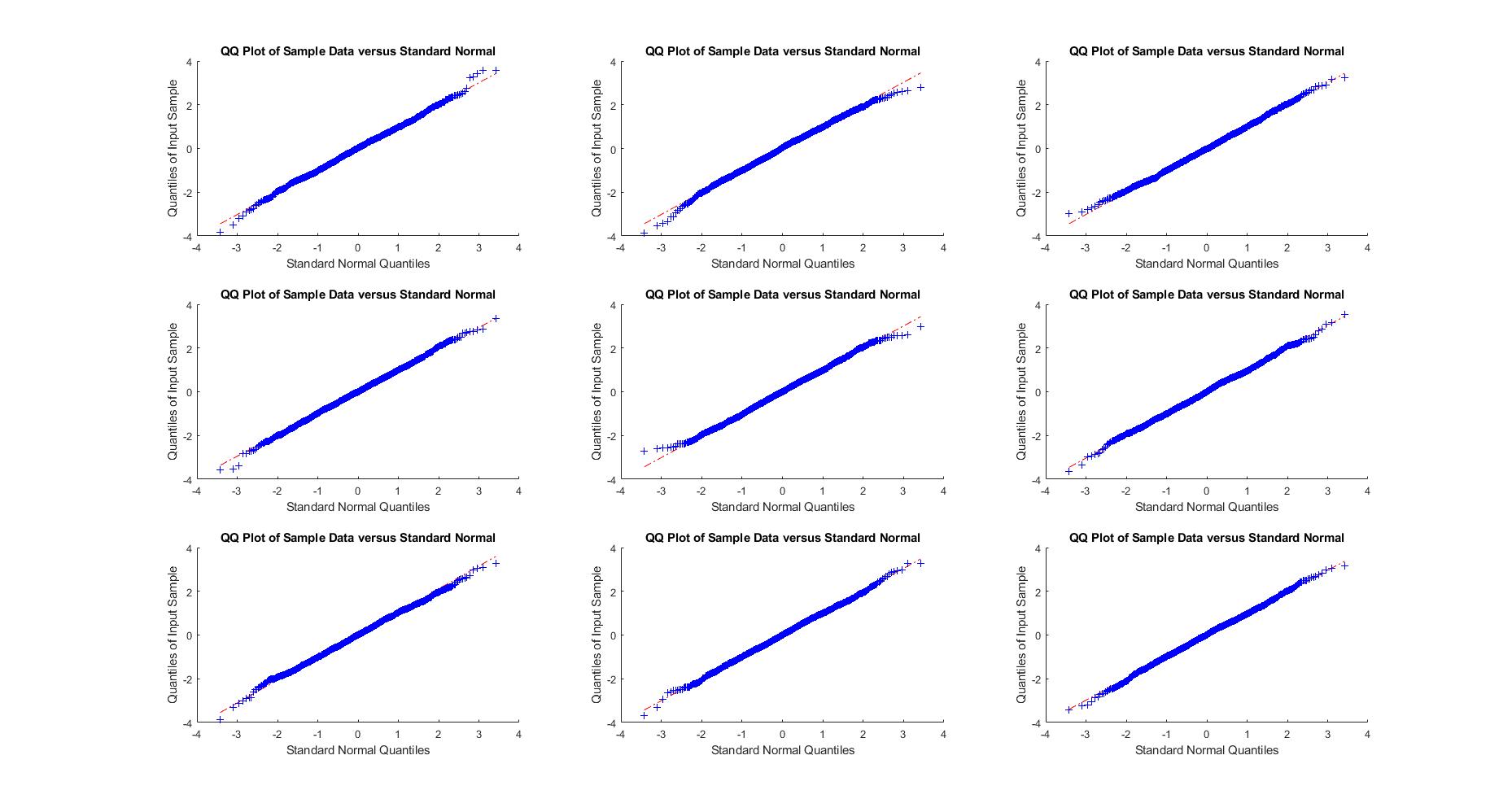}
		\caption{qqplots for normalized LSS of 9 different SBMs, test function $x^2$   }
		\label{qqplot:simuch2a}
	\end{center}
\end{figure}

\begin{example}
	Same setting except the test function $f$ is $x^4$.	Simulation results are shown in Table \ref{simu:ch2b} and Figure \ref{qqplot:simuch2b}. The conclusion is similar to that of Example \ref{simu:ch2a}.
\end{example}	
\begin{table}
	\centering
	\begin{tabular}{m{1.1cm} | m{4cm} | m{4cm}| m{1.6cm}}
		& Asymptotic & Empirical & maximal absolute difference\\
		\hline
		\small Mean of	$L_n(x^4)$ & \begin{subfigure}[c]{1\linewidth}
			\label{Theoretical mean for x^4 lb}
			\includegraphics[width=\linewidth]{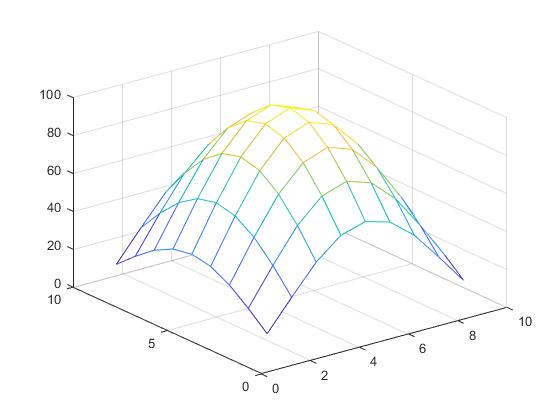}
		\end{subfigure}& \begin{subfigure}[c]{1\linewidth}
			\includegraphics[width=\linewidth]{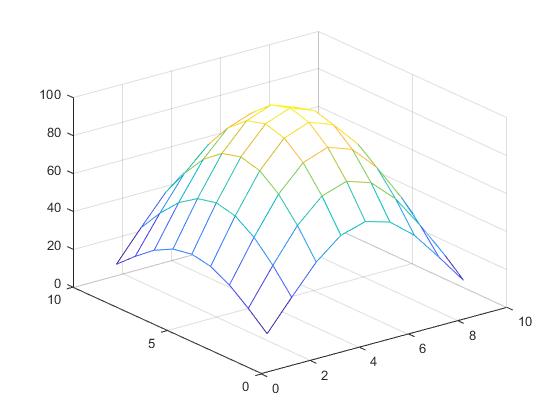}
			\label{Numerical mean for x^4 lb}
		\end{subfigure}	 &  0.3100\\
		\hline
		\small Variance of $L_n(x^4)$  & 	\begin{subfigure}[c]{1\linewidth}
			
			\includegraphics[width=\linewidth]{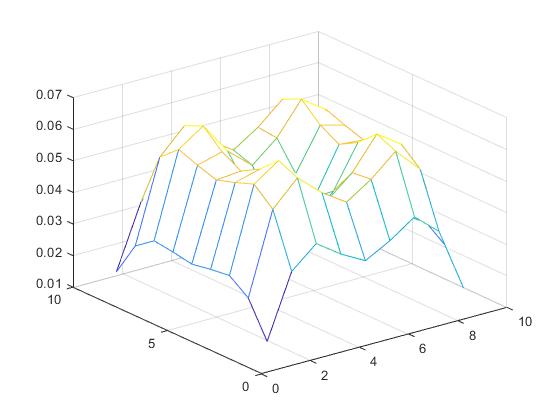}
			\label{Theoretical var for x^4 lb}
		\end{subfigure} & 	\begin{subfigure}[c]{1\linewidth}
			\includegraphics[width=\linewidth]{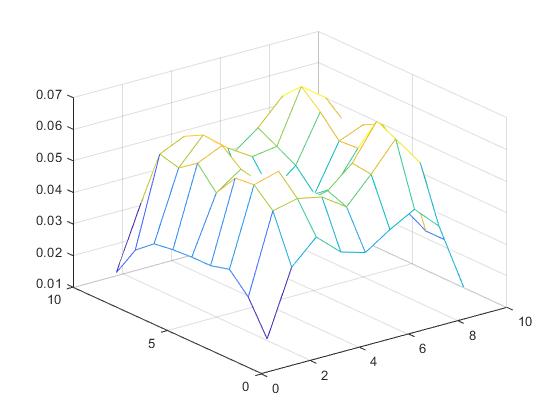}
			\label{Numerical var for x^4 lb}
		\end{subfigure} &0.0065\\ 
		
		\hline
	\end{tabular}	
	\caption{Comparison of the asymptotic mean, variance and their empirical values obtained by Monte Carlo for the LSS  $L_n(x^4)$. Empirical values use $800$ repetitions.}
	\label{simu:ch2b}
\end{table}
%		
%		\begin{figure}
%			\begin{center}
%				\includegraphics[scale=0.55]{figures/EmpMean112.jpg}
%				\includegraphics[scale=0.55]{figures/TheoMean112.jpg}
%				\caption{Empirical Mean (L) vs Theoretical Mean (R) (max absolute error= 0.3100)}
%			\end{center}
%		\end{figure}
%		
%		
%		\begin{figure}
%			\begin{center}
%				\includegraphics[scale=0.6]{figures/EmpVar112.jpg}
%				\includegraphics[scale=0.6]{figures/TheoVar112.jpg}
%				\caption{Empirical Var(L) vs Theoretical Var (R) (max absolute error= 0.0065)}
%			\end{center}
%		\end{figure}

\begin{figure}
	\begin{center}
		\includegraphics[scale=0.23]{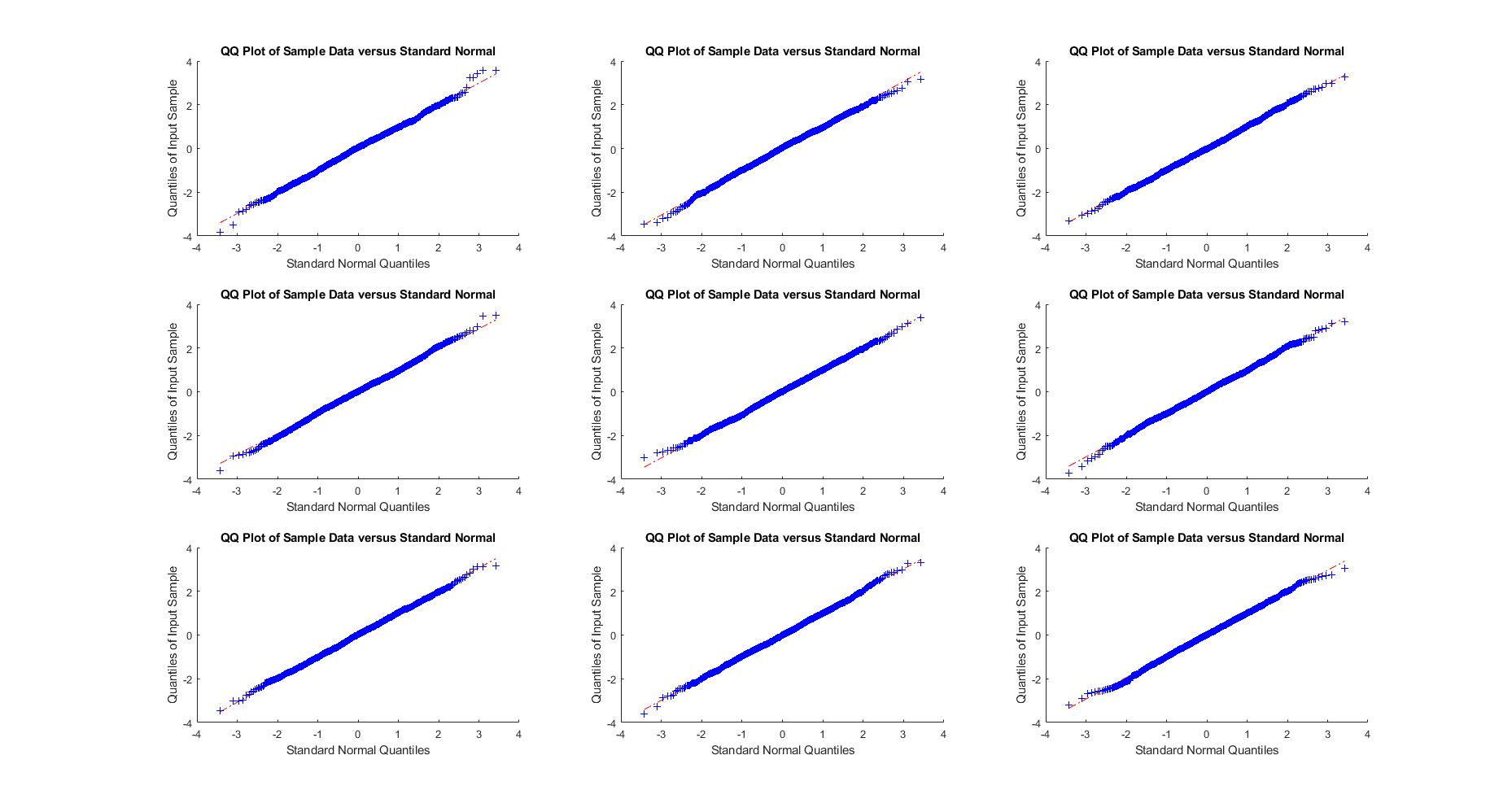}
		\caption{qqplots for normalized LSS of 9 different SBMs, test function $x^4$  }
		\label{qqplot:simuch2b}
	\end{center}
\end{figure}

\subsection{Experiments on the data-driven matrix $\hat{H}$}

We have also conducted numerical experiments for the data-driven matrix $\hat{H}$. The simulation set-up much follows the one used in Section \ref{sec:applicationtosbm}. The main purpose is to verify whether the limiting distributions of  linear spectral statistics of $\hat{H}$ would be the same as those of $H$.

Towards this end, we display relative qqplots of linear spectral statistics from $H$ and $\hat{H}$, respectively. Under distributional identity, qqplots would coincide with the identity line $y=x$.

\begin{example}\label{ex:hatsimulation}
	The SBM parameters are as follows:
	
	$K=6$, $\alpha=[0.1,0.15,0.2,0.25,0.1,0.2]$, $N=1000$, $N_r=1600$, $\tilde{P}=(p_i-q_j)I+q_j11^T,$ where $p_i=\frac{2i-1}{10}$, $q_j=\frac{2j-1}{10}$.
	Test functions are $f_1=x^4,$  $f_2=x^5,$ $f_3=exp(x).$
\end{example}

\begin{figure}
	\begin{center}
		\includegraphics[scale=0.235]{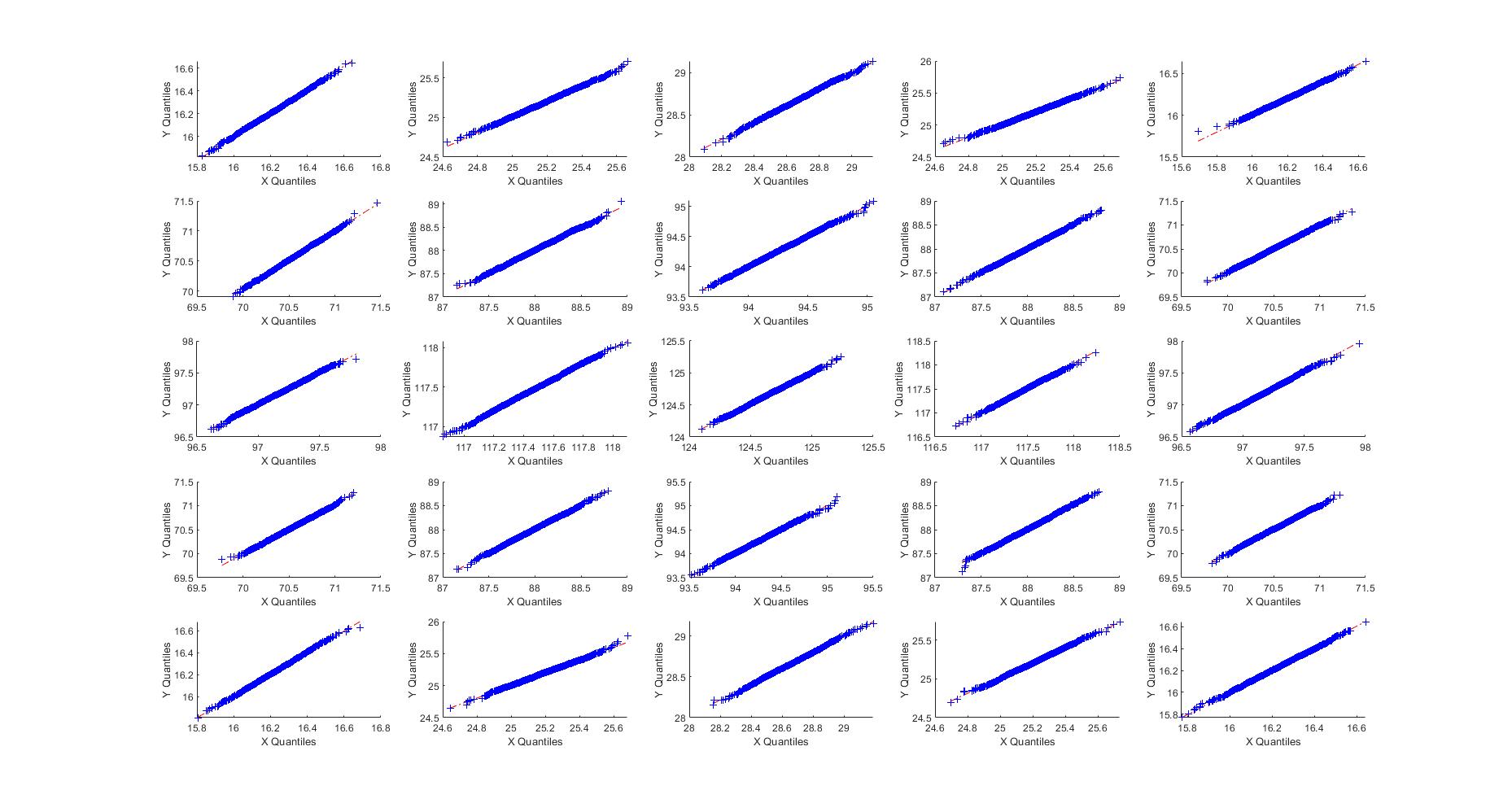}
		\caption{qqplots  of  $	L_n(x^4)[H]-L_n(x^4)[\hat{H}]$  }
		\label{qqplot:hatx4}
	\end{center}
\end{figure}

\begin{figure}
	\begin{center}
		\includegraphics[scale=0.235]{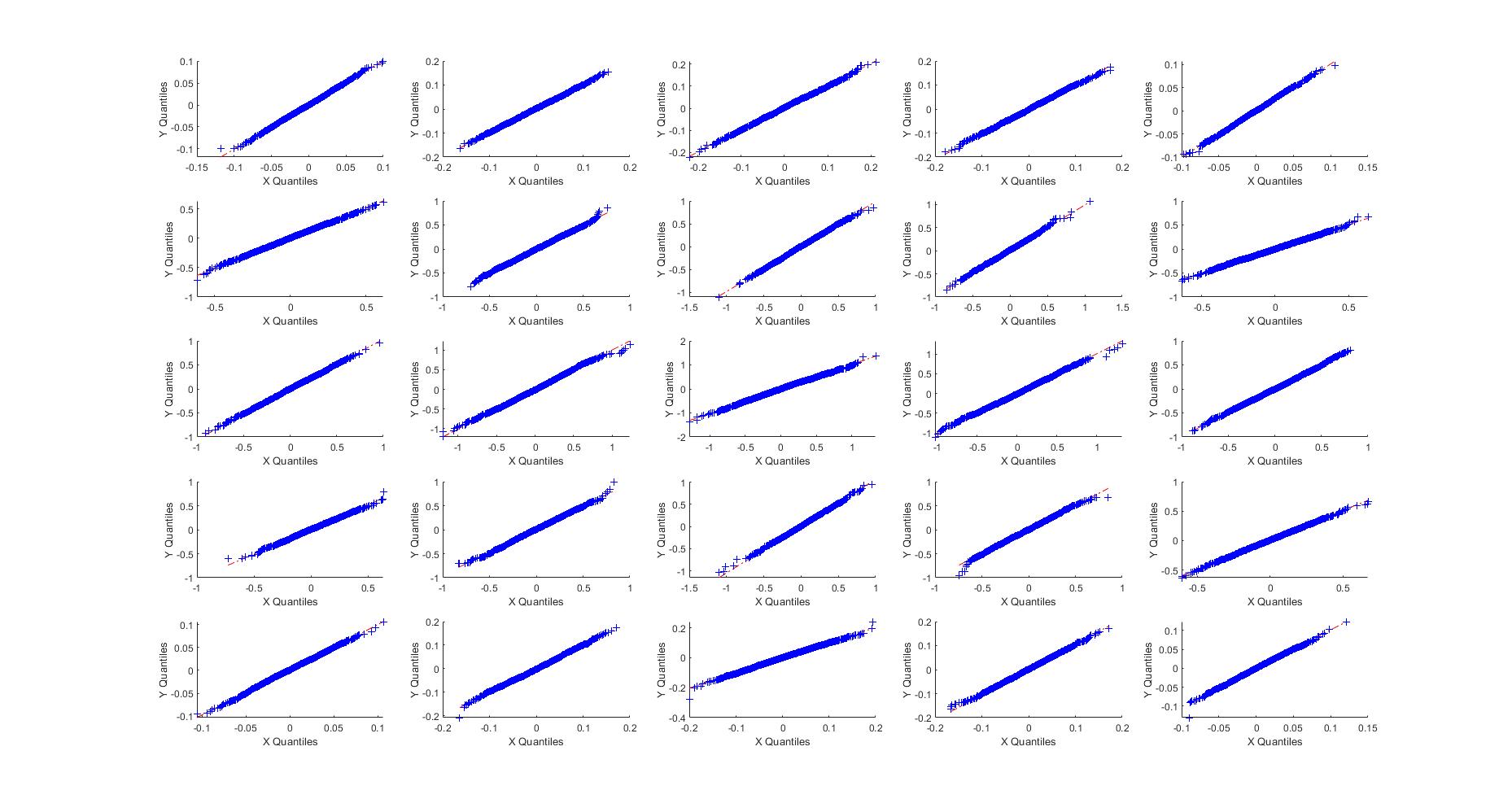}
		\caption{qqplots  of  $L_n(x^5)[H]-L_n(x^5)[\hat{H}]$  }
		\label{qqplot:hatx5}
	\end{center}
\end{figure}

\begin{figure}
	\begin{center}
		\includegraphics[scale=0.235]{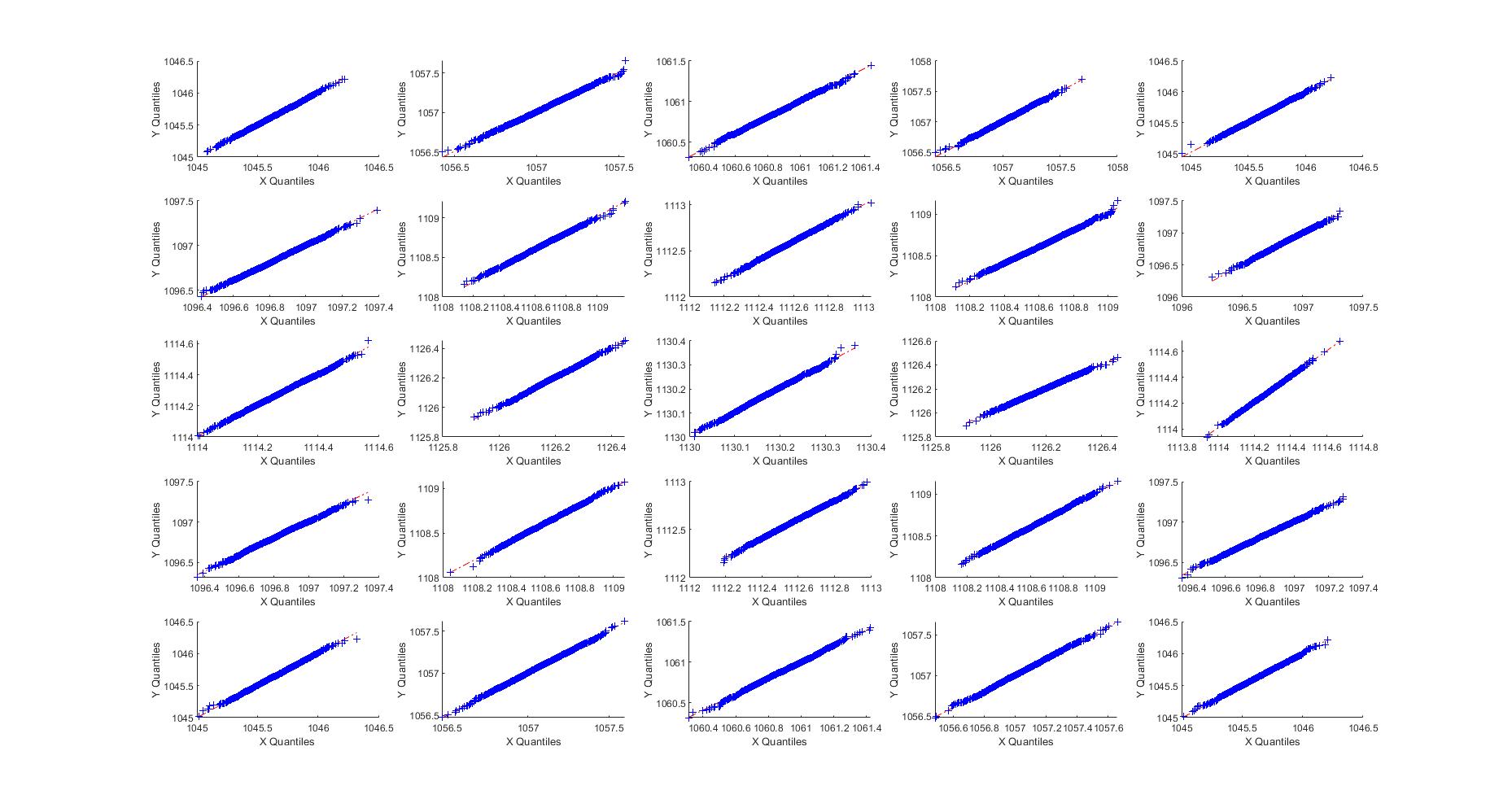}
		\caption{qqplots  of  $L_n(exp(x))[H]-L_n(exp(x))[\hat{H}]$  }
		\label{qqplot:hatexpx}
	\end{center}
\end{figure}

The empirical qqplots are given in Figure \ref{qqplot:hatx4} \ref{qqplot:hatx5} \ref{qqplot:hatexpx}. It can be seen that these qqplots are basically on the line $y=x$, which gives a good empirical confirmation of Theorem \ref{thm:mainBlockWigneremp}.

\section{Conclusion}

In this paper, we consider two applicable renormalizations $\left(\frac{A_{ij}-p_{ij}}{\sqrt{n}}\right)$ and $\left(\frac{A_{ij}-\hat{p}_{ij}}{\sqrt{n}}\right)$ of adjacency matrices of the stochastic block models. The CLTs of linear spectral statistics for both renormalizations are derived. The situations are fundamentally different from the existing literature in the sense that $\left(\frac{A_{ij}-p_{ij}}{\sqrt{n}}\right)$ induces a block-Wigner-type matrix whose LSD is no longer guaranteed to be the semicircle law but governed by the so-called quadratic vector equations introduced in \cite{Ajanki}. And the CLT for LSS also requires finer tools from the local law estimations. Meanwhile, $\left(\frac{A_{ij}-\hat{p}_{ij}}{\sqrt{n}}\right)$ is further perturbed by a low rank yet correlated structure, whose non-decaying correlations among the entries increase the difficulty of analysis.

We discuss several directions for future research. First, the CLTs introduced here are still in the dense regime of the stochastic block model. While \cite{Kevin2018} provides a more subtle analysis of the local law for the Erd\H{o}s-R\'enyi model in the sparser regime, it makes a local law for the sparse stochastic block model possible. Thus, the CLT for LSS of SBM in the sparse regime could be doable.

Second, a natural question is that for more general Wigner-type matrices, for instance, when the patterns explore more complex structures, can we get some CLTs or non-CLTs?   For instance, if the number of communities for the SBM is growing along with $n$ or the random graph model is defined via a graphon approach \cite{Airoldi2013}\cite{Yizhe2019}, then how will the linear spectral statistics  behave?

\newpage
\appendix
\section{Detailed calculations for the proof  of Theorem \ref{thm:mainBlockWigner}}\label{sec:detailedcalculation}
In this section, we will show the details of the calculation of the mean function in Section \ref{sec:mean1}-\ref{sec:mean2}, covariance function in Section \ref{sec:covlm}-\ref{sec:covend}, proof of normality in Section \ref{sec:gaussianitynonhat}, and tightness of the process in Section \ref{sec:tightness} for the block-Wigner-type matrices $H$. 
%
%In Section \ref{app:pfofhatversion}, we prove theorem \ref{thm:mainBlockWigneremp} which fill the image of certain non-Wigner-type matrices introduced in Section \ref{sec:hatversionintro}.
%
%First, we prove the Corollary \ref{Cor:QVE} which provides bounds for moments via the local law.
%
%\begin{proof}[Proof of Corollary \ref{Cor:QVE}]
%By definition, $\forall \varepsilon, D>0, $ there exists $n_0=n_0(\varepsilon,D)$ s.t. when $n>n_0$ 
%	$$P(\max_{ij}|G_{ij}-\delta_{ij}m_i(z)|>\frac{n^{\varepsilon}}{\sqrt{n}})\le n^{-D}.$$
%	Note also that for the resolvent we have the essential upper bound that 
%	$|G_{ij}(z)|\le |\frac{1}{\Im(z)}|$, then 
%	\begin{align*}
%		&\E |G_{ij}(z)-\delta_{ij} m_i(z)|^k\\ 
%		=&\E|G_{ij}(z)-\delta_{ij} m_i(z)|^k1_{|G_{ij}(z)-\delta_{ij} m_i(z)|\le \frac{n^{\varepsilon}}{\sqrt{n}}}+\E|G_{ij}(z)-\delta_{ij} m_i(z)|^k1_{|G_{ij}(z)-\delta_{ij} m_i(z)|>\frac{n^{\varepsilon}}{\sqrt{n}}}\\ 
%		\le &(\frac{n^{\varepsilon}}{\sqrt{n}})^k+\frac{1}{|Im(z)|^k}\frac{1}{n^D}.
%	\end{align*}
%	
%	So as long as we can fix $\varepsilon$ to be small positive number and let $k$ be upper bounded by some constant $K_0$ (which is rational since we need differentiation at most 4 times, which yields a limited power for $|G_{ij}(z)|$), $D= \frac{1}{2}K_0+1$, then we  may get the upper bound 
%	$$\E |G_{ij}(z)-\delta_{ij} m_i(z)|^k\le \frac{n^{\varepsilon}}{n^{k/2}}, \forall n\ge n_0(\varepsilon/K_0,\frac{1}{2}K_0+1).$$
%\end{proof}

\begin{remark}
	 Corollary~\ref{Cor:QVE} will be extensively used in our proof. Since $\varepsilon>0$ is arbitrarily small and essentially $n^{\varepsilon}$ can be substituted by $\log(n)^k$ for some large enough $k$ in these large-deviation bounds. Sometimes we will use $n^{\varepsilon}$ for simplicity when it is actually  $n^{k_0\varepsilon}$ for some  positive integer $k_0$ which is independent of $n$.
\end{remark}

%\subsection{System of equations for the mean function of $Tr G(z)$}\label{app:mean}
Recall that in Section \ref{sec:outlineofpfblockWigner}, we decompose the mean function $Tr G(z)$ into several components.
Starting from  $I_{1,1}$ in \eqref{eq:meandecomposition}, we need to calculate $\E \underline{T_lG(z)T_mG(z)}$ to order 1. 

\subsection{System of equations for $\E \underline{G(z)T_lG(z)T_m}$}\label{sec:mean1}

\begin{proof} [Proof of Lemma \ref{eq:trtgtg}]
	By the identity $G(z)=\frac{1}{z}(HG(z)-I)$, we have
	
	\begin{equation*}
		\begin{aligned}
			\E \underline{G(z)T_lG(z)T_m} &= \frac { 1 } { z } \E \underline { H G T_l G T_m } - \frac { 1 } { z } \underline {T_l G T_m}=\frac { 1 } { z } \E \underline{ H G T_l G T_m } - \delta_{lm}\frac { 1 } { z }\E \underline {T_l G}.
		\end{aligned}
	\end{equation*}
	
	Then by the cumulant expansion formula,
	\begin{align*}
		&\E \underline{ H G T_l G T_m } =\frac{1}{n} \E\sum _ { i j } H _ {i j } ( G T_l G T_m ) _ { ij }\\
		=&  \sum _ { i j } \frac{\kappa_{ij}^{(2)}}{n^2}\E \frac { \partial ( G T_l G T_m ) _ { j i } } { \partial  H_{ij}  }+ \sum _ { i j } \frac{\kappa_{ij}^{(3)}}{2!n^{5/2}}\E \frac { \partial^2 ( G T_l G T_m) _ { j i } } { \partial  H_{ij}^2  }\\ 
		&+\sum _ { i j } \frac{\kappa_{ij}^{(4)}}{3!n^{3}}\E \frac { \partial^3 ( G T_l G T_m ) _ { j i } } { \partial  H_{ij}^3  }+\varepsilon_{GTGT,4},
	\end{align*}
	where by the cumulant expansion and the trivial bound, the error term satisfies 
	$|\varepsilon_{GTGT,4}| \leq C \sum_{i,j}\sup _{t}\left|f^{(3+1)}(t)\right| E\left[|H_{ij}|^{3+2}\right]=O(\frac{1}{\sqrt{n}})$ thus minor.

	In the meantime, note that when we take derivatives $\frac{\partial^k (GT_lGT_m)_{ji}}{\partial^k H_{ij}}$,  the terms with the largest order of magnitude should be the ones with the form $(\cdot)_{ii}(\cdot)_{jj}\times(\cdot)$, which  will be of order 1 since  $||G||\le \frac{1}{\Im z}$ and $||T_l||=1,\forall l\in [K]$, so 
	
	\begin{equation}\label{eq:trivialboundforsumofo1terms}
		\begin{aligned}&\sum _ { i j } \frac{\kappa_{ij}^{(3)}}{2!n^{5/2}}\E \frac { \partial^2 ( G T_l G T_m ) _ { j i } } { \partial  H_{ij}^2  }=O(\frac{1}{\sqrt{n}}).\\
			&\sum _ { i j } \frac{\kappa_{ij}^{(4)}}{3!n^{3}}\E \frac { \partial^2 ( G T_l G T_m ) _ { j i } } { \partial  H_{ij}^2  }=O(\frac{1}{n}).\\
		\end{aligned}
	\end{equation}
	It follows that
	\begin{align*}
		&z\E \underline{GT_lGT_m}+\delta_{lm}\E \underline{T_lG}=\E \underline{H G T_l G T_m} = \frac{1}{n}\E\sum _ { i j } H _ {i j } ( G T_l G T_m ) _ { ij }\\ =& \frac{1}{n}\sum _ { i j }\frac{\kappa_ { i j }^{(2)}}{n}\E\frac { \partial ( G T_l G T_m ) _ { j i } } { \partial  H_{ij}  }+O(\frac{1}{\sqrt{n}})\\  
%		=&-\frac{1}{n^2}\sum_{ij}\kappa_{ij}^{(2)}\E(  e_j'G(e_ie_j'+e_je_i')GT_lGT_me_i+\E  e_j'GT_lG(e_ie_j'+e_je_i')GT_me_i )+O(\frac{1}{\sqrt{n}})\\ 
		=&-\frac{1}{n^2}\sum_{ij}\kappa_{ij}^{(2)}\E[  G_{ji}(GT_lGT_m)_{ji}+G_{jj}(GT_lGT_m)_{ii} \\ 
		& +(GT_lG)_{ji}(GT_m)_{ji}+ (GT_lG)_{jj}(GT_m)_{ii}]+O(\frac{1}{\sqrt{n}})\\ 
%		=&-\frac{1}{n^2}\sum_{ij}\kappa_{ij}^{(2)}\E(  G_{jj}(GT_lGT_m)_{ii} + (GT_lG)_{jj}(GT_m)_{ii})+O(\frac{n^{\varepsilon}}{\sqrt{n}})\\ 
		=&-\frac{1}{n^2}\sum_{k_1,k_2=1}^K\sum_{i\in C_{k_1},j\in C_{k_2}}Q^{(2)}_{k_1k_2}\E(G_{jj}(GT_lGT_m)_{ii} + (GT_lG)_{jj}(GT_m)_{ii})+O(\frac{n^{\varepsilon}}{\sqrt{n}})\\ 
%		=&-\sum_{k_1,k_2=1}^KQ^{(2)}_{k_1k_2}\E( \underline{T_{k_2}G}\ \underline{GT_lGT_mT_{k_1}} + \underline{GT_lGT_{k_2}}\ \underline{GT_mT_{k_1}})+O(\frac{n^{\varepsilon}}{\sqrt{n}})\\ 
%		=&-\sum_{k_2=1}^KQ^{(2)}_{mk_2}\E( \underline{T_{k_2}G}\ \underline{GT_lGT_m} + \underline{GT_lGT_{k_2}}\ \underline{GT_m})+O(\frac{n^{\varepsilon}}{\sqrt{n}})\\ 
		=&-\sum_{k=1}^KQ^{(2)}_{mk}(\alpha_{k}M_{k}(z)\E \underline{GT_lGT_m} +\alpha_mM_m(z) \E \underline{GT_lGT_{k}})+O(\frac{n^{\varepsilon}}{\sqrt{n}}).
	\end{align*}
	If we adopt the notation 
	$$X_{m}^{(l)}:=\E \underline{GT_lGT_m},$$
	then we may rewrite the above system of equations as 
	\begin{equation}
		\begin{aligned}
			(1+\frac{1}{z}\sum_{k=1}^KQ^{(2)}_{mk}\alpha_kM_k+\frac{1}{z}\alpha_mM_mQ^{(2)}_{mm})X_m^{(l)}+\frac{1}{z}\alpha_mM_m\sum_{k=1,k\ne m}^KQ^{(2)}_{mk}X_k^{(l)}=-\delta_{ml}\frac{1}{z}\alpha_l M_l.
		\end{aligned}
	\end{equation} 
	Now we have the  system of equations (\ref{eq:GTGTl}) for vector $[\E \underline{GT_lGT_1},\cdots, \E \underline{GT_lGT_K}]$ and the  system of equations (\ref{eq:MGTGT}) for matrix $\big(\E \underline{GT_lGT_k} \big)_{l,k=1}^K$.

\end{proof}    
\begin{remark}
	Further from the QVE (\ref{eq:QVE}), for $z \in \mathbb{C}_+$ and sufficiently bounded away from the spectrum of $H$, we have
	\begin{equation*}
		-\frac{1}{M_{l}(z)}=z+\sum_{m=1}^{K} Q^{(2)}_{lm} \alpha_mM_{m}(z), \quad \text { for all } \quad l=1, \ldots, K,\quad 
	\end{equation*}
	one can see that for different $l$'s. The coefficient matrices are the same, after simplification, we have that the matrix $M_{GTGT}(z)=\big( \underline{G(z)T_lG(z)T_m}\big)_{l,m=1}^K$ adopts this simple explicit form
	\begin{equation}
		M_{GTGT}(z)=-\big(Q^{(2)}-Diag ([\frac{1}{\alpha_1M_1^2(z)},\cdots,\frac{1}{\alpha_KM_K^2(z)}]^{\top})\big)^{-1},
	\end{equation}
	which is symmetric and in accordance with the tracial property which leads to $$Tr(GT_lGT_m)=Tr(GT_mGT_l).$$
	
	One may be concerned about the singularity problem. Simply note that $|M_j(z)|\le \frac{1}{|\Im z|}, j\in[K]$. Then when $|\Im z|^2\ge \max_{k\in [K]} \alpha_k \sum_{j=1}^KQ^{(2)}_{kj},$ the matrix $$\big(Q^{(2)}-Diag ([\frac{1}{\alpha_1M_1^2(z)},\cdots,\frac{1}{\alpha_KM_K^2(z)}]^{\top})\big)$$ becomes diagonal dominant, thus non-singular. Similar things happen when we are near the real axis but also sufficiently bounded  away from the edge.	 Then we can ensure the existence and uniqueness of the solutions of our systems of equations. All  we have to pay is to select a larger contour when we apply the Cauchy integral trick to proceed from the trace of the resolvent to the linear spectral statistics. Due to the homogeneity of the coefficient, similar arguments hold for other systems of equations of interest. 
	
	Specifically, we introduce the parameter $\varepsilon_0$, s.t. for $z\in \mathbb{C}\backslash B_{\varepsilon}(\sigma(\hat{H}))$, the existence and uniqueness of the solution are guaranteed by the above mechanism.
\end{remark}

\subsection{Leading term for $\frac{1}{n}\E\sum_{i,j=1}^n\kappa_{ij}^{(2)}G_{ii}G_{jj}$ and system of equations for $\E Tr(T_lG)$}\label{sec:mean2}

The next task is to identify the leading term of {$\frac{1}{n}\E\sum_{i,j=1}^n\kappa_{ij}^{(2)}G_{ii}G_{jj}$.  Note that we need to calculate these terms up to the order $1$. The
	problem arises that  the trivial upper bound \(|G_{jj}(z)-m_{j}(z)|\prec \frac{1}{\sqrt{n}}\) is far from enough since it only yields
	\(\frac{1}{n}\sum_{i,j=1}^n\kappa_{ij}^{(2)}|G_{ii}G_{jj}-m_im_j|\prec n^{1/2}\).

	Recall the decomposition in \eqref{eq:meandecomposition2}, we further write
	\begin{equation}
		\begin{aligned}
%			&\frac{1}{n}\E\sum_{i,j=1,i\ne j}^{n}\kappa_{ij}^{(2)}G_{ii}G_{jj}=\frac{1}{n}\E\sum_{k_1,k_2=1}^{K}\sum_{i\in C_{k_1},j\in C_{k_2},i\ne j}\kappa_{ij}^{(2)}G_{ii}G_{jj}\\ 
		I_{1,2}	=&\frac{1}{n}\E\sum_{k_1,k_2=1}^{K}Q^{(2)}_{k_1k_2}\sum_{i\in C_{k_1},j\in C_{k_2}}G_{ii}G_{jj}-\frac{1}{n}\E\sum_{k=1}^{K}Q^{(2)}_{kk}\sum_{i\in C_k}G_{ll}^2 \\
			=&I_{1,2,1}+I_{1,2,2}.
		\end{aligned}
	\end{equation}
	where 
	\begin{equation}
		\begin{aligned}
			I_{1,2,1}=&\frac{1}{n}\E\sum_{k_1,k_2=1}^{K}Q^{(2)}_{k_1k_2}\sum_{i\in C_{k_1},j\in C_{k_2}}G_{ii}G_{jj}=n\E\sum_{k_1,k_2=1}^{K}Q^{(2)}_{k_1k_2}\underline{T_{k_1}G}\ \underline{T_{k_2}G}.\\ 
		\end{aligned}
	\end{equation}
	Note that we cannot calculate $I_{1,2,1}$ to the desired order directly.
	
	Simply notice that by local law, we have 
	\begin{equation}
		\begin{aligned}
			\E[|Tr(T_{k}G)-\E Tr(T_{k}G)|^2]\le{n^{\varepsilon}}.
		\end{aligned}
	\end{equation}
	By Cauchy-Schwarz inequality it's easy to see that 
	\begin{equation}
		\begin{aligned}
			\E[|(Tr(T_{k_1}G)-\E Tr(T_{k_1}G))(Tr(T_{k_2}G)-\E Tr(T_{k_2}G))|]\le{n^{\varepsilon}}.
		\end{aligned}
	\end{equation}
	Then we have
	\begin{equation}
		\begin{aligned}
			I_{1,2,1}=&{n}\E\sum_{k_1,k_2=1}^{K}Q^{(2)}_{k_1k_2}\underline{T_{k_1}G}\ \underline{T_{k_2}G}={n}\sum_{k_1,k_2=1}^{K}Q^{(2)}_{k_1k_2}\E \underline{T_{k_1}G}\E \underline{T_{k_2}G}+O(\frac{n^{\varepsilon}}{n}).\\ 
		\end{aligned}
	\end{equation}
	And for $I_{1,2,2}$ we can get the simple formulation
	
	\begin{equation}
		\begin{aligned}
			I_{1,2,2}=&-\frac{1}{n}\E\sum_{k=1}^KQ^{(2)}_{kk}\sum_{i\in C_k}G_{ii}^2=-\sum_{k=1}^KQ^{(2)}_{kk}\alpha_kM_k^2+O(\frac{n^{\varepsilon}}{\sqrt{n}}).
		\end{aligned}
	\end{equation}
	In other words, again we have obtained a function of $(\E (Tr(T_{k_1}G)), \E(Tr(T_{k_2}G)))$ on the RHS, note that the leading order terms of $(\E (Tr(T_{k_1}G)), \E(Tr(T_{k_2}G)))$, which are of order $n$, are known. So this motivates us to derive a system of equations for the subleading order terms of $\{\E Tr(T_kG)\}_{k=1}^K$, which are of order 1.

	\begin{proof}[Proof of Lemma \ref{eq:trtgsubleading}]
		By the cumulant expansion formula, we have the following equality for $\E Tr(T_kG)$:		
		\begin{equation*}
			\begin{aligned}
				\E Tr(T_kG)=&-\frac{\alpha_kn}{z} + \E\frac{1}{z}Tr(HGT_k)=-\frac{\alpha_kn}{z}+ \E\frac{1}{z}\sum_{ij}(H_{ij}e_j'GT_ke_i)\\  
				=&-\frac{\alpha_k n}{z}
				- \E\frac{1}{z}\sum_{ij}\frac{\kappa_{ij}^{(2)}}{n}[G_{jj}(GT_k)_{ii}+G_{ji}(GT_k)_{ji}]+\E\frac{1}{z}\sum_{ij}\frac{\kappa_{ij}^{(3)}}{2!n^{3/2}}\frac{\partial^2 e_j'GT_ke_i}{\partial H_{ij}^2}\\ 
				&+\E\frac{1}{z}\sum_{ij}\frac{\kappa_{ij}^{(4)}}{3!n^2}\frac{\partial^3 e_j'GT_ke_i}{\partial H_{ij}^3}+\varepsilon_{\tilde{I}_{1,5}}\\  
				=&-\frac{\alpha_k n}{z}- \E\frac{1}{z}\sum_{k_1,k_2=1}^K\sum_{i\in C_{k_1},j\in C_{k_2},i\ne j}\frac{\kappa_{ij}^{(2)}}{n}[G_{jj}(GT_k)_{ii}+G_{ji}(GT_k)_{ji}]\\ 
				&+\E\frac{1}{z}\sum_{ij}\frac{\kappa_{ij}^{(4)}}{3!n^2}\frac{\partial^3 e_j'GT_ke_i}{\partial H_{ij}^3}+\varepsilon_{\tilde{I}_1}\\ 
				=&-\frac{\alpha_k n}{z}-\tilde{I}_{1,1}-\tilde{I}_{1,2}-\tilde{I}_{1,4}+\varepsilon_{\tilde{I}_{1}},
			\end{aligned}
		\end{equation*}
		where $\varepsilon_{\tilde{I}_{1}}$ consists of higher-order expansions of the formula 
		\begin{equation*}
			\begin{aligned}
				\varepsilon_{\tilde{I}_1,3}&=\frac{1}{z}\E\sum_{i,j}\frac{\kappa_{ij}^{(3)}}{2!n^{3/2}}2[G_{jj}G_{ij}(GT_k)_{ii}+G_{jj}G_{ii}(GT_k)_{ji}+G_{ji}G_{jj}(GT_k)_{ii}+G_{ji}G_{ji}(GT_k)_{ji}],
			\end{aligned}
		\end{equation*}

		$$|	\varepsilon_{\tilde{I}_1,5}|\le \sum_{i,j}C\sup_t|f^{(p+1)}(t)|\E[|t|^{p+2}]=O(\frac{1}{\sqrt{n}}).$$
		
		Then it suffices to show that 
		$$	\varepsilon_{\tilde{I}_1,3}^{l,m}=\frac{1}{z}\E\sum_{i\in C_l,j\in C_m}\frac{\kappa_{ij}^{(3)}}{2!n^{3/2}}2[G_{jj}G_{ij}(GT_k)_{ii}+G_{jj}G_{ii}(GT_k)_{ji}+G_{ji}G_{jj}(GT_k)_{ii}+G_{ji}G_{ji}(GT_k)_{ji}]$$ 
		are minor. Let $a=[G_{11},\ldots,G_{nn}]$,  $b_k=[(T_kG)_{11},\ldots,(T_kG)_{nn}]$.
		\begin{equation*}
			\begin{aligned}
				&z\varepsilon_{\tilde{I}_1,3}^{l,m}\\ 
				=&\E\sum_{i\in C_l,j\in C_m}\frac{\kappa_{ij}^{(3)}}{2!n^{3/2}}2[G_{jj}G_{ij}(GT_k)_{ii}+G_{jj}G_{ii}(GT_k)_{ji}+G_{ji}G_{jj}(GT_k)_{ii}+G_{ji}G_{ji}(GT_k)_{ji}]\\ 
				=&\frac{Q^{(3)}_{lm}}{n^{3/2}}(aGb_k^{\top}+a(GT_k)a^{\top}+b_kGa^{\top})+O(\frac{1}{\sqrt{n}})=O(\frac{1}{\sqrt{n}}).
			\end{aligned}
		\end{equation*}
		\begin{equation}
			\begin{aligned}
				\tilde{I}_{1,1}=& \E\frac{1}{z}\sum_{k_1,k_2=1}^K\sum_{i\in C_{k_1},j\in C_{k_2},i\ne j}\frac{\kappa_{ij}^{(2)}}{n}G_{jj}(GT_k)_{ii}\\ 
				=& \E\frac{1}{z}\sum_{k_2=1}^K\sum_{i\in C_{k},j\in C_{k_2}}\frac{Q^{(2)}_{k_1k_2}}{n}Tr(T_{k_2}G)Tr(GT_k)-\E\frac{1}{z}\sum_{i\in C_k}\frac{Q^{(2)}_{kk}}{n}G_{ii}^2\\ 
				=&\frac{1}{z}\sum_{k_2=1}^K\frac{Q^{(2)}_{k_1k_2}}{n}\E Tr(T_{k_2}G)\E Tr(GT_k)-\E\frac{1}{z}\frac{Q^{(2)}_{kk}}{n}\alpha_k n M_{k}^2\\ &+\frac{1}{z}\sum_{k_2=1}^K\frac{Q^{(2)}_{k_1k_2}}{n}[\E Tr(T_{k_2}G) Tr(GT_k)-\E Tr(T_{k_2}G) \E Tr(GT_k)]\\ 
				&-\E\frac{1}{z}\sum_{i\in C_k}\frac{Q^{(2)}_{kk}}{n}[G_{ii}^2-M_k^2(z)].\\ 
			\end{aligned}
		\end{equation}
		
		One may notice that similar to the cases above, we have 
		\begin{align*}
			&|\E Tr(T_{k_2}G) Tr(GT_k)-\E Tr(T_{k_2}G) \E Tr(GT_k)|\\ 
			&\le |\E [Tr(T_{k_2}G)|-\E Tr(T_{k_2}G)][Tr(GT_k)-\E Tr(GT_k)]|\le  n^{2\varepsilon} ,
		\end{align*}
		and
		\begin{align*}
			\E|G_{ii}^2(z)-M_{k}^2(z)|=\E |G_{ii}(z)-M_{k}(z)||G_{ii}(z)+M_{k}(z)|\le \frac{n^{\varepsilon}}{\sqrt{n\Im(z)}}\frac{2}{Im(z)},\forall i\in C_k,
		\end{align*}
		thus, \begin{equation}
			\tilde{I}_{1,1}=\frac{1}{z}\sum_{k_2=1}^K\frac{Q^{(2)}_{k_1k_2}}{n}\E Tr(T_{k_2}G)\E Tr(GT_k)-\E\frac{1}{z}Q^{(2)}_{kk}\alpha_k  M_{k}^2+o(1).
		\end{equation}
		
		Similarly, we can proceed to $\tilde{I}_{1,2}$,
		\begin{equation}
			\begin{aligned}
				\tilde{I}_{1,2}=& \E\frac{1}{z}\sum_{k_1,k_2=1}^K\sum_{i\in C_{k_1},j\in C_{k_2},i\ne j}\frac{\kappa_{ij}^{(2)}}{n}G_{ji}(GT_k)_{ji}\\
				=&\E\frac{1}{z}\sum_{k_1,k_2=1}^K\sum_{i\in C_{k_1},j\in C_{k_2}}\frac{Q^{(2)}_{k_1k_2}}{n}G_{ji}(GT_k)_{ji}-\E\frac{1}{z}\frac{Q^{(2)}_{kk}}{n}\sum_{i\in C_k}G_{ii}(GT_k)_{ii}\\ 
				=&\frac{1}{z}\sum_{k_2=1}^K{Q^{(2)}_{kk_2}}\E \underline{G	T_{k_2}GT_k}-\E\frac{1}{z}Q^{(2)}_{kk}\alpha_kM_k^2+o(1). 
			\end{aligned}
		\end{equation}
		
		\begin{equation}
			\begin{aligned}
				\tilde{I}_{1,4}&=\E\frac{1}{z}\sum_{i,j}\frac{\kappa_{ij}^{(4)}}{3!n^2}\frac{\partial^3 e_j'GT_ke_i}{\partial H_{ij}^3}=\E\frac{1}{z}\sum_{i\in C_k,j,i\ne j}\frac{\kappa_{ij}^{(4)}}{n^2}G_{ii}^2G_{jj}^2+O(\frac{n^{\varepsilon}}{\sqrt{n}})\\ 
				&=\E\frac{1}{z}\sum_{l=1}^K\sum_{i\in C_k,j\in C_l,i\ne j}\frac{Q_{kl}^{(4)}}{n^2}M_k^2M_l^2+O(\frac{n^{\varepsilon}}{\sqrt{n}})=\frac{1}{z}\sum_{l=1}^K{Q_{kl}^{(4)}}\alpha_k\alpha_lM_k^2M_l^2+O(\frac{n^{\varepsilon}}{\sqrt{n}}).\\ 
			\end{aligned}	
		\end{equation}
		
		Then we may derive the system of equation for $Y_k=\E Tr(T_kG)-\alpha_knM_k$,
		\begin{equation}
			\begin{aligned}
				\alpha_knM_k+Y_k=&-\frac{\alpha_kn}{z}-\frac{1}{z}\sum_{k_2=1}^KQ^{(2)}_{kk_2}\frac{(\alpha_{k_2}nM_{k_2}+Y_{k_2})(\alpha_knM_k+Y_k)}{n}\\ 
				&-\frac{1}{z}\sum_{k_2=1}^K{Q^{(2)}_{kk_2}}\E \underline{GT_{k_2}GT_k}+\frac{2Q^{(2)}_{kk}\alpha_kM_k^2}{z}-\frac{1}{z}\sum_{l=1}^K{Q_{kl}^{(4)}}\alpha_k\alpha_lM_k^2M_l^2.
			\end{aligned}
		\end{equation}
		
		The above equation could be decomposed into two parts, one is of order $n$, while the other is of order 1. One may easily verify that the order $n$ part yields
		$$M_k=-\frac{1}{z}-\frac{\sum_{k_2=1}^KQ^{(2)}_{kk_2}\alpha_{k_2}M_{k_2}M_k}{z},$$
		which directly follows from the quadratic vector equation \eqref{eq:QVE}, thus canceled.

		The order $1$ part yields
		\begin{equation*}
			\begin{aligned}
				Y_k=&-\frac{1}{z}\sum_{l=1}^KQ^{(2)}_{kl}[\alpha_{l}M_{l}Y_k+\alpha_{k}M_{k}Y_{l}]-\frac{1}{z}\sum_{l=1}^K{Q^{(2)}_{kl}}\E \ul{GT_{l}GT_k}+\frac{2Q^{(2)}_{kk}\alpha_kM_k^2}{z}\\ 
				&-\frac{1}{z}\sum_{l=1}^K{Q_{kl}^{(4)}}\alpha_k\alpha_lM_k^2M_l^2,
			\end{aligned}
		\end{equation*}
		which reformulates into our (\ref{eq:TrG-n}).

	\end{proof}
	
	%	
	%	\subsection{System of equations for $\frac{1}{n}\E Tr(G(z_1)T_{s_1}G(z_2)T_{s_2}\cdots G(z_t)T_{s_t})$}
	%	
	%	{\color{red} STOP, for now}
	%	The above calculation motivate us to derive a unified frame for this kind self-reproducing structures.
	%	
	%	Assume that we have knowledge for all the lower structures.
	%	
	%	\begin{align*}
	%		&\frac{1}{n}\E Tr[\Pi_{i=1}^s(G(z_i)T_{s_i})]\\ =&\frac{1}{n}\E Tr[\frac{HG(z_1)T_{s_1}-T_{s_1}}{z_1}\Pi_{i=2}^s(G(z_i)T_{s_i})]\\ 
	%		=&\frac{1}{n}\E Tr[\frac{HG(z_1)T_{s_1}-T_{s_1}}{z_1}\Pi_{i=2}^s(G(z_i)T_{s_i})]\\ 
	%        =&\frac{1}{nz_1}\E Tr[HG(z_1)T_{s_1}\Pi_{i=2}^s(G(z_i)T_{s_i})]-\frac{\delta_{s_1s_t}}{nz_1}\E Tr[\Pi_{i=2}^s(G(z_i)T_{s_i})]\\ 
	%        =&\frac{1}{nz_1}\E \sum_{j,m}H_{jm}[\Pi_{i=1}^t(G(z_i)T_{s_i})]_{mj}-\frac{\delta_{s_1s_t}}{nz_1}\E Tr[\Pi_{i=2}^s(G(z_i)T_{s_i})]\\ 
	%	\end{align*}
	% where 
	%	\begin{align*}
	%	&\E \sum_{j,m}H_{jm}[\Pi_{i=1}^t(G(z_i)T_{s_i})]_{mj}\\ 
	%	=& \E \sum_{j,m}\frac{\kappa_{jm}^{2}}{n}\sum_{k=1}[\Pi_{i=1}^{k-1}(G(z_i)T_{s_i})(G(z_k)(e_me_j'+e_je_m)G(z_k))\Pi_{i=k+1}^{t}(G(z_i)T_{s_i})]_{mj}
	%	\end{align*}

	\subsection{System of equations for $Cov_{lm}(z_1,z_2)$}\label{sec:covlm}

	As stated in Section \ref{sec:outlineofpfblockWigner}, the estimation for  $Cov(z_1,z_2)$  can be decomposed into the summation of the block-wise covariance functions $\{Cov_{lm}(z_1,z_2)\}_{l,m=1}^K$. In this subsection, we will derive the system of equations for $\{Cov_{lm}(z_1,z_2)\}_{l,m=1}^K$.
	
	In this section and thereafter, we will use $\langle \cdot\rangle$ for centered random variables. First, note that for any two random variables $X$ and $Y$, we have
	\begin{align*}
		\E\langle X\rangle \langle Y\rangle=\E[X-\E X][Y-\E Y]=\E X[Y-\E Y]=\E X\langle Y\rangle,
	\end{align*}
	then by cumulant expansion formula, 
	\begin{equation}
		\begin{aligned}
			&\frac{1}{n^2}z_1\E Cov_{lm}(z_1,z_2)=z_1\E\langle \underline{G(z_1)T_l}\rangle\langle\underline{T_mG(z_2)} =z_1\E\underline{G(z_1)T_l}\langle\underline{T_mG(z_2)} \rangle\\ 
			=&\E\underline{HG(z_1)T_l}\langle\underline{T_mG(z_2)} \rangle
			=\frac{1}{n}\E\sum_{i\in C_l,j}H_{ij}G_{ij}(z_1)\langle\underline{T_mG(z_2)} \rangle\\ 
			=&\frac{1}{n}\sum_{a+b=0}^{\infty}\E\sum_{i\in C_l,j}\frac{\kappa^{(a+b+1)}_{ij}}{n^{\frac{a+b+1}{2}}a!b!}\frac{\partial^a G_{ij}(z_1)\partial^b\langle\underline{T_mG(z_2)}\rangle}{\partial H_{ij}^{(a+b)}} \\ 
			=&\sum_{a+b=0}^5I_{2,(a,b)}+\varepsilon_{I_2}.
		\end{aligned}
	\end{equation}

	Now we proceed to the detailed treatments for the terms $\{I_{2,(a,b)}\}_{a+b=1}^3$.  
	It can be  shown that $\sum_{a+b=4}^5I_{2,(a,b)}$ are minor via similar calculations, the details for calculating $\sum_{a+b=4}^5I_{2,(a,b)}$ are tedious and of minor importance thus omitted here and in the proof of \eqref{eq:covariancefunc}.
	
	\begin{proof}[Proof of \eqref{eq:covariancefunc}]

			\begin{align*}
				&I_{2,(1,0)}\\ 
%				&=-\frac{1}{n}\E\sum_{i\in C_l,j}\frac{\kappa_{ij}^{(2)}}{n}e_{i}'G(z_1)(e_{j}e_{i}'+e_{i}e_{j}')G(z_1)e_{j}\langle\underline{T_mG(z_2)} \rangle\\ 
%				 =&-\frac{1}{n}\E\sum_{k_1=1}^K\sum_{i\in C_l,j\in C_{k_1},i\ne j}\frac{Q^{(2)}_{lk_1}}{n}(G_{ij}^2(z_1)+G_{ii}(z_1)G_{jj}(z_1))\langle\underline{T_mG(z_2)} \rangle\\
%				=&-\frac{1}{n}\E\sum_{k_1=1}^K\sum_{i\in C_l,j\in C_{k_1}}\frac{Q^{(2)}_{lk_1}}{n}(G_{ij}^2(z_1)+G_{ii}(z_1)G_{jj}(z_1))\langle\underline{T_mG(z_2)} \rangle\\ 
%				&+ \frac{1}{n}\E\sum_{i\in C_l}\frac{Q^{(2)}_{ll}}{n}2G_{ii}^2(z_1)\langle\ul{T_mG(z_2)} \rangle\\ 
				=&-\frac{1}{n}\E\sum_{k_1=1}^K\frac{Q^{(2)}_{lk_1}}{n}[Tr(T_lG(z_1)T_{k_1}G(z_1))+Tr(T_lG(z_1))Tr(T_{k_1}G(z_1))]\langle \underline{T_mG(z_2)}\rangle\\ 
				&+\frac{2Q^{(2)}_{ll}}{n^2}\E\sum_{i\in C_l}[G_{ii}^2(z_1)-M_l^2(z_1)]\langle\underline{T_mG(z_2)}\rangle \\ 
%				=&-\E\sum_{k_1=1}^K\frac{Q^{(2)}_{lk_1}}{n}\langle\underline{T_lG(z_1)T_{k_1}G(z_1)}\rangle\langle \underline{T_mG(z_2)}\rangle-\E\sum_{k_1=1}^KQ^{(2)}_{lk_1}\underline{T_lG(z_1)}\ \underline{T_{k_1}G(z_1)}\ \langle \underline{T_mG(z_2)}\rangle\\ 
%				&+O(\frac{n^{\varepsilon}}{n^{5/2}})\\ 
				=&-\E\sum_{k_1=1}^K\frac{Q^{(2)}_{lk_1}}{n}\langle\underline{T_lG(z_1)T_{k_1}G(z_1)}\rangle\langle \underline{T_mG(z_2)}\rangle-\E\sum_{k_1=1}^KQ^{(2)}_{lk_1}\langle\underline{T_lG(z_1)}\rangle\ \langle\underline{T_{k_1}G(z_1)}\rangle\ \langle \underline{T_mG(z_2)}\rangle\\ 
				&-\E\sum_{k_1=1}^KQ^{(2)}_{lk_1}\langle\underline{T_lG(z_1)}\rangle\underline{T_{k_1}G(z_1)}\langle \underline{T_mG(z_2)}\rangle-\E\sum_{k_1=1}^KQ^{(2)}_{lk_1}\underline{T_lG(z_1)}\langle\underline{T_{k_1}G(z_1)}\rangle\langle \underline{T_mG(z_2)}\rangle \\ 
				&+O(\frac{n^{\varepsilon}}{n^{5/2}})\\ 
%				=&-\E\sum_{k_1=1}^K\frac{Q^{(2)}_{lk_1}}{n}\langle\underline{T_lG(z_1)T_{k_1}G(z_1)}\rangle\langle \underline{T_mG(z_2)}\rangle-\E\sum_{k_1=1}^KQ^{(2)}_{lk_1}\langle\underline{T_lG(z_1)}\rangle\ \langle\underline{T_{k_1}G(z_1)}\rangle\ \langle \underline{T_mG(z_2)}\rangle\\ 
%				&-\sum_{k_1=1}^KQ^{(2)}_{lk_1}\E\underline{T_{k_1}G(z_1)}\E\langle\underline{T_lG(z_1)}\rangle\langle \underline{T_mG(z_2)}\rangle-\sum_{k_1=1}^KQ^{(2)}_{lk_1}\E\underline{T_lG(z_1)}\E\langle\underline{T_{k_1}G(z_1)}\rangle\langle \underline{T_mG(z_2)}\rangle\\ 
%				& +O(\frac{n^{\varepsilon}}{n^{5/2}})\\ 
%				=&-\sum_{k_1=1}^KQ^{(2)}_{lk_1}\E\underline{T_{k_1}G(z_1)}\E\langle\underline{T_lG(z_1)}\rangle\langle \underline{T_mG(z_2)}\rangle-\sum_{k_1=1}^KQ^{(2)}_{lk_1}\E\underline{T_lG(z_1)}\E\langle\underline{T_{k_1}G(z_1)}\rangle\langle \underline{T_mG(z_2)}\rangle\\ 
%				& +O(\frac{n^{\varepsilon}}{n^{5/2}})\\ 
				=&-\sum_{k_1=1}^KQ^{(2)}_{lk_1}\alpha_{k_1}M_{k_1}(z_1)Cov_{lm}(z_1,z_2)-\sum_{k_1=1}^KQ^{(2)}_{lk_1}\alpha_lM_l(z_1)Cov_{k_1m}(z_1,z_2) +O(\frac{n^{\varepsilon}}{n^{5/2}}).
			\end{align*}

			\begin{align*}
				&I_{2,(0,1)}\\
%				&=-\frac{1}{n}\E\sum_{i\in C_l,j}\frac{\kappa^{(2)}_{ij}}{n}G_{ij}(z_1)\frac{1}{n}\sum_{k\in C_m}e_{k}'G(z_2)(e_{j}e_{i}'+e_{i}e_{j}')G(z_2)e_{k}\\  
%				=&-\frac{1}{n}\E\sum_{k_1=1}^K\sum_{i\in C_l,j\in C_{k_1},i\ne j}\frac{Q^{(2)}_{lk_1}}{n}G_{ij}(z_1)\frac{1}{n}\sum_{k\in C_m}(G_{kj}(z_2)G_{ik}(z_2)+G_{ki}(z_2)G_{jk}(z_2))\\ 
				=&-\frac{1}{n}\E\sum_{k_1=1}^K\sum_{i\in C_l,j\in C_{k_1}}\frac{Q^{(2)}_{lk_1}}{n}G_{ij}(z_1)\frac{1}{n}\sum_{k\in C_m}(G_{kj}(z_2)G_{ik}(z_2)+G_{ki}(z_2)G_{jk}(z_2))\\ 
				&+ \frac{1}{n}\E \sum_{i\in C_l}\frac{Q^{(2)}_{ll}}{n}G_{ii}(z_1)\frac{1}{n}\sum_{k\in C_m}(G_{ki}(z_2)G_{ik}(z_2)+G_{ki}(z_2)G_{ik}(z_2))\\ 
				=&-\E\sum_{k_1=1}^K\sum_{i\in C_l,j\in C_{k_1}}\frac{2Q^{(2)}_{lk_1}}{n^3}G_{ij}(z_1)(GT_mG)_{ij}(z_2)+ \E \sum_{i\in C_l}\frac{2Q^{(2)}_{ll}}{n^3}G_{ii}(z_1)(GT_mG)_{ii}(z_2)\\ 
				=&-\sum_{k_1=1}^K\frac{2Q^{(2)}_{lk_1}}{n^2}\E\ul{G(z_1)T_lG(z_2)T_mG(z_2)T_{k_1}}+\frac{2Q^{(2)}_{ll}}{n^2}M_{l}(z_1)\E \ul{(T_lGT_mG)(z_2)}+O(\frac{n^{\varepsilon}}{n^{5/2}}).
			\end{align*}

		We claim that in $I_{2,(1,0)}$ both $\frac{1}{n}\E\langle\underline{ T_lG(z_1)T_{k_1}G(z_1)} \rangle \langle \underline{T_mG(z_2)}\rangle$ and the triple-product term $\E\langle\underline{T_lG(z_1)}\rangle\ \langle\underline{T_{k_1}G(z_1)}\rangle\ \langle \underline{T_mG(z_2)}\rangle$
		will be the minor terms. The second one is of order $\frac{n^{\varepsilon}}{n^3}$  by Cauchy inequality, thus minor. The first one, however, requires a little bit effort.
		
		To get an sufficient upper bound for $\frac{1}{n}\E\langle\underline{ T_lG(z_1)T_{k_1}G(z_1)} \rangle \langle \underline{T_mG(z_2)}\rangle$, we only need to show that $\E|\langle \underline{T_{k_1}G(z)T_{k_2}G(z)}\rangle|^2$ is of order $O(n^{-t})$ for some $t>0$. By intuition from the classic Wigner matrix, the essential order for $\E|\langle \underline{T_{k_1}G(z)T_{k_2}G(z)}\rangle|^2$  should be $O(n^{-2})$. We refer to the Section \ref{sec:boundtgtgvar} for the details.
		
		Also, $I_{2,(0,1)}$ gives rise to the quantities $\E \ul{G(z_1)T_lG(z_2)T_mG(z_2)T_{k_1}}$ which will be treated in Section \ref{sec:tgtgtg} and $\E\ul{G(z_2)T_lG(z_2)T_m}$ which has already been studied in Section \ref{sec:mean1}.
		\begin{equation}
			\begin{aligned}
				I_{2,(2,0)}=&\frac{1}{n}\E\sum_{i\in C_l,j}\frac{\kappa_{ij}^{(3)}}{n^{3/2}}[(G_{ij}(z_1))^3+3G_{ii}(z_1)G_{jj}(z_1)G_{ij}(z_1)]\langle \underline{T_mG(z_2)}\rangle \\ 
				=&\frac{1}{n}\E\sum_{i\in C_l,j}\frac{\kappa_{ij}^{(3)}}{n^{3/2}}[3G_{ii}(z_1)G_{jj}(z_1)G_{ij}(z_1)]\langle\underline{T_mG(z_2)}\rangle+O(\frac{n^{\varepsilon}}{n^3}).
			\end{aligned}
		\end{equation}
		
		Note that one argument for higher-order expansion terms in the  cumulant expansion that we will use over and over again is that we can ignore the diagonal terms in many situations since there are only $n$ diagonal terms which are at most $O(1)$ each. To be more specific,
		\begin{align*}
			\frac{n^{\varepsilon}}{n^{\frac{d+1}{2}+2}}|\sum_{m=1}^n\frac{\partial^dG_{mm}}{\partial  H_{mm}^d}|\le \frac{n^{\varepsilon}}{n^{\frac{d+1}{2}+1}}\le \frac{C'n^{\varepsilon}}{n^{(d+3)/2}}=o(n^{-2}), d=2,3,
		\end{align*}
		then w.l.o.g. we can ignore the diagonal terms here. Further, because  $\langle\underline{T_mG}\rangle$ is $O_{\prec}(\frac{n^{\varepsilon}}{n})$, we only need to show that $\frac{1}{n}\E \sum_{i\in C_l,j\in C_m}\frac{\kappa_{ij}^{(3)}}{n^{3/2}}G_{ij}G_{ii}G_{jj}=o(\frac{1}{n^{1+t}})$ for any $t>0$ to  ensure that $I_{2,(2,0)}$ is vanishing.
		
		Note that the trivial bounds for $G_{ij}G_{ii}G_{jj}$ will not be sufficient. The trick here is to apply the cumulant expansion formula one more time to get certain equation of  $\frac{1}{n}\E \sum_{i\in C_l,j\in C_m}\frac{\kappa_{ij}^{(3)}}{n^{3/2}}G_{ij}G_{ii}G_{jj}=o(\frac{1}{n^{1+t}})$, hence the sufficient bounds.

		By cumulant expansion, we have
		\begin{equation}
			\begin{aligned}
				&\frac{1}{n}\E \sum_{i\in C_l,j\in C_m}\frac{\kappa_{ij}^{(3)}}{n^{3/2}}G_{ij}G_{ii}G_{jj}=\frac{Q^{(3)}_{lm}}{n^{5/2}}\E\sum_{i\in C_l,j\in C_m}G_{ij}G_{ii}G_{jj}-\frac{Q^{(3)}_{lm}}{n^{5/2}}\E\sum_{i\in C_l}\delta_{lm}G_{ii}^3\\ 
%				=&\frac{Q^{(3)}_{lm}}{n^{5/2}}\E\sum_{i\in C_l,j\in C_m}G_{ij}G_{ii}G_{jj}+O(n^{-3/2})\\ 
%				=&\frac{Q^{(3)}_{lm}}{n^{5/2}}\E\frac{1}{z}\sum_{i\in C_l,j\in C_m}(\sum_{k}H_{ik}G_{kj}-\delta_{ij})G_{ii}G_{jj}+O(n^{-3/2})\\ 
				=&\frac{Q^{(3)}_{lm}}{n^{5/2}}\E\frac{1}{z}\sum_{i\in C_l,j\in C_m}\sum_{k}H_{ik}G_{kj}G_{ii}G_{jj}+O(n^{-3/2})\\ 
				=&\frac{Q^{(3)}_{lm}}{n^{5/2}}\E\frac{1}{z}\sum_{i\in C_l,j\in C_m}\sum_{k}\frac{\kappa^{(2)}_{ik}}{n}\frac{\partial (G_{kj}G_{ii}G_{jj})}{\partial H_{ik}}+\frac{Q^{(3)}_{lm}}{n^{5/2}}\E\frac{1}{z}\sum_{i\in C_l,j\in C_m}\sum_{k}\frac{\kappa^{(3)}_{ik}}{2!n^{3/2}}\frac{\partial^2 (G_{kj}G_{ii}G_{jj})}{\partial H_{ik}^2}\\ 
				&+\frac{Q^{(3)}_{lm}}{n^{5/2}}\E\frac{1}{z}\sum_{i\in C_l,j\in C_m}\sum_{k}\frac{\kappa^{(4)}_{ik}}{3!n^2}\frac{\partial^3 (G_{kj}G_{ii}G_{jj})}{\partial H_{ik}^3}+O(n^{-3/2}).
			\end{aligned}
		\end{equation}
		
		Then the problem becomes to derive bounds for the terms 
		\begin{align*}
			&\frac{Q^{(3)}_{lm}}{n^{5/2}}\E\frac{1}{z}\sum_{i\in C_l,j\in C_m}\sum_{k}\frac{\kappa^{(2)}_{ik}}{n}\frac{\partial (G_{kj}G_{ii}G_{jj})}{\partial H_{ik}}\\ 
			=&\frac{Q^{(3)}_{lm}}{n^{7/2}}\E\frac{1}{z}\sum_{i\in C_l,j\in C_m}\sum_{k_1=1}^K\sum_{k\in C_{k_1}}\kappa_{ik}^{(2)}[-(G_{ik}G_{jk}+G_{ij}G_{kk})G_{ii}G_{jj}\\ 
			&-2G_{kj}G_{ii}G_{ik}G_{jj}-2G_{kj}G_{ii}G_{ij}G_{kj}]\\ 
%			=&-\frac{Q^{(3)}_{lm}}{n^{7/2}}\E\frac{1}{z}\sum_{i\in C_l,j\in C_m}\sum_{k_1=1}^KQ^{(2)}_{lk_1}[3(GT_{k_1}G)_{ij}G_{ii}G_{jj}+G_{ij}Tr(T_{k_1}G)G_{ii}G_{jj}\\ 
%			&+2(GT_{k_1}G)_{jj}G_{ii}G_{ij}]\\ 
			=&-\frac{Q^{(3)}_{lm}}{n^{7/2}}\E\frac{1}{z}\sum_{i\in C_l,j\in C_m}\sum_{k_1=1}^KQ^{(2)}_{lk_1}Tr(T_{k_1}G)G_{ij}G_{ii}G_{jj}+O(\frac{K}{n^{3/2}})\\ 
			=&-\frac{Q^{(3)}_{lm}}{n^{7/2}z}\sum_{k_1=1}^KQ^{(2)}_{lk_1}\E Tr(T_{k_1}G)\E\sum_{i\in C_l,j\in C_m} G_{ij}G_{ii}G_{jj}+O(\frac{Kn^{\varepsilon}}{n^{3/2}}).\\ 
		\end{align*}

		In the meantime,
		\begin{align*}
			&\frac{Q^{(3)}_{lm}}{n^{5/2}}\E\frac{1}{z}\sum_{i\in C_l,j\in C_m}\sum_{k}\frac{\kappa_{ik}^{(3)}}{n^{3/2}}\frac{\partial^2 (G_{kj}G_{ii}G_{jj})}{\partial H_{ik}^2}\\ 
			=&\sum\{\text{at least two of }i,j,k\text{ would be the same}\} +\sum\{i,j,k\text{ are mutually different}  \}\\ 
%			=&O(\frac{1}{n^{5/2}}\frac{1}{n^{3/2}}n^2)+O(\frac{1}{n^{5/2}}\frac{1}{n^{3/2}}n^3\frac{n^{\varepsilon}}{\sqrt{n}})\\ 
			=&O(\frac{n^{\varepsilon}}{n^{3/2}}),
		\end{align*}
		while the trivial upper bound is already sufficient for the $4$-th order term  
		\begin{align*}
			&\frac{Q^{(3)}_{lm}}{n^{5/2}}\E\frac{1}{z}\sum_{i\in C_l,j\in C_m}\sum_{k}\frac{\kappa^{(4)}_{ik}}{n^{2}}\frac{\partial^3 (G_{kj}G_{ii}G_{jj})}{\partial H_{ik}^3}=O(n^{-3/2}).\\ 
		\end{align*}
		
		Thus, from  above estimations we know 
		\begin{equation}\label{eq:cumulantreproduce}
			\begin{aligned}
				&\frac{1}{n}\E \sum_{i\in C_l,j\in C_m}\frac{\kappa^{(3)}_{ij}}{n^{3/2}}G_{ij}G_{ii}G_{jj}\\ 
%				=&-\frac{Q^{(3)}_{lm}}{n^{7/2}z}\sum_{k_1=1}^KQ^{(2)}_{lk_1}\E Tr(T_{k_1}G)\E\sum_{i\in C_l,j\in C_m} G_{ij}G_{ii}G_{jj}+O(\frac{Kn^{\varepsilon}}{n^{3/2}})\\ 
				=& -\frac{Q^{(3)}_{lm}}{n^{5/2}z}\sum_{k_1=1}^KQ^{(2)}_{lk_1}\E\frac{ Tr(T_{k_1}G)}{n}\E\sum_{i\in C_l,j\in C_m} G_{ij}G_{ii}G_{jj}+O(\frac{Kn^{\varepsilon}}{n^{3/2}}),
			\end{aligned}
		\end{equation}
		instantly we come to the conclusion that 
		\begin{equation}
			\E\sum_{i\in C_l, j\in C_m}G_{ij}G_{ii}G_{jj}=O(Kn^{1+\varepsilon}).
		\end{equation}
		Thus, $	\frac{1}{n}\E\sum_{i\in C_l,j\in C_m}\frac{\kappa^{(3)}_{ij}}{n^{3/2}}G_{ij}G_{ii}G_{jj}$ is minor and  $I_{2,(2,0)}$ is also minor.
%		\begin{align*}
%			\frac{1}{n}\E\sum_{i\in C_l,j\in C_m}\frac{\kappa^{(3)}_{ij}}{n^{3/2}}G_{ij}G_{ii}G_{jj} =O(\frac{Kn^{\varepsilon}}{n^{3/2}})
%		\end{align*}

			\begin{align*}
				&I_{2,(1,1)}\\ 
%				=&\frac{1}{n}\E\sum_{i\in C_l,j}\frac{2\kappa_{ij}^{(3)}}{2!n^{3/2}}e_{i}'G(z_1)(e_{j}e_{i}'+e_{i}e_{j}')G(z_1)e_{j}\frac{1}{n}\sum_{k\in C_m}e_{k}'G(z_2)(e_{j}e_{i}'+e_{i}e_{j}')G(z_2)e_{k} \\ 
				=&\frac{1}{n}\E\sum_{i\in C_l,j}\frac{2\kappa_{ij}^{(3)}}{2!n^{3/2}}
				(G_{ij}^2(z_1)+G_{ii}(z_1)G_{jj}(z_1))\frac{1}{n}\sum_{k\in C_m}[G_{kj}(z_2)G_{ik}(z_2)+G_{ki}(z_2)G_{jk}(z_2)]\\ 
%				=&\frac{1}{n^2}\E\sum_{i\in C_l,j}\frac{2\kappa_{ij}^{(3)}}{2!n^{3/2}}
%				(G_{ij}^2(z_1)+G_{ii}(z_1)G_{jj}(z_1))2(GT_mG)_{ij}(z_2)\\ 
				=& \frac{1}{n^{7/2}}\sum_{k_1=1}^K\sum_{i\in C_l,j\in C_{k_1}}2Q^{(3)}_{lk_1}G_{ii}(z_1)G_{jj}(z_1)(GT_mG)_{ij}(z_2)+O(\frac{n^{\varepsilon}}{n^{5/2}}).
			\end{align*}
		
		Simply note that $|G_{ii}(z)|\le |\frac{1}{\Im z}|$, $||G(z)||\le |\frac{1}{\Im z}|$ and $||T_m||=1$, we have
		%			\begin{align*}
		%			&\frac{1}{n^{5/2}}\sum_{i\in C_l,j\in C_{k_1}}Q^{(3)}_{lk_1}G_{ii}(z_1)G_{jj}(z_1)(GT_mG)_{ij}(z_2)\\ 
		%			=&\frac{1}{n^{5/2}}\frac{1}{z}\sum_{i\in C_l,j\in C_{k_1}}Q^{(3)}_{lk_1}G_{ii}(z_1)G_{jj}(z_1)(\sum_kH_{ik}(GT_mG)_{kj}-(T_mG)_{ij})\\ 
		%			=&\frac{Q^{(3)}_{lk_1}}{n^{5/2}}\E\frac{1}{z}\sum_{i\in C_l,j\in C_{k_1}}\sum_{k}\frac{\kappa_{ik}^{(2)}}{n}\frac{\partial ((GT_mG)_{kj}G_{ii}G_{jj})}{\partial H_{ik}}\\ 
		%			&+\frac{Q^{(3)}_{lk_1}}{n^{5/2}}\E\frac{1}{z}\sum_{i\in C_l,j\in C_{k_1}}\sum_{k}\frac{\kappa^{(3)}_{ik}}{2!n^{3/2}}\frac{\partial^2 ((GT_mG)_{kj}G_{ii}G_{jj})}{\partial H_{ik}^2}\\ 
		%			&+\frac{Q^{(3)}_{lk_1}}{n^{5/2}}\E\frac{1}{z}\sum_{i\in C_l,j\in C_{k_1}}\sum_{k}\frac{\kappa_{ik}^{(4)}}{3!n^2}\frac{\partial^3 ((GT_mG)_{kj}G_{ii}G_{jj})}{\partial H_{ik}^2}+O(n^{-3/2}) \\ 
		%			=&\frac{Q^{(3)}_{lk_1}}{n^{5/2}}\E\frac{1}{z}\sum_{i\in C_l,j\in C_{k_1}}\sum_{k}\frac{\kappa_{ik}^{(2)}}{n}\frac{\partial ((GT_mG)_{kj}G_{ii}G_{jj})}{\partial H_{ik}}\\ 
		%			&+\frac{Q^{(3)}_{lk_1}}{n^{5/2}}\E\frac{1}{z}\sum_{i\in C_l,j\in C_{k_1}}\sum_{k}\frac{\kappa_{ik}^{(3)}}{2!n^{3/2}}\frac{\partial^2 ((GT_mG)_{kj}G_{ii}G_{jj})}{\partial H_{ik}^2}+O(n^{-3/2}).
		%		\end{align*}
		\begin{equation}
			\begin{aligned}
				&\E\frac{1}{n^{7/2}}\sum_{i\in C_l,j\in C_{k_1}}Q^{(3)}_{lk_1}G_{ii}(z_1)G_{jj}(z_1)(GT_mG)_{ij}(z_2)\\ 
				=&\E\frac{1}{n^{7/2}}\sum_{i\in C_l,j\in C_{k_1}}Q^{(3)}_{lk_1}diag(G(z_1))\times (GT_mG)(z_2)\times diag(G(z_1))\\ 
				=&O(n^{-5/2}).
			\end{aligned}
		\end{equation}

		Further,		\begin{align*}
			&I_{2,(0,2)}\\ 
%=\frac{1}{n}\E\sum_{i\in C_l,j}\frac{\kappa_{ij}^{(3)}}{n^{3/2}}G_{ij}(z_1)\frac{1}{n}\sum_{k\in C_m}e_k'G(z_2)(e_je_i'+e_ie_j')G(z_2)(e_je_i'+e_ie_j')G(z_2)e_k\\ 
			=&-\E\sum_{k_1=1}^K\sum_{i\in C_l,j\in C_{k_1}}\frac{\kappa_{lk_1}^{(3)}}{n^{7/2}}G_{ij}(z_1)\sum_{k\in C_m}(G_{kj}G_{ij}G_{ik}+G_{kj}G_{ii}G_{jk}+G_{ki}G_{jj}G_{ik}+G_{ki}G_{ji}G_{jk})\\ 
			=&-\E\sum_{k_1=1}^K\sum_{i\in C_l,j\in C_{k_1}}\frac{\kappa_{lk_1}^{(3)}}{n^{7/2}}G_{ij}(z_1)[(GT_mG)_{ij}G_{ij}+(GT_mG)_{jj}G_{ii}+(GT_mG)_{ii}G_{jj}\\ 
			&+(GT_mG)_{ji}G_{ji}](z_2)\\ 
%			=&{\color{black}O(\frac{1}{n^{7/2}}\times n \times 1)+O(\frac{1}{n^{7/2}}\times n^2 \times \frac{n^{\varepsilon}}{\sqrt{n}})}+{\color{black}O(\frac{1}{n^{7/2}}\times n \times 1)+O(\frac{1}{n^{7/2}}\times n^2 \times \frac{n^{\varepsilon}}{\sqrt{n}})}\\  
			=&O(\frac{n^{\varepsilon}}{n^{5/2}}),
		\end{align*}
		where we use the fact that 
			\begin{align*}
				&\sum_{i\in C_l,j\in C_{k_1}}\frac{Q^{(3)}_{lk_1}}{n^{7/2}}G_{ii}(GT_mG)_{ij}(GT_mG)_{jj}|= \frac{Q^{(3)}_{lk_1}}{n^{7/2}}(diag(T_lG)\times(GT_mG)\times diag(T_{k_1}GT_mG))\\ 
				=&O(\frac{1}{n^{5/2}}).
			\end{align*}

		%		{\color{red} anything wrong?}
		%		The problem lies in the bound for the term
		%		\begin{equation*}
		%		\begin{aligned}
		%		&-\E\sum_{k_1=1}^K\sum_{i\in C_l,j\in C_{k_1}}\frac{\kappa_{lk_1}^3}{n^{7/2}}G_{ij}(z_1)((GT_mG)_{jj}G_{ii}+(GT_mG)_{ii}G_{jj})\\ 
		%		=&-\E\sum_{k_1=1}^K\sum_{i\in C_l,j\in C_{k_1}}\frac{\kappa_{lk_1}^3}{n^{7/2}}G_{ij}(z_1)((GT_mG)_{jj}(G_{ii}-M_l)+(GT_mG)_{ii}(G_{jj}-M_{k_1}))\\ 
		%		&-\E\sum_{k_1=1}^K\sum_{i\in C_l,j\in C_{k_1}}\frac{\kappa_{lk_1}^3}{n^{7/2}}G_{ij}(z_1)((GT_mG)_{jj}M_l+(GT_mG)_{ii}M_{k_1})\\ 
		%		=&o(\cdot)-\E\sum_{k_1=1}^K\sum_{i\in C_l,j\in C_{k_1}}\frac{\kappa_{lk_1}^3}{n^{7/2}}G_{ij}(z_1)((GT_mG)_{jj}M_l+(GT_mG)_{ii}M_{k_1}) 
		%		\end{aligned}
		%		\end{equation*}
		%		
		%		{\color{red} need again a system of equation argument? }
		%		
		%		
			\begin{align*} 
%				=&\frac{1}{n}\E\sum_{i\in C_l,j}\frac{\kappa_{ij}^{(4)}}{n^2}e_i'G(z_1)(e_je_i'+e_ie_j')G(z_1)(e_je_i'+e_ie_j')G(z_1)(e_je_i'+e_ie_j')G(z_1)e_j\langle\underline{ T_mG(z_2)}\rangle\\ 
			I_{2,(3,0)}	=&\frac{1}{n}\E\sum_{k_1=1}^K\sum_{i\in C_l,j\in C_{k_1},i\ne j}\frac{Q^{(4)}_{lk_1}}{n^2}(G_{ij}^4+6G_{ij}^2G_{ii}G_{jj}+G_{ii}^2G_{jj}^2)\langle\underline{ T_mG(z_2)}\rangle\\ 
%				=&\frac{1}{n}\E\sum_{k_1=1}^K\sum_{i\in C_l,j\in C_{k_1},i\ne j}\frac{Q^{(4)}_{lk_1}}{n^2}(G_{ii}^2G_{jj}^2)\langle\underline{ T_mG(z_2)}\rangle
%				+O(\frac{1}{n}n^2\frac{n^{\varepsilon}}{n}\frac{1}{n^2}\frac{n^{\varepsilon}}{n})\\ 
				=&\sum_{k_1=1}^K\sum_{i\in C_l,j\in C_{k_1},i\ne j}\frac{Q^{(4)}_{lk_1}}{n^3}M_l^2M_{k_1}^2\E\langle\underline{ T_mG(z_2)}\rangle
				\\ &+\sum_{k_1=1}^K\sum_{i\in C_l,j\in C_{k_1},i\ne j}\frac{Q^{(4)}_{lk_1}}{n^3}\E(G_{ii}^2G_{jj}^2-M_l^2M_{k_1}^2)\langle\underline{ T_mG(z_2)}\rangle +O(\frac{n^{\varepsilon}}{n^3})\\ 
%				=&0+O(n^2\times\frac{1}{n^3}\times\frac{n^{\varepsilon}}{\sqrt{n}}\times \frac{n^{\varepsilon}}{n})+O(\frac{n^{\varepsilon}}{n^3})\\ 
				=&O(\frac{n^{2\varepsilon}}{n^{5/2}}).
			\end{align*}
			\begin{align*}
			I_{2,(2,1)}	=&-\frac{3}{n}\E\sum_{i\in C_l,j}\frac{2\kappa_{ij}^{(4)}}{3!n^2}[G_{ij}^3(z_1)+G_{ij}(z_1)G_{ii}(z_1)G_{jj}(z_1)]\\ 
				&\times\frac{3}{n}\sum_{k\in C_m}[G_{ki}(z_2)G_{jk}(z_2)+G_{kj}(z_2)G_{ik}(z_2)]\\ 
				=&-\frac{3}{n^2}\E\sum_{i\in C_l,j}\frac{2\kappa_{ij}^{(4)}}{n^2}[G_{ij}^3(z_1)+G_{ij}(z_1)G_{ii}(z_1)G_{jj}(z_1)]2(GT_mG)_{ij}(z_2)\\ 
				=&O(\frac{n^{\varepsilon}}{n^{5/2}}).
			\end{align*}
			\begin{align*}
		I_{2,(1,2)}		=	&-\frac{3}{n}\E\sum_{i\in C_l,j}\frac{2\kappa^{(4)}_{ij}}{3!n^2}[G_{ij}^2(z_1)+G_{ii}(z_1)G_{jj}(z_1)]\\ 
				&\times\frac{1}{n}[(GT_mG)_{ij}G_{ij}+(GT_mG)_{ii}G_{jj}+(GT_mG)_{ji}G_{ji}+(GT_mG)_{jj}G_{ii}](z_2)\\ 
%				=	&-\frac{3}{n}\E\sum_{i\in C_l,j}\frac{2\kappa_{ij}^{(4)}}{3!n^2}[G_{ii}(z_1)G_{jj}(z_1)]\frac{1}{n}[(GT_mG)_{ii}G_{jj}+(GT_mG)_{jj}G_{ii}](z_2)+O(\frac{n^{\varepsilon}}{n^{5/2}}) \\ 
%				=	&-\E\sum_{k=1}^K\sum_{i\in C_l,j\in C_k}\frac{6Q^{(4)}_{lk}}{3!n^4}[G_{ii}(z_1)G_{jj}(z_1)][(GT_mG)_{ii}G_{jj}+(GT_mG)_{jj}G_{ii}](z_2)+O(\frac{n^{\varepsilon}}{n^{5/2}}) \\ 
				=&-\E\sum_{k=1}^K\frac{Q^{(4)}_{lk}}{n^2}[\alpha_kM_l(z_1)M_k(z_1)\ul{(GT_mGT_l)(z_2)}M_k(z_2)\\ 
				&+\alpha_lM_l(z_1)M_k(z_1)\ul{(GT_mGT_k)(z_2)}M_l(z_2)]+O(\frac{n^{\varepsilon}}{n^{5/2}}). \\ 
			\end{align*}
		\begin{align*}
			&I_{2,(0,3)}\\
%			=&-\frac{1}{n}\E\sum_{i\in C_l,j}\frac{\kappa_{ij}^{(4)}}{n^2}G_{ij}(z_1)\\ 
%			&\times\frac{1}{n}\sum_{k\in C_m}e_k'G(z_2)(e_je_i'+e_ie_j')G(z_2)(e_je_i'+e_ie_j')G(z_2)(e_je_i'+e_ie_j')G(z_2)e_k\\ 
%			=&-\E\sum_{i\in C_l,j}\frac{\kappa^{(4)}_{ij}}{n^4}G_{ij}(z_1)\sum_{k\in C_m}[G_{kj}G_{ij}G_{ij}G_{jk}+G_{kj}G_{ij}G_{ii}G_{jk}+G_{kj}G_{ii}G_{jj}G_{ik}+G_{kj}G_{ii}\\ 
%			&\times G_{ji}G_{jk}+G_{ki}G_{jj}G_{ij}G_{ik}+G_{ki}G_{jj}G_{ii}G_{jk}+G_{ki}G_{ji}G_{jj}G_{ik}+G_{ki}G_{ji}G_{ji}G_{jk}](z_2)\\ 
			=&-\frac{1}{n}\E\sum_{i\in C_l,j}\frac{\kappa_{ij}^{(4)}}{n^2}G_{ij}(z_1)\frac{1}{n}[G_{ij}^2(G^2)_{jj}+G_{ij}^2(G^2)_{jj}+G_{ii}G_{jj}(G^2)_{ij}+G_{ii}G_{ji}(G^2)_{jj}\\ &+G_{jj}G_{ij}(G^2)_{ii}+G_{jj}G_{ii}(G^2)_{ji}+G_{ji}G_{jj}(G^2)_{ii}+G_{ji}^2(G^2)_{ji}](z_2)\\ 
%			=&O(\frac{1}{n}\times n^2\times \frac{1}{n^2}\times\frac{n^{\varepsilon}}{\sqrt{n}}\times\frac{1}{n}\times 1)\\ 
			=&O(\frac{n^{\varepsilon}}{n^{5/2}}).
		\end{align*}
		
		Similarly, we can derive that $\sum_{a+b=4}^5I_{2,(a,b)}$ is also minor. Further note that
		\begin{align*}
			\varepsilon_{I_2}=\frac{1}{n}\sum_{i,j}C\sup_t|f^{(5+1)}(t)|E[|\xi|^{5+2}]=O(\frac{1}{n^{5/2}}).
		\end{align*}
		
	To conclude, the covariance function $\{Cov_{lm}(z_1,z_2)\}_{l,m\in [K]}$ satisfies the following system of equations to order 1
		\begin{equation}\label{eq:covlm}
			\begin{aligned}
				&z_1Cov_{lm}(z_1,z_2)\\ 
				=&-\sum_{k_1=1}^KQ^{(2)}_{lk_1}\alpha_{k_1}M_{k_1}(z_1)Cov_{lm}(z_1,z_2)-\sum_{k_1=1}^KQ^{(2)}_{lk_1}\alpha_lM_l(z_1)Cov_{k_1m}(z_1,z_2)\\ 
				&-\E\sum_{k_1=1}^K{2Q^{(2)}_{lk_1}}\ul{G(z_1)T_lG(z_2)T_mG(z_2)T_{k_1}}+{2Q^{(2)}_{ll}}M_{l}(z_1)\E \ul{(T_lGT_mG)(z_2)}\\ 
				&-\sum_{k=1}^K{Q^{(4)}_{lk}}[\alpha_kM_l(z_1)M_k(z_1)\E\ul{(GT_mGT_l)(z_2)}M_k(z_2)\\ 
				&\ +\alpha_lM_l(z_1)M_k(z_1)\E\ul{(GT_mGT_k)(z_2)}M_l(z_2)].
			\end{aligned}
		\end{equation}
		
	\end{proof}
	Now we have derived the system of equations for $\{Cov_{lm}(z_1,z_2)\}_{K\times K}$. However, several questions still need to be answered. First we will show that terms with the form $\E [\langle\underline{T_lGT_mG}\rangle \langle\underline{T_kGT_jG}\rangle$ would be minor in Section \ref{sec:boundtgtgvar}. Second, we will establish a system of equations  for  $\E \ul{(T_lG(z_1)T_mG(z_2)T_rG(z_2))}$  in Section \ref{sec:tgtgtg}.

	\subsection{Bound for $\E|\langle\underline{T_lGT_mG} \rangle|^2$}\label{sec:boundtgtgvar}
	
	Now we show that is  $\E|\langle\underline{T_lGT_mG} \rangle|^2$ of minor order for any $l,m\in[K]$. We start from the trivial bound that $\E|\langle\underline{T_lGT_mG} \rangle|^2=O(1)$. Again, we apply the cumulant expansion formula to $\E\langle \underline{T_lGT_mG}\rangle\langle\underline{T_rG^*T_sG^*}\rangle$.
	\begin{align*}
		&\E\langle \underline{T_lGT_mG}\rangle\langle\underline{T_rG^*T_sG^*}\rangle\\ 
%		=&\E\langle \underline{GT_lGT_m}\rangle\langle\underline{T_rG^*T_sG^*}\rangle\\ 
%		=&\frac{1}{z}\E \underline{(HG-I)T_lGT_m}\langle\underline{T_rG^*T_sG^*}\rangle\\
%		=&\frac{1}{z}\E \underline{HGT_lGT_m}\langle\underline{T_rG^*T_sG^*}\rangle-\frac{1}{z}\delta_{lm}\E \underline{T_lG}\langle\underline{T_rG^*T_sG^*}\rangle\\ 
		=&\frac{1}{n}\frac{1}{z}\E \sum_{ij}H_{ij}(GT_lGT_m)_{ji}\langle\underline{T_rG^*T_sG^*}\rangle-\frac{1}{z}\delta_{lm}\E \underline{T_lG}\langle\underline{T_rG^*T_sG^*}\rangle\\ 
		=&\E\frac{1}{nz}\sum_{ij}\frac{\kappa^{(2)}_{ij}}{n}\frac{\partial e_j'(GT_lGT_m)e_i}{\partial H_{ij}}\langle\underline{T_rG^*T_sG^*}\rangle+ \E\frac{1}{nz}\sum_{ij}\frac{\kappa^{(2)}_{ij}}{n}(GT_lGT_m)_{ji}\frac{\partial \underline{T_rG^*T_sG^*}}{\partial H_{ij}}\\ 
		&+\E\frac{1}{nz}\sum_{ij}\frac{\kappa^{(3)}_{ij}}{2!n^{3/2}}\frac{\partial^2 e_j'(GT_lGT_m)e_i}{\partial H_{ij}^2}\langle\underline{T_rG^*T_sG^*}\rangle+\E\frac{1}{nz}\sum_{ij}\frac{\kappa^{(3)}_{ij}}{2!n^{3/2}}(GT_lGT_m)_{ji}\frac{\partial^2 \underline{T_rG^*T_sG^*}}{\partial H_{ij}^2}\\ 
		&+\E\frac{1}{nz}\sum_{ij}\frac{\kappa^{(3)}_{ij}}{2!n^{3/2}}2\frac{\partial e_j'(GT_lGT_m)e_i}{\partial H_{ij}}\frac{\partial \underline{T_rG^*T_sG^*}}{\partial H_{ij}}+o(\frac{1}{n^2})\\ 
		&+\E\frac{1}{nz}\sum_{ij}\frac{\kappa^{(4)}_{ij}}{3!n^2}\frac{\partial^3 e_j'(GT_lGT_m)e_i}{\partial H_{ij}^3}\langle\underline{T_rG^*T_sG^*}\rangle+\E\frac{1}{nz}\sum_{ij}\frac{\kappa^{(4)}_{ij}}{3!n^2}3\frac{\partial^2 e_j'(GT_lGT_m)e_i}{\partial H_{ij}^2} \frac{\partial \underline{T_rG^*T_sG^*}}{\partial H_{ij}}     \\ 
		&+\E\frac{1}{nz}\sum_{ij}\frac{\kappa^{(4)}_{ij}}{3!n^2}3\frac{\partial e_j'(GT_lGT_m)e_i}{\partial H_{ij}} \frac{\partial^2 \underline{T_rG^*T_sG^*}}{\partial H_{ij}^2}   +\E\frac{1}{nz}\sum_{ij}\frac{\kappa^{(4)}_{ij}}{3!n^{2}}(GT_lGT_m)_{ji}\frac{\partial^3 \underline{T_rG^*T_sG^*}}{\partial H_{ij}^3}\\
		=&\E\frac{1}{nz}\sum_{ij}\frac{\kappa^{(2)}_{ij}}{n}[G_{ji}(GT_lGT_m)_{ji}+G_{jj}(GT_lGT_m)_{ii}]\langle\underline{T_rG^*T_sG^*}\rangle\\ &+\E\frac{1}{nz}\sum_{ij}\frac{\kappa^{(2)}_{ij}}{n}[(GT_lG)_{ji}(GT_m)_{ji}+(GT_lG)_{jj}(GT_m)_{ii}]\langle\underline{T_rG^*T_sG^*} \rangle+O(\frac{1}{n})\\ 
		&+ \frac{1}{nz}\E\sum_{ij}\frac{\kappa^{(2)}_{ij}}{n}(GT_lGT_m)_{ji}\frac{1}{n}\sum_{k=1}^N [(T_rG^*)_{ki}(G^*T_sG^*)_{jk}+(T_rG^*)_{kj}(G^*T_sG^*)_{ik}]\\ 
		&+\frac{1}{nz}\E\sum_{ij}\frac{\kappa^{(2)}_{ij}}{n}(GT_lGT_m)_{ji}\frac{1}{n}\sum_{k=1}^N [(T_rG^*T_sG^*)_{ki}G^*_{jk}+(T_rG^*T_sG^*)_{kj}G^*_{ik}]\\ 
		=&\E\frac{1}{z}\sum_{k_1=1,k_2=1}^KQ^{(2)}_{k_1k_2}\ul{T_{k_2}G}\ \ul{GT_lGT_mT_{k_1}}\langle\underline{T_rG^*T_sG^*}\rangle\\ 
		&+\E\frac{1}{z}\sum_{k_1=1,k_2=1}^KQ^{(2)}_{k_1k_2}\ul{T_mG}\ \ul{GT_lGT_{k_2}}\langle\underline{T_rG^*T_sG^*}\rangle+O(\frac{1}{n})\\ 
		=&\E\frac{1}{z}\sum_{k_2=1}^KQ^{(2)}_{mk_2}\alpha_{k_2}M_{k_2}\frac{1}{n} \langle \underline{GT_lGT_m}\rangle \langle\underline{T_rG^*T_sG^*}\rangle\\ 
		&+\frac{1}{z}\sum_{k_2=1}^KQ^{(2)}_{mk_2}\alpha_mM_m\langle \underline{GT_lGT_{k_2}}\rangle\langle\underline{T_rG^*T_sG^*} \rangle +O(\frac{1}{n}).
	\end{align*}

	Fix $l,r,s,$ then we can see that we will have the system of equation for vector $\mathbf{R}=[R_1,\cdots,R_K]^{\top}$ 
	$$Co_1*\mathbf{R}=[O(\frac{1}{n}),\cdots,O(\frac{1}{n}),\frac{1}{z}\E \underline{T_lG}\langle \underline{T_rG^*T_sG^*}\rangle+O(\frac{1}{n}),O(\frac{1}{n}),\cdots,O(\frac{1}{n})]^{\top},$$
	where $$R_m=\E \langle\underline{GT_lGT_m}\rangle  \langle\underline{G^*T_rG^*T_s}\rangle, m\in[K].$$
	
	For simplicity of illustration, we will not distinguish $\langle\underline{GT_lGT_m}\rangle $ from $\langle \underline{T_rG^*T_sG^*}\rangle$ below. Note further that $||Co_1||=O(1)$ and $Co_1$ is non-degenerate. Our bound for $\E|\langle\underline{T_lGT_mG} \rangle|^2$ then essentially comes down to the two parts,  $\frac{1}{z}\E \underline{T_lG}\langle \underline{T_rG^*T_sG^*}\rangle$ and an $O(\frac{1}{n})$ which comes from the terms  $\E\frac{1}{nz}\sum_{ij}\frac{\kappa^{(3)}_{ij}}{2!n^{3/2}}\frac{\partial^2 e_j'(GT_lGT_m)e_i}{\partial H_{ij}^2}\langle\underline{T_rG^*T_sG^*}\rangle$  and $\E\frac{1}{nz}\sum_{ij}\frac{\kappa^{(4)}_{ij}}{3!n^2}\frac{\partial^3 e_j'(GT_lGT_m)e_i}{\partial H_{ij}^3}\langle\underline{T_rG^*T_sG^*}\rangle$, while the reminder part of the above expression contributes only with an order $O(\frac{1}{n^2})$.
	
	Note that
	$$|\E \underline{T_lG}\langle \underline{T_rG^*T_sG^*}\rangle|^2=|\E \langle\underline{T_lG}\rangle\langle \underline{T_rG^*T_sG^*}\rangle|^2\le \E |\langle\underline{T_lG}\rangle|^2\E|\langle \underline{T_rG^*T_sG^*}\rangle|^2 , $$
	by (\ref{eq:isotropic}) we instantly know that $\E |\langle\underline{T_lG}\rangle|^2\E|\langle \underline{T_rG^*T_sG^*}\rangle|^2 \le \E |\langle\underline{T_lG}\rangle|^2=O(\frac{n^{\varepsilon}}{n^2})$. Thus,   $\E|\langle\underline{T_lGT_mG} \rangle|^2=O(\frac{n^{\varepsilon}}{n})$. Note that once we obtain this bound for $\E|\langle\underline{T_lGT_mG} \rangle|^2$, we know that they are minor, then we may claim that the system of equations \eqref{eq:covlm} is an order 1 matrix equation, the solution should be also of order 1. Now the bound for  $\E |\langle\underline{T_lG}\rangle|^2$ is improved to $O(\frac{1}{n^2})$ and $\E|\langle\underline{T_lGT_mG} \rangle|^2=O(\frac{1}{n})$.

	Then  we know  that $\frac{1}{z}\E \underline{T_lG}\langle \underline{T_rG^*T_sG^*}\rangle$ will contribute to the bound via an  $O(\frac{1}{n})$. Now suppose that $\E|\langle\underline{T_lGT_mG} \rangle|^2=O(\frac{1}{n^t})$ for some $t\ge 0$. Then we know that $\frac{1}{z}\E \underline{T_lG}\langle \underline{T_rG^*T_sG^*}\rangle$ will then contribute to the bound via an $O(\frac{1}{n^{1+\frac{t}{2}}})$. And we may repeat the above process as long as we can establish the bounds simultaneously for $\E\frac{1}{nz}\sum_{ij}\frac{\kappa^{(3)}_{ij}}{2!n^{3/2}}\frac{\partial^2 e_j'(GT_lGT_m)e_i}{\partial H_{ij}^2}\langle\underline{T_rG^*T_sG^*}\rangle$ and $\E\frac{1}{nz}\sum_{ij}\frac{\kappa^{(4)}_{ij}}{3!n^2}\frac{\partial^3 e_j'(GT_lGT_m)e_i}{\partial H_{ij}^3}\langle\underline{T_rG^*T_sG^*}\rangle$, which is totally applicable since they also share a similar structure of the form $\E O(\frac{1}{n})\langle\underline{T_rG^*T_sG^*}\rangle$. Thus, we may improve the bound from $O(\frac{1}{n^t})$ to $O(\frac{1}{n^{1+\frac{t}{2}}})$ over and over again a sequence of improved bounds $O(\frac{1}{n^{\frac{3}{2}}})$, $O(\frac{1}{n^{\frac{7}{4}}})$, $O(\frac{1}{n^{\frac{15}{8}}})$, $\cdots$, until we reach the limit $O(\frac{1}{n^2})$.
	
	Also note that the above argument only yields a bound $\E|\langle\underline{T_lGT_mG} \rangle|^2=O(\frac{n^{\delta}}{n^2})$, $\delta>0$. However, we can further derive system of equations for all the $\E O(\frac{1}{n})\langle\underline{T_rG^*T_sG^*}\rangle$ terms above individually and establish bounds individually, since now $\langle\underline{T_rG^*T_sG^*}\rangle$ would either break after differentiation or remain the way they are to reproduce terms of the form $\E O(\frac{1}{n})\langle\underline{T_rG^*T_sG^*}\rangle$. However, note that  $\E|\langle\underline{T_lGT_mG} \rangle|^2=O(\frac{n^{\delta}}{n^2})$ will cease to produce higher-order structures.  We then may obtain compact bounds for those $\E O(\frac{1}{n})\langle\underline{T_rG^*T_sG^*}\rangle$  again via the system of equations approach. The whole process is repetitive and tedious, thus omitted. Then we can conclude that those $\E O(\frac{1}{n})\langle\underline{T_rG^*T_sG^*}\rangle$ above are of order $O(\frac{1}{n^2})$ and further conclude that  $\E|\langle\underline{T_lGT_mG} \rangle|^2=O(\frac{1}{n^2})$ can be achieved.

	\subsection{System of equations for $\E \ul{T_lG(z_1)T_mG(z_2)T_rG(z_2)}$}
	\label{sec:tgtgtg}
	
	Now we want to derive the system of equations that  $\{\E\underline{T_lG(z_1)T_mG(z_2)T_rG(z_2)}\}_{l,m,r=1}^K$ satisfies to order 1. Easy to observe that the contribution of higher-order cumulants would vanish and we have

	\begin{align*}
			&z_1\E\underline{G(z_1)T_lG(z_2)T_mG(z_2)T_r}=\E\underline{(HG(z_1)-I)T_lG(z_2)T_mG(z_2)T_r}\\ 
%			=&\E\underline{HG(z_1)T_lG(z_2)T_mG(z_2)T_r}-\E\underline{T_lG(z_2)T_mG(z_2)T_r}\\ 
			=&\E\underline{HG(z_1)T_lG(z_2)T_mG(z_2)T_r}-\delta_{rl}\E\underline{G(z_2)T_lG(z_2)T_m},\\ 
		\end{align*}
	where \small
	\begin{align*}
		&\E\underline{HG(z_1)T_lG(z_2)T_mG(z_2)T_r}=\frac{1}{n}\E\sum_{i,j}H_{ij}(G(z_1)T_lG(z_2)T_mG(z_2)T_r)_{ji}\\ 
%		=&\frac{1}{n}\E\sum_{i,j}\frac{\kappa_{ij}^{(2)}}{n}e_{j}'\frac{\partial (G(z_1)T_lG(z_2)T_mG(z_2)T_r)}{\partial H_{ij}}e_{i}+\frac{1}{n}\E\sum_{i,j}\frac{\kappa_{ij}^{(3)}}{2!n^{3/2}}e_{j}'\frac{\partial^2 (G(z_1)T_lG(z_2)T_mG(z_2)T_r)}{\partial^2 H_{ij}}e_{i}\\  
%		&+\frac{1}{n}\E\sum_{i,j}\frac{\kappa_{ij}^{(4)}}{3!n^2}e_{j}'\frac{\partial^3 (G(z_1)T_lG(z_2)T_mG(z_2)T_r)}{\partial H_{ij}^3}e_{i}+O(\frac{1}{\sqrt{n}})\\ 
		=&-\frac{1}{n}\E\sum_{i,j}\frac{\kappa_{ij}^{(2)}}{n}[G_{ji}(z_1)(G(z_1)T_1G(z_2)T_mG(z_2)T_r)_{ji}+G_{jj}(z_1)(G(z_1)T_lG(z_2)T_mG(z_2)T_r)_{ii}]\\ 
		&-\frac{1}{n}\E\sum_{i,j}\frac{\kappa_{ij}^{(2)}}{n}[(G(z_1)T_lG(z_2))_{ji}(G(z_2)T_mG(z_2)T_r)_{ji}+(G(z_1)T_lG(z_2))_{jj}(G(z_2)T_mG(z_2)T_r)_{ii}]\\ 
		&-\frac{1}{n}\E\sum_{i,j}\frac{\kappa_{ij}^{(2)}}{n}[(G(z_1)T_lG(z_2)T_mG(z_2))_{ji}(G(z_2)T_r)_{ji}+(G(z_1)T_lG(z_2)T_mG(z_2))_{jj}(G(z_2)T_r)_{ii}]\\ 
		&+O(\frac{n^{\varepsilon}}{\sqrt{n}})\\ 
		%		=&-\frac{1}{n}\E\sum_{k_1,k_2=1}^K\sum_{i\in C_{k_1},j\in C_{k_2}}\frac{Q^{(2)}_{k_1k_2}}{n}[G_{ji}(z_1)(G(z_1)T_lG(z_2)T_mG(z_2)T_r)_{ji}+G_{jj}(z_1)(G(z_1)T_lG(z_2)T_mG(z_2)T_r)_{ii}]\\ 
		%		&-\frac{1}{n}\E\sum_{k_1,k_2=1}^K\sum_{i\in C_{k_1},j\in C_{k_2}}\frac{Q^{(2)}_{k_1k_2}}{n}[(G(z_1)T_lG(z_2))_{ji}(G(z_2)T_mG(z_2)T_r)_{ji}+(G(z_1)T_lG(z_2))_{jj}(G(z_2)T_mG(z_2)T_r)_{ii}]\\ 
		%		&-\frac{1}{n}\E\sum_{k_1,k_2=1}^K\sum_{i\in C_{k_1},j\in C_{k_2}}\frac{Q^{(2)}_{k_1k_2}}{n}[(G(z_1)T_lG(z_2)T_mG(z_2))_{ji}(G(z_2)T_r)_{ji}+(G(z_1)T_lG(z_2)T_mG(z_2))_{jj}(G(z_2)T_r)_{ii}]\\ 
		%		&+O(\frac{n^{\varepsilon}}{\sqrt{n}})+\frac{1}{n}\E\sum_{k_1=1}^K\sum_{i\in C_{k_1}}\frac{Q^{(2)}_{k_1k_1}}{n}[G_{ii}(z_1)(G(z_1)T_lG(z_2)T_mG(z_2)T_r)_{ii}+G_{ii}(z_1)(G(z_1)T_lG(z_2)T_mG(z_2)T_r)_{ii}]\\ 
		%		&+\frac{1}{n}\E\sum_{k_1=1}^K\sum_{i\in C_{k_1}}\frac{Q^{(2)}_{k_1k_1}}{n}[(G(z_1)T_lG(z_2))_{ii}(G(z_2)T_mG(z_2)T_r)_{ii}+(G(z_1)T_lG(z_2))_{ii}(G(z_2)T_mG(z_2)T_r)_{ii}]\\ 
		%		&+\frac{1}{n}\E\sum_{k_1=1}^K\sum_{i\in C_{k_1}}\frac{Q^{(2)}_{k_1k_1}}{n}[(G(z_1)T_lG(z_2)T_mG(z_2))_{ii}(G(z_2)T_r)_{ii}+(G(z_1)T_lG(z_2)T_mG(z_2))_{ii}(G(z_2)T_r)_{ii}]\\ 
		=&-\frac{1}{n^2}\E\sum_{k_1,k_2=1}^KQ^{(2)}_{k_1k_2}Tr(T_{k_2}G(z_1))Tr(G(z_1)T_lG(z_2)T_mG(z_2)T_rT_{k_1})\\ 
		&-\frac{1}{n^2}\E\sum_{k_1,k_2=1}^KQ^{(2)}_{k_1k_2}Tr(G(z_1)T_lG(z_2)T_{k_2})Tr(G(z_2)T_mG(z_2)T_rT_{k_1})\\ 
		&-\frac{1}{n^2}\E\sum_{k_1,k_2=1}^KQ^{(2)}_{k_1k_2}Tr(G(z_1)T_lG(z_2)T_mG(z_2)T_{k_2})Tr(G(z_2)T_rT_{k_1})+O(\frac{n^{\varepsilon}}{\sqrt{n}})\\ 
		%		&-\E\sum_{k_2=1}^K{Q^{(2)}_{rk_2}}\underline{(G(z_1)T_lG(z_2)T_mG(z_2)T_{k_2})}\ \underline{(G(z_2)T_r)}+O(\frac{n^{\varepsilon}}{\sqrt{n}})\\ 
		=&-\sum_{k_2=1}^KQ^{(2)}_{rk_2}\E\underline{(T_{k_2}G(z_1))}\ \E\underline{(G(z_1)T_lG(z_2)T_mG(z_2)T_r)} \\
		&-\sum_{k_2=1}^KQ^{(2)}_{rk_2}\E\underline{(G(z_1)T_lG(z_2)T_{k_2})}\ \E\underline{(G(z_2)T_mG(z_2)T_r)}\\ 
		&-\sum_{k_2=1}^K{Q^{(2)}_{rk_2}}\E\underline{(G(z_1)T_lG(z_2)T_mG(z_2)T_{k_2})}\ \E\underline{(G(z_2)T_r)}+O(\frac{n^{\varepsilon}}{\sqrt{n}}),
	\end{align*}
	\normalsize	 the last line follows from the fact that $\E|\langle\underline{ T_lGT_mG} \rangle|=O(\frac{n^{\varepsilon}}{n})$.
	
%	
%	Similarly, we can derive a similar matrix form of the equation. However, there is no  unified form for the slices of this 3-D tensor $[\E \ul{T_lG(z_1)T_mG(z_2)T_rG(z_2)}]_{l,m,r=1}^K$ at this level.
	
	\begin{remark}
		One may notice that if we fix $l$ and $m$, we will get the system of equations for $\{\E\underline{G(z_1)T_lG(z_2)T_mG(z_2)T_{k}}\}_{k=1}^K$
		with a fixed coefficient matrix
		$$Diag([-\frac{1}{M_1(z)},\cdots,-\frac{1}{M_K(z)}])+Q*diag([\alpha_1M_1,\cdots,\alpha_KM_K]).$$ It then becomes apparent that the coefficient matrix is actually universal regardless of  $l$ and $m$, which suggests that we can slice the  tensor $[\E \ul{G(z_1)T_lG(z_2)T_mG(z_2)T_r}]_{l,m,r=1}^K$ into $K$ matrices in which each  satisfies a matrix equation and we can actually do the slicing in any of the 3 directions. We believe that this phenomenon discloses certain supersymmetric patterns explored by the higher-order tensors  $[\E \ul{\Pi_{i=1}^r (G(z_{s_i})T_{t_i})}], s_i\in[q],t_i\in[K]$.
	\end{remark}

	Also, the system of equations for $[\E \ul{T_lG(z_1)T_mG(z_2)T_rG(z_2)}]_{l,m,r=1}^K$ induces another question. We still need to calculate $\{  \E \underline{G(z_1)T_lG(z_2)T_m}\}_{l,m=1}^K$.
	
	\subsection{System of equations for $\E \underline{G(z_1)T_lG(z_2)T_m}$}\label{sec:covend}
	
	\begin{lemma}
		The vector  $\mathbf{X}^{(l)}_{\mathbf{G1TG2T}}{(z_1,z_2)}$:=$ [\E \underline{G(z_1)T_lG(z_2)T_1},\cdots,\E \underline{G(z_1)T_lG(z_2)T_K}]^{\top}$ satisfies the following equation $$	Co_2{(z_1,z_2)}\mathbf{X}^{(l)}_{\mathbf{G1TG2T}}{(z_1,z_2)}=\mathbf{B}^{(l)}_{(z_1,z_2)},$$
		where 
		\begin{equation*}
			\mathbf{B}^{(l)}_{(z_1,z_2)}=  [0,\ldots,0,-\frac{\alpha_lM_l(z_2)}{z_1},0,\ldots,0]^T.
		\end{equation*}   
	\end{lemma}
	\begin{proof}
		Analog to the case of $\E \underline{G(z)T_lG(z)T_m}$, we derive the following system of equation
		\begin{equation*}
			\begin{aligned}
				&\E \underline{G(z_1)T_lG(z_2)T_m}\\ 
%				=&\frac{1}{z_1}\E \underline{HG(z_1)T_lG(z_2)T_m}-\delta_{lm}\frac{1}{z_1}\E \underline{T_lG(z_2)}\\ 
				=&\frac{1}{nz_1}\E\sum_{i,j}H_{ij}(G(z_1)T_lG(z_2)T_m)_{ji}-\delta_{lm}\frac{1}{z_1}\alpha_lM_l(z_2)+O(\frac{n^{\varepsilon}}{\sqrt{n}})\\ 
%				=&-\frac{1}{nz_1}\E \sum_{i,j}\frac{\kappa_{ij}^{(2)}}{n}G_{jj}(z_1)(G(z_1)T_lG(z_2)T_m)_{ii}-\frac{1}{nz_1}\sum_{ij}\frac{\kappa_{ij}^{(2)}}{n}(G(z_1)T_lG(z_2))_{jj}(G(z_2)T_m)_{ii}\\  &-\delta_{lm}\frac{1}{z_1}\alpha_lM_l(z_2)+O(\frac{n^{\varepsilon}}{\sqrt{n}})\\ 
				=&-\frac{1}{z_1}\sum_{k_1=1}^KQ^{(2)}_{mk_1}\alpha_{k_1}M_{k_2}(z_1)\E \underline{G(z_1)T_lG(z_2)T_m}\\ 
				&-\frac{1}{z_1}\sum_{k_1=1}^KQ^{(2)}_{mk_1}\alpha_mM_m(z_2)\E \underline{G(z_1)T_lG(z_2)T_{k_1}}
				-\delta_{lm}\frac{1}{z_1}\alpha_lM_l(z_2)+O(\frac{n^{\varepsilon}}{\sqrt{n}}).
			\end{aligned}
		\end{equation*} 
		Reorganizing the proof  with (\ref{eq:QVE})   yields the matrix form.
	\end{proof}

	Similar to $M_{GTGT}$, for  $M_{G1TG2T}(z_1,z_2):=\big( \E\underline{G(z_1)T_lG(z_2)T_m}\big)_{l,m=1}^K$, we have the following simplified form
	\begin{equation}
		M_{G1TG2T}(z_1,z_2)=-\big(Q^{(2)}-Diag ([\frac{1}{\alpha_1M_1(z_1)M_1(z_2)},\cdots,\frac{1}{\alpha_KM_K(z_1)M_K(z_2)}]^{\top})\big)^{-1},
	\end{equation}
	which, again, is in accordance with the fact that the matrix $M_{G1TG2T}(z_1,z_2)$ should be symmetric
	\begin{align*}
		Tr(G(z_1)T_lG(z_2)T_m)&=Tr(G(z_1)T_lG(z_2)T_m)^{\top}=Tr(T_mG(z_2)T_lG(z_1))\\ 
		&=Tr(G(z_1)T_mG(z_2)T_l).
	\end{align*}

	\subsection{Proof of normality}\label{sec:gaussianitynonhat}
	
	In this subsection and Section \ref{sec:gaussianityhat} only, we will use $i$ for the unit imaginary number $\sqrt{-1}$.
	
	To recover the covariance structure and prove normality, it's natural to adopt the following setting as in \cite{Khorunzhy1996} which can be viewed as variation of the Tikhomirov-Stein method. We will prove that for any integer $q$ and arbitrary
	collection $z_{1}, \ldots, z_{q}$ of complex numbers from $\mathbb{C}\backslash B_{\varepsilon_0}(\sigma(H))$, where $\epsilon_0$ is taken as in the Proof of Lemma \ref{eq:trtgtg} to ensure the uniqueness and existence of the solution. The joint probability distribution of random
	variables $(\langle G(z_1)\rangle, \ldots, \langle G(z_{q})\rangle)$ converges as $n \rightarrow \infty$ to the $q$-dimensional Gaussian distribution
	with zero mean and the covariance matrix $\left[g\left(z_{s}, z_{t}\right)\right]_{s, t=1}^{q}$ specified in the following section.

	\begin{proof}
		
		% roughly speaking, we wanna prove that a real random variable $R$ is a gaussian random variable with mean zero, it  suffices to prove that
		%			
		%			$$\E e^{it R}=e^{-\frac{1}{2} \sigma^{2} t^{2}}, $$
		%			further, we may instead choose to prove that 
		%			$it\E Re^{it R}=-\sigma t^2e^{-\frac{1}{2} \sigma^{2} t^{2}}$ along with certain boundary conditions. 
		Note that to prove that the process $Tr(G(z))$ converges to a Gaussian process, we need to prove the real and imaginary parts of $Tr(G(z))$ are jointly Gaussian in the limiting sense, so the first thing to do is to construct an adequate process. 
		
		Let $\gamma(z)=\gamma^{(n)}(z):=\Re \langle Tr G(z)\rangle$ and $\theta(z)=\theta^{(n)}(z)=\Im \langle Tr G(z)\rangle$, further
		$$
		\Psi(z, c)=\left\{\begin{array}{ll}{\gamma(z)} & {\text { if } c=\gamma} \\ {\theta(z)} & {\text { if } c=\theta}\end{array}\right.
		$$
		and
		$$
		(a(c), b(c))=\left\{\begin{array}{ll}{(1 / 2,1 / 2)} & {\text { if } c=\gamma} \\ {(1 / 2 i, 1 / 2 i)} & {\text { if } c=\theta},\end{array}\right.
		$$
		then $\E\{\Psi(z, c)\}=0$. And now we wanna prove that $\forall \text{ fixed }q\in \mathbb{Z}_+,\ \{z_{s}\}_{s=1}^q\in \{ \mathbb{C}\backslash B_{\varepsilon_0}(\sigma(\hat{H}))\}^q$, $ \{c_{s}\}_{s=1}^q \in\{\gamma, \theta\}^q$, the joint probability distribution of random variables
		$\Psi\left(z_{1}, c_{1}\right), \ldots, \Psi\left(z_{q}, c_{q}\right)$ is the $q$-dimensional Gaussian distribution with zero mean and covariance matrix
		$$
		\begin{aligned} \E \left\{\Psi\left(z_{s}, c_{s}\right) \Psi\left(z_{t}, c_{t}\right)\right\}=& a\left(c_{s}\right) a\left(c_{t}\right) g\left(z_{s}, z_{t}\right)+a\left(c_{s}\right) b\left(c_{t}\right) g\left(z_{s}, z_{t}^{*}\right)+\\ & a\left(c_{t}\right) b\left(c_{s}\right) g\left(z_{s}^{*}, z_{t}\right)+b\left(c_{s}\right) b\left(c_{t}\right) g\left(z_{s}^{*}, z_{t}^{*}\right),\ \end{aligned}
		$$
		where $g(z_s,z_t)$ should be in accordance with our covariance function $Cov(z_s,z_t)$ previously defined in Section \ref{sec:covlm}.
		Then we need to consider the characteristic function of these random variables $\Psi\left(z_{1}, c_{1}\right), \ldots, \Psi\left(z_{q}, c_{q}\right)$, which we
		shall write in the form
		\begin{align*} e_q=e_{q}^{(n)}\left(T_{q}, C_{q}, Z_{q}\right):=&\prod_{s=1}^{q} \exp \left\{i \tau_{s}\left[a\left(c_{s}\right) Tr\langle G(z_{s})\rangle+b\left(c_{s}\right) Tr\langle G(z_{s}^{*})\rangle\right]\} \right.,
		\end{align*}
		where  $T_{q}=\left(\tau_{1}, \ldots, \tau_{q}\right), C_{q}=\left(c_{1}, \ldots, c_{q}\right), Z_{q}=\left(z_{1}, \ldots, z_{q}\right)$.  
		For simplicity, we shall use $e_q$ for $e_q^{(n)}(T_q,C_q,Z_q)$ when there is no confusion. Also, we would use $a_s$ and $b_s$ to denote $a(c_s)$ and $b(c_s)$.
		Instantly, we have
		$$
		\frac{\partial}{\partial \tau_{s}} \E\left\{e_{q}\right\}=i \E\left\{e_{q}\left[a_s Tr \langle G(z_{s})\rangle+b_sTr \langle G(z_{s}^{*})\rangle\right]\right\},
		$$
		and our main goal is to show that there exist sequences  of coefficients of the covariance matrices $({\Sigma}_{s t}^{(n)})_{ s, t=1}^q$ s.t. for each fixed $T_{q}$
		$$\lim _{n \rightarrow \infty}\left|\E\left\{e_{q}^{(n)}\left[a_sTr\langle G(z_{s})\rangle+b_s Tr\langle G(z_{s}^{*})\rangle\right]\right\}-i \sum_{t=1}^{q} \tau_{s} \Sigma_{s t}^{(n)} \E\left\{e_{q}^{(n)}\right\}\right|=0, \ z \in \mathbb{C}\backslash\mathbb{R},$$
		and further, the limits of all these coefficients exist
		$$
		\Sigma_{s t}=\lim _{n \rightarrow \infty} \Sigma_{s t}^{(n)}
		$$
		and are in accordance with our previous results on the covariance function.
% Then we may claim via Tikhomirov method that 
%		the limiting characteristic function has the Gaussian form $\exp \left(-1 / 2 \sum_{s, t=1}^{q} \Sigma_{s t} \tau_{s} \tau_{t}\right)$  where $\Sigma_{st}$ matches with our covariance result. 
		
		First, we need  to calculate $\E\left\{e_{q}Tr \langle G(z)\rangle\right\}$. By resolvent identity and cumulant expansion, we have
		\begin{align*}
			&\E\left\{e_{q} \langle TrG(z)\rangle\right\}=\E\left\{\langle e_{q} \rangle Tr G(z)\right\}=\E\langle e_{q} \rangle \sum_{j=1}^n G_{ii}(z)=\frac{1}{z}\E\langle e_{q} \rangle \sum_{j=1}^n (HG)_{ii}(z)\\ 
%			=&z^{-1}\sum_{j, m=1}^{n}	(\sum_{d=1}^3\frac{\kappa^{(d+1)}_{mj}}{d!}\E[\frac{\partial^d \langle e_q\rangle G_{jm}}{\partial H_{mj}^d}]+\varepsilon_{I_3,mj}).\\ 
			=&z^{-1}\sum_{j, m=1}^{n}	(\sum_{a+b=1}^3\frac{\kappa^{(a+b+1)}_{mj}}{(a+b)!}\E[\frac{\partial^a \langle e_q\rangle \partial^bG_{jm}}{\partial H_{mj}^{a+b}}]+\varepsilon_{I_3,mj}).\\ 
			=&\sum_{a+b=1}^3I_{3,(a,b)}+\varepsilon_{I_3}.\\
		\end{align*}		
		%			Denote by E $_{m j}$ the conditional expectation $\mathrm{E}\left\{\cdot | W_{m j}=w\right\}$ and rewrite the right-hand
		%			side of $(\mathbb{N} .69)$ in the form
		%			$$
		%			z^{-1} n^{-1 / 2} \sum_{j, m=1}^{n} \int \mathrm{E}_{m j}\left\{e_{q}^{C} G_{j m}\right\} w d P_{m j}(w)
		%			$$
		%	
		
		Note  that higher-order expansion terms vanish, namely  $\varepsilon_{I_3}\le \sum_{mj}|\varepsilon_{I_3,mj}|=O(\frac{1}{\sqrt{n}})$, thus minor.

		We begin with \begin{align*}
			&zI_{3,(1,0)}=-\frac{1}{n}\E\sum_{j,m}\kappa^{(2)}_{jm}(G^2_{jm}+G_{jj}G_{mm})\langle e_q\rangle\\ 
			%			=&-\frac{1}{n}\E\langle\sum_{j,m}\kappa^{(2)}_{jm}(G^2_{jm}+G_{jj}G_{mm})\rangle e_q\\ 
			%			=&-\frac{1}{n}\E\langle\sum_{k,l=1}^K\sum_{j\in C_k,m\in C_k}\kappa^{(2)}_{jm}(G^2_{jm}+G_{jj}G_{mm})\rangle e_q\\ 
%			=&-\frac{1}{n}\E\langle\sum_{k,l=1}^KQ^{(2)}_{kl}(Tr(T_kGT_lG)+Tr(T_kG)Tr(T_lG))\rangle e_q+\frac{1}{n}\E\langle\sum_{k=1}^KQ^{(2)}_{kk}Tr(T_kG^2)\rangle e_q\\
			%			=&-\frac{1}{n}\E\langle\sum_{k,l=1}^KQ^{(2)}_{kl}Tr(T_kG)Tr(T_lG)\rangle e_q  +o(1)\\
			=&-\frac{1}{n}\E\sum_{k,l=1}^KQ^{(2)}_{kl}[Tr(T_kG)Tr(T_lG)-\E (Tr(T_kG)Tr(T_lG))]e_q\\
			=&-\frac{1}{n}\E\sum_{k,l=1}^KQ^{(2)}_{kl}\Big\{[Tr(T_kG)-\E (Tr(T_kG)][Tr(T_lG)-\E Tr(T_lG)]e_q+e_qTr(T_lG)\E [Tr(T_kG)]\\ 
			&+e_qTr(T_kG)\E [Tr(T_lG)]-e_q\E[Tr(T_lG)]\E [Tr(T_kG)]-e_q\E [Tr(T_kG)Tr(T_lG)]\Big\}\\	
			=&-\frac{1}{n}\E\sum_{k,l=1}^KQ^{(2)}_{kl}\Big\{[Tr(T_kG)-\E (Tr(T_kG)][Tr(T_lG)-\E Tr(T_lG)]e_q\\ &+e_q[Tr(T_lG)-\E Tr(T_lG)]\E [Tr(T_kG)]+e_q[Tr(T_kG)-\E Tr(T_kG)]\E [Tr(T_lG)]\\ 
			&+e_q\E[Tr(T_lG)]\E [Tr(T_kG)]-e_q\E [Tr(T_kG)Tr(T_lG)]\Big\}\\	
%			=&-\frac{1}{n}\E\sum_{k,l=1}^KQ^{(2)}_{kl}\Big\{[Tr(T_kG)-\E (Tr(T_kG)][Tr(T_lG)-\E Tr(T_lG)]e_q+e_qTr(T_lG)\E [Tr(T_kG)]\\ 
%			&+e_qTr(T_kG)\E [Tr(T_lG)]-e_q\E[Tr(T_lG)]\E [Tr(T_kG)]-e_q\E [Tr(T_kG)Tr(T_lG)]\Big\}\\	
			=&-\frac{1}{n}\E\sum_{k,l=1}^KQ^{(2)}_{kl}\Big\{e_q[Tr(T_lG)-\E Tr(T_lG)]\E [Tr(T_kG)]\\
			&+e_q[Tr(T_kG)-\E Tr(T_kG)]\E [Tr(T_lG)]\Big\}+o(1)\\
			=&-\sum_{k,l=1}^KQ_{kl}^{(2)}\Big \{\alpha_kM_k\E[\langle T_lG\rangle e_q]+\alpha_lM_l\E[\langle T_kG\rangle e_q] \Big\}	+o(1).		
		\end{align*}
		
		In other words, we observe a system of equations structure for $\{\E [\langle Tr T_lG\rangle e_q]\}_{l=1}^K$ here. Though we don't need to derive the system explicitly, we still need to compare the system of equations for $\{\E [\langle Tr T_lG\rangle e_q]\}_{l=1}^K$ with \eqref{eq:covlm}. $I_{3,(1,0)}$ above shows the matching between the coefficient parts. Later we will compare the constant parts of both systems.
		
		Before we proceed further, we need to calculate $\frac{\partial e_q}{\partial H_{mj}}$, 
		Note that	
		\begin{align*}
			\frac{\partial \langle e_q \rangle}{\partial H_{jm}}=\frac{\partial e_q }{\partial H_{jm}}=&\frac{\partial \left\{\prod_{s=1}^{q} \exp \left\{i \tau_{s}\left[a_s Tr\langle G(z_{s})\rangle+b_s Tr\langle G(z_{s}^{*})\rangle\right]\right\}\right.}{\partial H_{jm}}\\
%			&=e_q\frac{\partial\left[\sum_{s=1}^qi\tau_s (a_s\langle Tr G(z_s)\rangle+b_s\langle Tr G(z_s^*)\rangle)\right]}{\partial H_{jm}}\\ 
			&=-e_q\left[ \sum_{s=1}^qi\tau_s(2a_s(G^2)_{mj}(z_s)+2b_s(G^2)_{mj}(z_s^*)) \right].	
		\end{align*}	
		
		It easily follows that
		\begin{align*}
			I_{3,(0,1)}
			=&\E\sum_{j,m=1}^n\frac{\kappa_{jm}^{(2)}}{nz}G_{jm}\frac{\partial \langle e_q \rangle}{\partial H_{jm}}\\ 
%			=&\E\sum_{j,m=1}^n\frac{\kappa_{jm}^{(2)}}{nz}G_{jm}\frac{\partial \left\{\prod_{s=1}^{q} \exp \left\{i \tau_{s}\left[a\left(c_{s}\right) Tr \langle G(z_{s})\rangle+b\left(c_{s}\right) Tr \langle G(z_{s}^{*})\rangle\right]\right\}\right.}{\partial H_{jm}} \\ 
			=&-\frac{1}{nz}\E \sum_{j,m}\kappa^{(2)}_{jm}G_{jm}e_q\left[ \sum_{s=1}^qi\tau_s(2a_s(G^2)_{mj}(z_s)+2b_s(G^2)_{mj}(z_s^*)) \right]\\ 
			=&-\frac{1}{nz}\E\sum_{s=1}^q\sum_{k,l=1}^KQ^{(2)}_{kl}\left[i\tau_s2a_sTr(T_kG(z)T_lG^2(z_s))+i\tau_s2b_sTr(T_kG(z)T_lG^2(z_s^*))\right]\\ 
			&+\frac{1}{nz}\E e_q\sum_{s=1}^q(i\tau_s)\sum_{k=1}^KM_k(z)\left[2a_s Tr(T_kG^2(z_s))+2b_sTr(T_kG^2(z_s^*))\right].
		\end{align*}
		
		$	I_{3,(2,0)},	I_{3,(1,1)},	I_{3,(0,2)}, 	I_{3,(3,0)},	I_{3,(2,1)}, 	I_{3,(0,3)}$ are minor and the detailed calculations are omitted here.
		
		Further, note that
		\begin{align*}
			&\frac{\partial^2e_q}{\partial H_{jm}^2}=\frac{\partial}{\partial H_{jm}}\left\{-e_q[\sum_{s=1}^qi\tau_s[2a_s(G^2)_{mj}(z_s)+2b_s(G^2)_{mj}(z_s^*)]]\right\}\\ 
			=&e_q\Big[\sum_{s=1}^q(2a_s(G^2)_{mj}(z_s)+2b_s(G^2)_{mj}(z_s^*))\Big]^2\\ 
			&+e_q\Big\{ \sum_{s=1}^q i\tau_s[2a_s((G^2)_{mm}G_{jj}+(G^2)_{mj}G_{jm})(z_s)+2a_s(G_{mm}(G^2)_{jj}+G_{mj}(G^2)_{jm})(z_s^*)\\ 
			&+2b_s((G^2)_{mm}G_{jj}+(G^2)_{mj}G_{jm})(z_s)+2b_s(G_{mm}(G^2)_{jj}+G_{mj}(G^2)_{jm})(z_s^*)]\Big\}.
		\end{align*}
		
		Thus,
		\begin{align*}
			&I_{3,(1,2)}\\
%			=&-\frac{1}{n^2z}\sum_{j,m}\frac{\kappa^{(4)}_{jm}}{3!}3(G^2_{jm}+G_{jj}G_{mm})\frac{\partial^2\langle e_q\rangle }{\partial H_{jm}^2}\\ 
			=&-\frac{1}{n^2z}\sum_{j,m}\frac{\kappa^{(4)}_{jm}}{2}(G^2_{jm}+G_{jj}G_{mm})e_q\Big[ \sum_{s=1}^qi\tau_s[2a_s(G^2)_{mm}(z)G_{jj}(z_s)\\ 
			&+2a_sG_{mm}(z)(G^2)_{jj}(z_s^*)+2b_s(G^2)_{mm}(z)G_{jj}(z_s)+2b_sG_{mm}(z)(G^2)_{jj}(z_s^*)]  \Big]\\ 
			=&-\frac{1}{n^2z}\sum_{k,l=1}^KQ^{(4)}_{kl}M_k(z)M_l(z)e_q\Big[ \sum_{t=1}^qi\tau_t(a_tM_k(z_t)\alpha_kTr(G^2(z_t)T_l)\\ 
			&+a_tM_l(z_t^*)\alpha_lTr(T_kG^2(z_t^*)) 
			+b_tM_k(z_t)\alpha_k Tr(T_lG^2(z_t))+b_tM_l(z_t^*)\alpha_lTr(T_kG^2(z_t^*))) \Big].
		\end{align*}
		Then we may conclude that
		\begin{equation} \label{eq:guassiancov}
			\begin{aligned}
				&\E\left\{e_{q}\left[a\left(c_{s}\right) \langle G(z_{s})\rangle+b\left(c_{s}\right) \langle G(z_{s}^{*})\rangle\right]\right\}\\ 
			\end{aligned}
	\end{equation}		
    \begin{align*}
				=&-\frac{a_s}{z_s}\sum_{k,l=1}^KQ_{kl}^{(2)}\Big \{\alpha_kM_k(z_s)\E[\langle T_lG(z_s)\rangle e_q]+\alpha_lM_l(z_s)\E[\langle T_kG(z_s)\rangle e_q] \Big\}\\
				&-\frac{b_s}{z^*_s}\sum_{k,l=1}^KQ_{kl}^{(2)}\Big \{\alpha_kM_k(z^*_s)\E[\langle T_lG(z^*_s)\rangle e_q]+\alpha_lM_l(z^*_s)\E[\langle T_kG(z^*_s)\rangle e_q] \Big\}\\
				&-\frac{a_s}{nz_s}\E e_q\sum_{t=1}^q\sum_{k,l=1}^KQ^{(2)}_{kl}\Big[i\tau_t2a_tTr(T_kG(z_s)T_lG^2(z_t))+i\tau_t2b_tTr(T_kG(z_s)T_lG^2(z_t^*))\Big]\\ 
				&+\frac{a_s}{nz_s}\E e_q\sum_{k=1}^KQ^{(2)}_{kk}M_k(z)\sum_{t=1}^q[2a_tTr(T_kG^2(z_t))+2b_tTr(T_kG^2(z_t^*))]\\ 
				&-\frac{a_s}{nz_s^*}\E e_q\sum_{t=1}^q\sum_{k,l=1}^KQ^{(2)}_{kl}[i\tau_t2a_tTr(T_kG(z_s^*)T_lG^2(z_t))+i\tau_t2b_tTr(T_kG(z_s*))T_lG^2(z_t^*)]\\ 
				&+\frac{a_s}{nz_s^*}\E e_q\sum_{k=1}^KQ^{(2)}_{kk}M_k(z_s^*)\sum_{t=1}^q[2a_tTr(T_kG^2(z_t))+2b_tTr(T_kG^2(z_t^*))]\\ 
				&-\frac{a_s}{nz_s}\E e_q\sum_{k,l=1}^KQ^{(4)}_{kl}M_{k}(z_s)M_l(z_s)\sum_{t=1}^qi\tau_t[a_tM_k(z_t)\alpha_k Tr(T_lG^2(z_t))\\ 
				&+a_tM_l(z_t^*)\alpha_lTr(T_kG^2(z_t^*)) +b_tM_k(z_t)\alpha_k Tr(T_lG^2(z_t))+b_tM_l(z_t^*)\alpha_lTr(T_kG^2(z_t^*))
				]\\ 
				&-\frac{a_s}{nz_s^*}\E e_q\sum_{k,l=1}^KQ^{(4)}_{kl}M_k(z_s^*)M_l(z_s^*)\sum_{t=1}^qi\tau_t[a_tM_k(z_t)\alpha_kTr(T_lG^2(z_t))\\ 
				&+a_tM_l(z_t^*)\alpha_lTr(T_kG^2(z_t^*))+b_tM_k(z_t)\alpha_k Tr(T_lG^2(z_t))+b_tM_l(z_t^*)\alpha_lTr(T_kG^2(z_t^*))].
	\end{align*}	
		
		Compare the above formula with 
		
		$$i \sum_{t=1}^{q} \tau_{t} \Sigma_{s t} \E\{e_{q}\} $$
		where 
		\begin{align*} \Sigma{st}=\E \left\{X\left(z_{s}, c_{s}\right) X\left(z_{t}, c_{t}\right)\right\}=& a\left(c_{s}\right) a\left(c_{t}\right) Cov(z_s,z_t)+a\left(c_{s}\right) b\left(c_{t}\right) Cov(z_s,z^*_t)+\\ & a\left(c_{t}\right) b\left(c_{s}\right) Cov(z_s,z^*_t)+b\left(c_{s}\right) b\left(c_{t}\right) Cov(z^*_s,z^*_t).  \end{align*}

		Further, note that $a(c_s)=b(c_s), \forall s,$ then  by QVE \eqref{eq:QVE} and piece-wise comparison, one can see that the above \eqref{eq:guassiancov} and  \eqref{eq:covlm} indeed lead to the same covariance structure. And the existence of the limit follows from our previous discussion in Section \ref{sec:covlm}.

		\subsection{Tightness of the process $\langle Tr G(z)\rangle$}\label{sec:tightness}
		
		After we establish the finite dimensional convergence, it remains to show that the process $Tr\langle G(z)\rangle, z\in \mathbb{C}\backslash B_{\varepsilon_0}(\sigma(H))$ is tight.

		In other words, we will show that 
		\begin{equation}
			\E |Tr\langle G(z_1)\rangle-Tr \langle G(z_2)\rangle|^2=O(|z_1-z_2|^2).
		\end{equation}
		Simply note that \begin{equation}
			\E |Tr\langle G(z_1)\rangle-Tr \langle G(z_2)\rangle|^2=\E |Tr  \langle G(z_1)(z_1-z_2)G(z_2)\rangle|^2,
		\end{equation}
		so we only need to show that $ \E|Tr  \langle G(z_1)G(z_2)\rangle|^2=O(1)$. It suffices to show that $$ \E|Tr  \langle T_lG(z_1)T_mG(z_2)\rangle|^2=O(1), \forall l,m\in[K],$$ which has been proved in Section \ref{sec:boundtgtgvar}. Therefore tightness is established.

	\end{proof} 
	
	\section{Proof of Theorem \ref{thm:mainBlockWigneremp}} \label{sec:pfofhatversion}
	Similar to the proof of Theorem \ref{thm:mainBlockWigner}, we first derive the mean function  $\E Tr(\hat{H}-z)^{-1}$ in Section \ref{sec:hatmeanpf} and the covariance function $ Cov(Tr(\hat{G}(z_1)),Tr(\hat{G}(z_2)))$ in Section \ref{sec:covhat}. Then we discuss the normality and the  tightness for this data-driven renormalized case in Sections \ref{sec:gaussianityhat} and \ref{sec:tightness2}.
	
	\subsection{Mean function $\E Tr(\hat{H}-z)^{-1}$ }\label{sec:hatmeanpf}

	\par 
	By the fact that $\|\hat{H}-H\|=o_p(\frac{\log(n)}{\sqrt{n}})$ and the resolvent expansion formula.  Note also that $\|\hat{H}-H\|$ is essentially bounded, we have
	\begin{equation}
		\begin{aligned}
			&\E Tr(\hat{G}(z))=\E Tr(G(z))-\E Tr[G(z)(\hat{H}-H)G(z)]\\ 
			&+\E Tr[G(z)(\hat{H}-H)G(z)(\hat{H}-H)G(z)]+o(\frac{\log(n)^3}{\sqrt{n}}).
		\end{aligned}
	\end{equation}
	$\E Tr(G(z))$ has been investigated in the previous sections. So we only need to estimate $\E Tr[G(z)(\hat{H}-H)G(z)]$ and $\E Tr[G(z)(\hat{H}-H)G(z)(\hat{H}-H)G(z)]$.
	\begin{align*}
		&\E Tr[G(z)(\hat{H}-H)G(z)]=\E Tr(\hat{H}-H)G(z)^2=\E\sum_{i,j=1}^n(\hat{H}_{ij}-H_{ij})(G^2)_{ji}\\ 
		=&-\E \sum_{i,j=1}^n\sum\limits_{\substack{\alpha\in C_{\sigma(i)},\\ \beta\in C_{\sigma(j)}}}\frac{H_{\alpha\beta}}{N_{\sigma(i)\sigma(j)}}(G^2)_{ji}=-\E \sum_{k,l=1}^K\sum\limits_{\substack{i\in C_{k}\\ j\in C_l}}\sum\limits_{\substack{\alpha\in C_k\\ \beta\in C_l}}\frac{H_{\alpha\beta}}{N_{kl}}(G^2)_{ji}\\ 
		=&-\E \sum_{k,l=1}^K\frac{1}{N_{kl}}\sum\limits_{\substack{i\in C_{k}\\ j\in C_l}}\sum\limits_{\substack{\alpha\in C_k\\ \beta\in C_l}}\sum_{d=1}^{\infty}\frac{\kappa_{\alpha\beta}^{(d+1)}}{d!n^{\frac{1+d}{2}}}\frac{\partial^d (G^2)_{ji}}{\partial H_{\alpha\beta}^d}=\E \sum_{d=1}^{\infty}J_{1,d},
	\end{align*}
	where 
	$$J_{1,d}:=-\sum_{k,l=1}^K\frac{1}{N_{kl}}\sum\limits_{\substack{i\in C_{k}\\ j\in C_l}}\sum\limits_{\substack{\alpha\in C_k\\ \beta\in C_l}}\frac{\kappa_{\alpha\beta}^{(d+1)}}{d!n^{\frac{1+d}{2}}}\frac{\partial^d (G^2)_{ji}}{\partial H_{\alpha\beta}^d}.$$
	\begin{align*}
		&\E J_{1,1}\\
%=&\E\sum_{k,l=1}^K\frac{1}{N_{kl}}\sum\limits_{\substack{i\in C_{k}\\ j\in C_l}}\sum\limits_{\substack{\alpha\in C_k\\ \beta\in C_l}}\frac{\kappa_{\alpha\beta}^{(2)}}{n}[(G^2)_{j\alpha}G_{\beta i}+(G^2)_{j\beta}G_{\alpha i}+G_{j\alpha}(G^2)_{\beta i}+G_{j\beta}(G^2)_{\alpha i}]\\ 
		=&\E\sum_{k,l=1}^K\frac{Q^{(2)}_{kl}}{nN_{kl}}[1_{C_{l}}G^21_{C_{k}}1_{C_{l}}G1_{C_{k}}+1_{C_{l}}G^21_{C_{l}}1_{C_{k}}G1_{C_{k}}+1_{C_{l}}G1_{C_{k}}1_{C_{l}}G^21_{C_{k}}\\ 
		&+1_{C_{l}}G1_{C_{l}}1_{C_{k}}G^21_{C_{k}}]\\ 
		&-\E\sum_{k=1}^K\frac{1}{N_{kk}}\sum\limits_{\substack{i\in C_{k}\\ 
				j\in C_k}}\sum\limits_{\substack{\alpha\in C_k}}\frac{\kappa_{\alpha\alpha}^{(2)}}{n}[(G^2)_{j\alpha}G_{\alpha i}+(G^2)_{j\alpha}G_{\alpha i}+G_{j\alpha}(G^2)_{\alpha i}+G_{j\alpha}(G^2)_{\alpha i}]\\ 
		=&\E\sum_{k,l=1}^K\frac{Q^{(2)}_{kl}}{nN_{kl}}[Tr(T_kGT_lG)1_{C_{l}}G1_{C_{k}}+Tr(T_lGT_lG)1_{C_{k}}G1_{C_{k}}+1_{C_{l}}G1_{C_{k}}Tr(T_lGT_kG) \\ &+1_{C_{l}}G1_{C_{l}}Tr(T_kGT_kG)]-\E \sum_{k=1}^K\frac{2}{N_{kk}}\sum_{k=1}^K\frac{Q^{(2)}_{kk}}{n}(1_{C_k}(GT_kG^2)1_{C_k}+1_{C_k}(GT_kG^2)1_{C_k} ) \\ 
%		=&\E\sum_{k=1}^K\frac{Q^{(2)}_{kk}}{nN_{kk}}[Tr(T_kGT_kG)1_{C_{k}}G1_{C_{k}}+Tr(T_kGT_kG)1_{C_{k}}G1_{C_{k}}+1_{C_{k}}G1_{C_{k}}Tr(T_kGT_kG) \\ &+1_{C_{k}}G1_{C_{k}}Tr(T_kGT_kG)]+O(\frac{1}{n^2}) \\ 
		%	=&\E\sum_{k=1}^K\frac{Q^{(2)}_{kk}}{N_{kk}}[4Tr(T_kGT_kG)\alpha_kM_k]+O(\frac{1}{n^2}).
		=&O(\frac{1}{n})
	\end{align*}
	Now further consider $\E J_{(1,2)}$.
	\begin{align*}
		&-\E J_{1,2} 
%		=&\E\sum_{k,l=1}^K\frac{1}{N_{kl}}\sum\limits_{\substack{i\in C_k\\ j\in C_l}}\sum\limits_{\substack{\alpha\in C_k\\ \beta\in C_l}}\frac{\kappa_{\alpha\beta}^{(3)}}{2!n^{3/2}}[G_{j\alpha}G_{\beta\alpha}(G^2)_{\beta i}+G_{j\alpha}G_{\beta\beta}(G^2)_{\alpha i}+ G_{j\beta}G_{\alpha\alpha}(G^2)_{\beta i} +G_{j\beta}G_{\alpha\beta}(G^2)_{\alpha i}\\ 
%		&+   G_{j\alpha}(G^2)_{\beta\alpha}G_{\beta i}+G_{j\alpha}(G^2)_{\beta\beta}G_{\alpha i}+ G_{j\beta}(G^2)_{\alpha\alpha}G_{\beta i} +G_{j\beta}(G^2)_{\alpha\beta}G_{\alpha i}\\ 
%		&+(G^2)_{j\alpha}G_{\beta\alpha}G_{\beta i}+(G^2)_{j\alpha}G_{\beta\beta}G_{\alpha i}+ (G^2)_{j\beta}G_{\alpha\alpha}G_{\beta i} +(G^2)_{j\beta}G_{\alpha\beta}G_{\alpha i}  ]\\ 
		=\sum_{k,l=1}^K\E\frac{Q^{(3)}_{kl}}{2!N_{kl}n^{3/2}}\\ &\times\sum\limits_{\substack{i\in C_k\\ j\in C_l}}[(GT_kGT_lG^2)_{ji}+Tr(T_lG)(GT_kG^2)_{ji}+Tr(T_kG)(GTG^2)_{ji}+(GT_lGT_kG^2)_{ji}\\ 
		&+(GT_kG^2T_lG)_{ji}+Tr(T_lG^2)(GT_kG)_{ji} +Tr(T_kG^2)(GT_lG)_{ji}+(GT_lG^2T_kG)_{ji} \\ 
		&+(G^2T_kGT_lG)_{ji}+Tr(T_lG)(G^2T_KG)_{ji}+Tr(T_kG)(G^2T_lG)_{ji}+(G^2T_lGT_kG)_{ji}   ]\\  
		&-\E\sum_{k=1}^K\frac{1}{N_{kk}}\sum\limits_{\substack{i\in C_k\\ j\in C_k}}\sum\limits_{\substack{\alpha\in C_k}}\frac{Q^{(2)}_{kk}}{2!n^{3/2}}[G_{j\alpha}G_{\alpha\alpha}(G^2)_{\alpha i}+G_{j\alpha}G_{\alpha\alpha}(G^2)_{\alpha i}+ G_{j\alpha}G_{\alpha\alpha}(G^2)_{\alpha i}\\ 
		& +G_{j\alpha}G_{\alpha\alpha}(G^2)_{\alpha i} 
		+   G_{j\alpha}(G^2)_{\alpha\alpha}G_{\alpha i}+G_{j\alpha}(G^2)_{\alpha\alpha}G_{\alpha i}+ G_{j\alpha}(G^2)_{\alpha\alpha}G_{\alpha i} +G_{j\alpha}(G^2)_{\alpha\alpha}G_{\alpha i}\\ 
		&+(G^2)_{j\alpha}G_{\alpha\alpha}G_{\beta i}+(G^2)_{j\alpha}G_{\alpha\alpha}G_{\alpha i}+ (G^2)_{j\alpha}G_{\alpha\alpha}G_{\alpha i} +(G^2)_{j\alpha}G_{\alpha\alpha}G_{\alpha i}  ]\\ 
		=&O(\frac{1}{n^{3/2}}).
	\end{align*}
	
	Similarly, decompose $\E J_{1,3}$ into $\{\E J_{1,3}^{kl}\}_{k,l=1}^K$
	\begin{align*}
		\E J_{1,3}^{kl}=&\E\frac{1}{N_{kl}}\sum\limits_{\substack{i\in C_k\\ j\in C_l}}\sum\limits_{\substack{\alpha\in C_k\\ \beta\in C_l}}\frac{\kappa_{\alpha\beta}^{(4)}}{3!n^{2}}[e_j'G^2(e_{\alpha}e_{\beta}'+e_{\beta}e_{\alpha'})G(e_{\alpha}e_{\beta}'+e_{\beta}e_{\alpha'})G(e_{\alpha}e_{\beta}'+e_{\beta}e_{\alpha'})Ge_i\\ 
		&+e_j'G(e_{\alpha}e_{\beta}'+e_{\beta}e_{\alpha'})G^2(e_{\alpha}e_{\beta}'+e_{\beta}e_{\alpha'})G(e_{\alpha}e_{\beta}'+e_{\beta}e_{\alpha'})Ge_i\\ 
		&+e_j'G(e_{\alpha}e_{\beta}'+e_{\beta}e_{\alpha'})G(e_{\alpha}e_{\beta}'+e_{\beta}e_{\alpha'})G^2(e_{\alpha}e_{\beta}'+e_{\beta}e_{\alpha'})Ge_i\\ 
		&+e_j'G(e_{\alpha}e_{\beta}'+e_{\beta}e_{\alpha'})G(e_{\alpha}e_{\beta}'+e_{\beta}e_{\alpha'})G(e_{\alpha}e_{\beta}'+e_{\beta}e_{\alpha'})G^2e_i]\\ 
		&=O(\frac{1}{n^2}).
	\end{align*}

	Note that the normalizing constant is of order $\frac{1}{n^4}$, while the summation is over 4 independent indices $i,j,\alpha,\beta$. Further note that each term in the summation will be in the form $G^{s_1}_{jt_1}G^{s_2}_{\bar{t}_1t_2}G^{s_3}_{\bar{t}_2t_3}G^{s_4}_{\bar{t}_3i}$, where the integers $1\le s_1,s_2,s_3,s_4\le 2$, $s_1+s_2+s_3+s_4=5$,  for each pair of $(t_1,\bar{t}_1)$, $(t_2,\bar{t}_2)$ and $(t_3,\bar{t}_3)$, they are either $(\alpha,\beta)$ or $(\beta,\alpha)$. Note that we have odd number of $\alpha$'s and $\beta$'s and the the first of all 8 indices is $j$, with the last one to be $i$, so at least two won't be diagonal terms when all four indices are different, note also that only one of  $G^2$ could appear, which means at least one of $G_{j\alpha}$, $G_{j\beta}$, $G_{\alpha i}$ or $G_{\beta i}$ would appear in any of the  products, which yields a order of $O_{\prec}(\frac{1}{\sqrt{n}})$, thus minor.
	
	To be more precise, we have
	\begin{align*}
		&\E\frac{1}{N_{kl}}\sum\limits_{\substack{i\in C_k\\ j\in C_l}}\sum\limits_{\substack{\alpha\in C_k\\ \beta\in C_l}}\frac{\kappa_{\alpha\beta}^{(4)}}{3!n^{2}}(G^2)_{j\alpha}G_{\beta\beta}G_{\alpha\alpha}G_{\beta i}\\ 
		=&\E\frac{1}{N_{kl}}\sum\limits_{\substack{i\in C_k\\ j\in C_l}}\sum\limits_{\substack{\alpha\in C_k\\ \beta\in C_l}}\frac{\kappa_{\alpha\beta}^{(4)}}{3!n^{2}}[(G^2)_{j\alpha}(G_{\beta\beta}G_{\alpha\alpha}-M_kM_l)G_{\beta i}+(G^2)_{j\alpha}M_kM_lG_{\beta i}]\\ 
		=&\E\frac{1}{N_{kl}}\sum\limits_{\substack{\alpha\in C_k\\ \beta\in C_l}}\frac{\kappa_{\alpha\beta}^{(4)}}{3!n^{2}}\frac{n^{\varepsilon}}{\sqrt{n}}|\sum\limits_{\substack{i\in C_k\\ j\in C_l}}(G^2)_{j\alpha}G_{\beta i}|+O(\frac{1}{n^2})=O(\frac{n^{\varepsilon}}{n^{3/2}}).
	\end{align*}
	
	For higher-order expansion terms of  $\E J_{1,d}^{kl}$, $d\ge 4$, simply notice that the normalizing constant would be of order $O(\frac{1}{n^{9/2}})$, while the summation is over 4 indices with any of the terms to be of $O(1)$ due to the trivial bound $\|G(z)\|\le \frac{1}{\Im(z)}$, thus  minor.

	Then we proceed to $\E Tr[G(z)(\hat{H}-H)G(z)(\hat{H}-H)G(z)]$.
	\begin{align*}
		&\E Tr[G(z)(\hat{H}-H)G(z)(\hat{H}-H)G(z)]=\E Tr[(\hat{H}-H)G(\hat{H}-H)G^2]\\	 
%		=&\E \sum_{i,j}(\hat{H}-H)_{ij}(G(\hat{H}-H)G^2)_{ji}\\ 
		=&\E \sum_{i,j}\sum_{\substack{\alpha\in C_{\sigma(i)}\\ \beta\in C_{\sigma(j)}}}\frac{-H_{\alpha\beta}}{N_{\sigma(i)\sigma(j)}}(G(\hat{H}-H)G^2)_{ji}\\ 
%		=&\E \sum_{i,j}\sum_{\substack{\alpha\in C_{\sigma(i)}\\ \beta\in C_{\sigma(j)}}}\frac{-1}{N_{\sigma(i)\sigma(j)}}\sum_{d=1}^{\infty}\frac{\kappa_{\alpha\beta}^{(d+1)}}{d!n^{\frac{d+1}{2}}}\frac{\partial^d}{\partial H_{\alpha\beta}^d}(G(\hat{H}-H)G^2)_{ji}\\ 
%		=&\E \sum_{i,j}\sum_{\substack{\alpha\in C_{\sigma(i)}\\ \beta\in C_{\sigma(j)}}}\frac{-1}{N_{\sigma(i)\sigma(j)}}\sum_{d=2}^{\infty}\frac{\kappa_{\alpha\beta}^{(d+1)}}{d!n^{\frac{d+1}{2}}}\frac{\partial^d}{\partial H_{\alpha\beta}^d}(G(\hat{H}-H)G^2)_{ji}\\
%		&+ \E \sum_{k,l}\sum_{\substack{i\in C_k,\\  j\in C_l}}\sum_{\substack{\alpha\in C_{\sigma(i)}\\ \beta\in C_{\sigma(j)}}}\frac{1}{N_{kl}}\frac{Q^{(2)}_{kl}}{n}
%		\{  [G(E_{\alpha\beta}+E_{\beta\alpha})G(\hat{H}-H)G^2]_{ji} \\ 
%		&+ [G(\frac{1}{N_{kl}}(E_{B(k,l)}+E_{B(l,k)})G^2)]_{ji}\\ 
%		&+[G(\hat{H}-H)G^2(E_{\alpha\beta}+E_{\beta\alpha})G]_{ji}\\ 
%		&+[G(\hat{H}-H)G(E_{\alpha\beta}+E_{\beta\alpha})G^2]_{ji} \} \\
		=&\E \sum_{k,l}\frac{1}{N_{kl}}\frac{Q^{(2)}_{kl}}{n}
		1_{C_l}[G(E_{B(k,l)}+E_{B(l,k)})G(\hat{H}-H)G^2]1_{C_k} \\ 
		&+\E \sum_{k,l}\frac{1}{N_{kl}}\frac{Q^{(2)}_{kl}}{n}
		1_{C_l}[G(E_{B(k,l)}+E_{B(l,k)})G^2]1_{C_k} \\
		&+ \E \sum_{k,l}\frac{1}{N_{kl}}\frac{Q^{(2)}_{kl}}{n}
		1_{C_l}[G(\hat{H}-H)G^2(E_{B(k,l)}+E_{B(l,k)})G]1_{C_k} \\ 
		&+ \E \sum_{k,l}\frac{1}{N_{kl}}\frac{Q^{(2)}_{kl}}{n}
		1_{C_l}[G(\hat{H}-H)G(E_{B(k,l)}+E_{B(l,k)})G^2]1_{C_k} \\ 
		&+\E \sum_{i,j}\sum_{\substack{\alpha\in C_{\sigma(i)}\\ \beta\in C_{\sigma(j)}}}\frac{-1}{N_{\sigma(i)\sigma(j)}}\sum_{d=2}^{\infty}\frac{\kappa_{\alpha\beta}^{(d+1)}}{d!n^{\frac{d+1}{2}}}\frac{\partial^d}{\partial H_{\alpha\beta}^d}(G(\hat{H}-H)G^2)_{ji}\\ 
		=&o(1),
	\end{align*} where $E_{B(k,l)}$ indicates the block matrix  $1_{C_k}1_{C_l}^{\top}$.
	
	Among the above terms we know that only  
	\begin{align*}
		&\E \sum_{k,l}\frac{1}{N_{kl}}\frac{Q^{(2)}_{kl}}{n}
		1_{C_l}[G(E_{B(k,l)}+E_{B(l,k)})G^2]1_{C_k} \\ 
		=&\E \sum_{k,l}\frac{1}{N_{kl}}\frac{Q^{(2)}_{kl}}{n}
		[Tr(T_lG)\delta_{lk}Tr(T_kGT_lG)+Tr(T_lG)Tr(T_kGT_kG)]\\ 
	\end{align*} is  $O(\frac{1}{n})$, while the others are $O(\frac{1}{n^{3/2}})$.
	
	\subsection{Covariance function $Cov(Tr(\hat{G}(z_1)),Tr(\hat{G}(z_2)))$}\label{sec:covhat}
	
	To calculate this covariance function, first we need to do a decomposition.
	\begin{equation}
		\begin{aligned}
			&\E Cov(Tr(\hat{G}(z_1)),Tr(\hat{G}(z_2))) \\ 
			=& \E (Tr(\hat{G}(z_1))-\E Tr(\hat{G}(z_1)))(Tr(\hat{G}(z_2))-\E Tr(\hat{G}(z_2))) \\ 
			=&\E\{[(Tr(\hat{G}(z_1))-Tr G(z_1))-(\E Tr\hat{G}(z_1)-\E TrG(z_1))+(TrG(z_1)-\E TrG(z_1))]\\ &\times
			[(Tr(\hat{G}(z_2))-Tr G(z_2))-(\E Tr\hat{G}(z_2)-\E TrG(z_2))+(TrG(z_2)-\E TrG(z_2))]\}\\ 
			=&\E[a_1a_2+a_1b_2+a_1c_2+b_1a_2+b_1b_2+b_1c_2+c_1a_2+c_1b_2+c_1c_2],
		\end{aligned}
	\end{equation}
	where $$\begin{aligned}
		&a_i=Tr(\hat{G}(z_i))-Tr(G(z_i)), i=1,2, \\ 
		&b_i=-\E (a_i), i=1,2,\\ 
		&c_i=Tr(G(z_i))-\E Tr(G(z_i)), i=1,2.
	\end{aligned}$$

Instantly, we know that $\E b_1c_2=\E b_2c_1=0$.

	First, we consider $\E a_1a_2$.
	\begin{align*}
		\E [a_1a_2]=&  \E (Tr\hat{G}(z_1)-TrG(z_1))(Tr\hat{G}(z_2)-TrG(z_2))\\ 
		=&\E [Tr(-G(z_1)(\hat{H}-H)G(z_1)+G(z_1)(\hat{H}-H)G(z_1)(\hat{H}-H)G(z_1)\\ 
		&-G(z_1)(\hat{H}-H)G(z_1)(\hat{H}-H)G(z_1)(\hat{H}-H)G(z_1))+o_P(\frac{\log(n)}{n})]\times \\
		&[Tr(-G(z_2)(\hat{H}-H)G(z_2)+G(z_2)(\hat{H}-H)G(z_2)(\hat{H}-H)G(z_2))\\ 
		&-G(z_2)(\hat{H}-H)G(z_2)(\hat{H}-H)G(z_2)(\hat{H}-H)G(z_2))+o_P(\frac{\log(n)}{n})].
	\end{align*}
	The problem would be that there would be too many terms (including the $\frac{1}{\sqrt{n}}$ terms) that need to be calculated provided with trivial bound $\E |TrG(z)(\hat{H}-H)G(z)|^2=O(1)$.  
	%Similarly, we can show that $$\E Tr(G(z)(\hat{H}-H)G(z)(\hat{H}-H)G(z))Tr(G(z^*)(\hat{H}-H)G(z^*)(\hat{H}-H)G(z^*))=O(\frac{1}{\sqrt{n}}).$$

	Hence, we need a more efficient bound for $\E |TrG(z)(\hat{H}-H)G(z)|^2$,
	\begin{align*}
		&\E Tr(G(z)(\hat{H}-H)G(z))Tr(G(z^*)(\hat{H}-H)G(z^*))\\ 
		=&\E \sum_{i,j}\sum_{\alpha\in C_{\sigma(i)},\beta\in C_{\sigma(j)}}\frac{H_{\alpha\beta}}{N_{\sigma(i)\sigma(j)}}(G(z))^2_{ji}Tr(G(z^*)(\hat{H}-H)G(z^*))\\ 
%		=&\E\sum_{k,l=1}^K\sum_{i,\alpha\in C_{k},j,\beta\in C_{l}}\frac{1}{N_{kl}}\sum_{d=1}^{\infty}\frac{Q^{(d+1)}_{kl}}{n^{\frac{d+1}{2}}}\frac{\partial^d}{\partial H_{\alpha\beta}^d}[Tr(G(z^*)(\hat{H}-H)G(z^*))(G^2(z))_{ji}]\\ 
%		=&\E\sum_{k,l=1}^K\sum_{i,\alpha\in C_{k},j,\beta\in C_{l}}\frac{1}{N_{kl}}\sum_{d=1}^{\infty}\frac{Q^{(d+1)}_{kl}}{n^{\frac{d+1}{2}}}\frac{\partial^d}{\partial H_{\alpha\beta}^d}[Tr(G(z^*)(\hat{H}-H)G(z^*))(G^2(z))_{ji}]\\ 
%		=&\E\sum_{k,l=1}^K\sum_{i,\alpha\in C_{k},j,\beta\in C_{l}}\frac{1}{N_{kl}}\sum_{d=2}^{\infty}\frac{Q^{(d+1)}_{kl}}{n^{\frac{d+1}{2}}}\frac{\partial^d}{\partial H_{\alpha\beta}^d}[Tr(G(z^*)(\hat{H}-H)G(z^*))(G^2(z))_{ji}]\\ 
%		&+\E\sum_{k,l=1}^K\sum_{i,\alpha\in C_{k},j,\beta\in C_{l}}\frac{1}{N_{kl}}\frac{Q^{(2)}_{kl}}{n}Tr[G(z^*)(E_{\alpha\beta}+E_{\beta\alpha})G(z^*)(\hat{H}-H)G(z^*)](G^2(z))_{ji}\\ 
%		&+Tr(G(z^*)(\hat{H}-H)G(z^*)(E_{\alpha\beta}+E_{\beta\alpha})G(z^*))(G^2(z))_{ji}\\ 
%		&+\frac{1}{N_{kl}}Tr(G(z^*)(E_{B(k,l)}+E_{B(l,k)})G(z^*))(G^2(z))_{ji}\\ 
%		&+Tr(G(z^*)(\hat{H}-H)G(z^*))(G^2(z)(E_{\alpha\beta}+E_{\beta\alpha})G(z)+G(z)(E_{\alpha\beta}+E_{\beta\alpha})G^2(z))_{ji}\}\\ 
		=&\E\sum_{k,l=1}^K\sum_{i,\alpha\in C_{k},j,\beta\in C_{l}}\frac{1}{N_{kl}}\sum_{d=2}^{\infty}\frac{Q^{(d+1)}_{kl}}{n^{\frac{d+1}{2}}}\frac{\partial^d}{\partial H_{\alpha\beta}^d}[Tr(G(z^*)(\hat{H}-H)G(z^*))(G^2(z))_{ji}]\\ 
		&+\E\sum_{k,l=1}^K\sum_{i\in C_{k},j\in C_{l}}\frac{1}{N_{kl}}\frac{Q^{(2)}_{kl}}{n}\{Tr[G(z^*)(E_{B(k,l)}+E_{B(l,k)})G(z^*)(\hat{H}-H)G(z^*)](G^2(z))_{ji}\\ 
		&+Tr(G(z^*)(\hat{H}-H)G(z^*)(E_{B(k,l)}+E_{B(l,k)})G(z^*))(G^2(z))_{ji}\\ 
		&+Tr(G(z^*)(E_{B(k,l)}+E_{B(l,k)})G(z^*))(G^2(z))_{ji}\\ 
		&+Tr(G(z^*)(\hat{H}-H)G(z^*))(G^2(z)(E_{B(k,l)}+E_{B(l,k)})G(z)+G(z)(E_{B(k,l)}+E_{B(l,k)})G^2(z))_{ji}\}\\ 
		=&O(\frac{1}{\sqrt{n}})+\E\sum_{k,l=1}^K\sum_{i\in C_{k},j\in C_{l}}\frac{1}{N_{kl}}\frac{Q^{(2)}_{kl}}{n}Tr(G(z^*)(E_{B(k,l)}+E_{B(l,k)})G(z^*))(G^2(z))_{ji}\\
		=&O(\frac{\log(n)}{\sqrt{n}}) .
	\end{align*}
Also, note that
	\begin{align*}
		&\E Tr(G(z)(\hat{H}-H)G(z)(\hat{H}-H)G(z))Tr(G(z^*)(\hat{H}-H)G(z^*)(\hat{H}-H)G(z^*))\\ 
%		=&\E \sum_{m,j=1}^n\sum_{\alpha\in C_{\sigma(m)},\beta\in C_{\sigma(j)}}\frac{-H_{\alpha\beta}}{N_{\sigma(m)\sigma(j)}}(G(z)(\hat{H}-H)G^2(z))_{ji}Tr(G(z^*)(\hat{H}-H)G(z^*)(\hat{H}-H)G(z^*))\\ 
		=&\sum_{k,l=1}^K\sum_{m,\alpha\in C_k,j,\beta\in C_l}\frac{Q^{(2)}_{kl}}{N_{kl}n}\Big\{ (G(E_{\alpha\beta}+E_{\beta\alpha})G(\hat{H}-H)G^2)_{ji}(z)Tr(G^*(\hat{H}-H)G^*(\hat{H}-H)G^*)\\ 
		&+(G(\hat{H}-H)G(E_{\alpha\beta}+E_{\beta\alpha})G^2)_{ji}(z)Tr(G^*(\hat{H}-H)G^*(\hat{H}-H)G^*)\\ 
		&+(G(\hat{H}-H)G^2(E_{\alpha\beta}+E_{\beta\alpha})G)_{ji}(z)Tr(G^*(\hat{H}-H)G^*(\hat{H}-H)G^*)\\ 
		&+\frac{1}{N_{kl}}(G(E_{B(k,l)}+E_{B(l,k)})G^2)_{ji}(z)Tr(G^*(\hat{H}-H)G^*(\hat{H}-H)G^*)\\ 
		&+(G(\hat{H}-H)G^2)_{ji}(z)Tr(G^*(E_{\alpha\beta}+E_{\beta\alpha})G^*(\hat{H}-H)G^*(\hat{H}-H)G^*) \\ 
		&+(G(\hat{H}-H)G^2)_{ji}(z)Tr(G^*(\hat{H}-H)G^*(E_{\alpha\beta}+E_{\beta\alpha})G^*(\hat{H}-H)G^*) \\ 
		&+(G(\hat{H}-H)G^2)_{ji}(z)Tr(G^*(\hat{H}-H)G^*(\hat{H}-H)G^*(E_{\alpha\beta}+E_{\beta\alpha})G^*) \\
		&+\frac{1}{N_{kl}}(G(\hat{H}-H)G^2)_{ji}(z)Tr(G^*(E_{B(k,l)}+E_{B(l,k)})G^*(\hat{H}-H)G^*) \\
		&+\frac{1}{N_{kl}}(G(\hat{H}-H)G^2)_{ji}(z)Tr(G^*(\hat{H}-H)G^*(E_{B(k,l)}+E_{B(l,k)})G^*) \Big\}+O(\frac{1}{\sqrt{n}})\\ 
		=&O(\frac{\log(n)}{\sqrt{n}}).
	\end{align*}	
	Thus, by Cauchy inequality we can show that
	\begin{align*}
		&\E [a_1a_2]\\ 
		=&  \E (Tr\hat{G}(z_1)-TrG(z_1))(Tr\hat{G}(z_2)-TrG(z_2))\\ 
		=&\E[Tr(-G(z_1)(\hat{H}-H)G(z_1)+G(z_1)(\hat{H}-H)G(z_1)(\hat{H}-H)G(z_1)+o_p(\frac{\log(n)}{\sqrt{n}}))\times \\
		&Tr(-G(z_2)(\hat{H}-H)G(z_2)+G(z_2)(\hat{H}-H)G(z_2)(\hat{H}-H)G(z_2))+o_p(\frac{\log(n)}{\sqrt{n}}))]\\ 
		=&O(\frac{\log(n)}{\sqrt{n}}).
	\end{align*} 
	In the meantime, by Section \ref{sec:hatmeanpf}, we know 
	\begin{align*}
		\E [a_1b_2]=&-  \E (Tr\hat{G}(z_1)-TrG(z_1))\E(Tr\hat{G}(z_2)-TrG(z_2))=O(\frac{\log(n)^2}{n}),\\ 
		\E [a_2b_1]=&-  \E (Tr\hat{G}(z_2)-TrG(z_2))\E(Tr\hat{G}(z_1)-TrG(z_1))=O(\frac{\log(n)^2}{n}),\\ 
		\E [b_1b_2]=&  \E (Tr\hat{G}(z_2)-TrG(z_2))\E(Tr\hat{G}(z_1)-TrG(z_1))=O(\frac{\log(n)^2}{n}).
	\end{align*} 
	While $\E c_1c_2$ is also known by Section \ref{sec:covlm},  we only need to consider $\E a_1c_2$.
	\begin{align*}
		|\E a_1c_2|^2&=|\E (Tr(\hat{G}(z_1)-G(z_1)))(Tr(G(z_2))-	\E Tr(G(z_2)))|^2\\ 
		&\le \E |Tr(\hat{G}(z_1)-G(z_1))|^2\E |Tr(G(z_2))-	\E Tr(G(z_2))|^2
	\end{align*}
	Recall the system of equations for  $\{Cov_{lm}\}_{l=1}^K,\forall m\in [K]$ in \ref{sec:covlm}, note that the entries of the coefficient matrices are of order 1, thus 
	$$\E |Tr(G(z_2))-\E Tr(G(z_2))|^2=O(1).$$
	So we have $\E [a_1c_2]=O(\frac{\log(n)}{\sqrt{n}}).$ 	Similarly,
	$\E [a_2c_1]=O(\frac{\log(n)}{\sqrt{n}}).$
	Then we see that only $\E c_1c_2$ will count, which means that we will have exactly the same covariance function as in Section \ref{sec:covlm}. It remains to show the normality.

	\subsection{Proof of normality for the data-driven version}\label{sec:gaussianityhat}

	\begin{proof}
		
		% roughly speaking, we wanna prove that a real random variable $R$ is a gaussian random variable with mean zero, it  suffices to prove that
		%			
		%			$$\E e^{it R}=e^{-\frac{1}{2} \sigma^{2} t^{2}}, $$
		%			further, we may instead choose to prove that 
		%			$it\E Re^{it R}=-\sigma t^2e^{-\frac{1}{2} \sigma^{2} t^{2}}$ along with certain boundary conditions. 
		
		The main procedure is exactly the same as that in Section \ref{sec:gaussianitynonhat}. For simplicity we will mainly focus more on the difference, some of the overlapping details will not be stated. Let $\hat{\gamma}(z)=\hat{\gamma}^{(n)}(z):=\Re \langle \hat{G}(z)\rangle$ and $\hat{\theta}(z)=\hat{\theta}^{(n)}(z)=\Im \langle \hat{G}(z)\rangle$,
		$$
		\hat{\Psi}(z, c)=\left\{\begin{array}{ll}{\hat{\gamma}(z)} & {\text { if } c=\gamma} \\ {\hat{\theta}(z)} & {\text { if } c=\theta}\end{array}\right. ,
		$$
		and extend the definition of $\hat{a}(c)$ and $\hat{b}(c)$, s.t.
		$$
		(\hat{a}(c), \hat{b}(c))=\left\{\begin{array}{ll}{(1 / 2,1 / 2)} & {\text { if } c=\hat{\gamma}} \\ {(1 / 2 i, 1 / 2 i)} & {\text { if } c=\hat{\theta}}\end{array}\right. .
		$$
		
		Apparently $\E\{{\hat{\Psi}(z, c)}\}=0$. Then our goal is to prove that $\forall \text{ fixed }q\in \mathbb{Z}_+,\ \{z_{s}\}_{s=1}^q\in \{ \mathbb{C}\backslash B_{\varepsilon_0}(\sigma({H}))\}^q,\ \{c_{s}\}_{s=1}^q \in\{\hat{\gamma}, \hat{\theta}\}^q$,  the joint probability distribution of random variables
		$\hat{\Psi}\left(z_{1}, c_{1}\right), \ldots, \hat{\Psi}\left(z_{q}, c_{q}\right)$ is the $q$-dimensional Gaussian distribution with zero mean and feasible covariance matrix.
%		$$
%		\begin{aligned} \E \left\{\hat{\Psi}\left(z_{s}, c_{s}\right) \hat{\Psi}\left(z_{t}, c_{t}\right)\right\}=& \hat{a}\left(c_{s}\right) \hat{a}\left(c_{t}\right) \hat{g}\left(z_{s}, z_{t}\right)+\hat{a}\left(c_{s}\right) \hat{b}\left(c_{t}\right) \hat{g}\left(z_{s}, z_{t}^{*}\right)+\\ & \hat{a}\left(c_{t}\right) \hat{b}\left(c_{s}\right) \hat{g}\left(z_{s}^{*}, z_{t}\right)+\hat{b}\left(c_{s}\right) \hat{b}\left(c_{t}\right) \hat{g}\left(z_{s}^{*}, z_{t}^{*}\right). \end{aligned}
%		$$
		Then we  consider the characteristic function of $\hat{\Psi}\left(z_{1}, c_{1}\right), \ldots, \hat{\Psi}\left(z_{q}, c_{q}\right)$,	
		\begin{align*} \hat{e}_{q}^{(n)}\left(T_{q}, C_{q}, Z_{q}\right)=&\prod_{s=1}^{q} \exp \left\{i \tau_{s}\left[\hat{a}\left(c_{s}\right) Tr\langle \hat{G}(z_{s})\rangle+\hat{b}\left(c_{s}\right) Tr\langle \hat{G}(z_{s}^{*})\rangle\right]\right\}.
		\end{align*}
		where  $T_{q}=\left(\tau_{1}, \ldots, \tau_{q}\right), C_{q}=\left(c_{1}, \ldots, c_{q}\right), Z_{q}=\left(z_{1}, \ldots, z_{q}\right)$.  And we will simply use $\hat{e}_q$ when there is no confusion. 
%Again note that
%		$$
%		\frac{\partial}{\partial \tau_{s}} \E\left\{\hat{e}_{q}\left(T_{q}\right)\right\}=i \E\left\{\hat{e}_{q}\left[\hat{a}\left(c_{s}\right) Tr \langle \hat{G}(z_{s})\rangle+\hat{b}\left(c_{s}\right)Tr \langle \hat{G}(z_{s}^{*})\rangle\right]\right\}.
%		$$
%		and we need to show that there exist sequences of the  coefficients of the covariance matrices $(\hat{\Sigma}_{s t}^{(n)})_{s, t=1}^q$ such that for each fixed $T_{q}$			
%		$$\lim _{n \rightarrow \infty}\left|\E\hat{e}_{q}^{(n)}\left[\hat{a}\left(c_{s}\right) Tr\langle \hat{G}(z_{s})\rangle+\hat{b}\left(c_{s}\right) Tr\langle \hat{G}
%		(z_{s}^{*})\rangle\right]-i \sum_{t=1}^{q} \tau_{s} \hat{\Sigma}_{s t}^{(n)} \E\hat{e}_{q}^{(n)}\right|=0,  z \in \mathbb{C}\backslash\mathbb{R},$$
%		and to show that limits of all those coefficients exist
%		$$
%		\hat{\Sigma}_{s t}=\lim _{n \rightarrow \infty}\hat{\Sigma}_{s t}^{(n)}
%		$$
%		and correspond to our previously calculated covariance functions.
		
%		Similar to Section \ref{sec:gaussianitynonhat}, 
%		we need to compute
%		$$
%		\E\left\{\hat{e}_{q} \langle \hat{G}(z)\rangle\right\}=\sum_{j=1}^{n} \E\left\{\langle\hat{e}_{q} \rangle\hat{G}_{j j}\right\}.
%		$$
		
		By the resolvent identity, we have
		\begin{align*}
			z\sum_{j=1}^{n} \E\left\{\langle \hat{e}_{q}\rangle \hat{G}_{j j}\right\}=&\sum_{j=1}^{n} \E\left\{\langle \hat{e}_{q}\rangle (G-G(\tilde{H}-H)G+G(\tilde{H}-H)G(\tilde{H}-H)G\right.\\ &\left.-\hat{G}(\tilde{H}-H)G(\tilde{H}-H)G(\hat{H}-H)G)_{jj}\right\}\\ 
			=	&J_{3}^{(1)}+J_{3}^{(2)}+J_{3}^{(3)}+O(\frac{\log(n)}{\sqrt{n}}).
		\end{align*}
		where 
		\begin{align*}
			J_{3}^{(1)}=	&\sum_{j=1}^{n} \E\left\{\langle\hat{e}_q \rangle G_{j j}\right\}=z^{-1} \sum_{j, m=1}^{n} \E\left\{\langle \hat{e}_q\rangle G_{j m} H_{m j}\right\}\\ 
			=&z^{-1}\sum_{j, m=1}^{n}	(\sum_{d=0}^p\frac{\kappa^{(d+1)}_{mj}}{d!}\E[\frac{\partial^d \langle \hat{e}_q\rangle G_{jm}}{\partial H_{mj}^d}]+\varepsilon_{mj})
			=\sum_{a+b=1}^3J^{(1)}_{3,(a,b)}+\varepsilon_{J_3^{(1)}}.
		\end{align*}

		Note that by Cauchy inequality and adopting the same way we deal with $e_q$, we can show that all the terms whose counterparts vanish in the case of $e_q$ will still vanish here. First, we need to approximate the derivatives.
		\begin{align*}
			\frac{\partial Tr\hat{G}}{\partial H_{jm}}&=\frac{\partial Tr(G-G(\hat{H}-H)G+G(\hat{H}-H)G(\hat{H}-H)G) +\frac{\log(n)}{\sqrt{n}}}{\partial H_{jm}} .
		\end{align*}
		\begin{align*}
			\frac{\partial \langle \hat{e}_q \rangle}{\partial H_{jm}}&=\E\langle e_q\rangle \frac{\partial \sum_{s=1}^{q} \left\{i \tau_{s}\left[\hat{a}\left(c_{s}\right) Tr \langle \hat{G}(z_{s})\rangle+\hat{b}\left(c_{s}\right) Tr \langle \hat{G}(z_{s}^{*})\rangle\right]\right\}}{\partial H_{jm}} .
		\end{align*}
		One should note that truncating the infinite expansions  to get approximation of the derivatives like this is always dangerous. However, note that the form of the higher-order expansion terms are always clear in the sense that $(\hat{H}-H)$ will contribute  one more $\frac{\log(n)}{\sqrt{n}}$. Also in our setting \eqref{defi:nonhatversion}, $\forall i,j\in[n]$, $H_{ij}$ is the averaging of centered Bernoulli random variable, thus always bounded. So we may use a finite expansion here.  We can see that 		
		$$		J^{(1)}_{3,(0,1)}=-\frac{1}{nz}\E \sum_{j,m}\kappa^{(2)}_{jm}G_{jm}\hat{e}_q\left[ \sum_{s=1}^qi\tau_s(2\hat{a}_s\frac{\partial Tr\hat{G}(z_s)}{\partial H_{jm}}+2\hat{b}_s\frac{\partial Tr\hat{G}(z_s^*)}{\partial H_{jm}}) \right]. $$
		
		Comparing with $I_{3,(0,1)}$, it's not hard to see that as long as we can prove 
		\begin{align*}
			\frac{1}{n}\E\sum_{j,m}\kappa^{(2)}_{jm}G_{jm}\frac{\partial Tr G(z_s)(\hat{H}-H)G(z_s)}{\partial H_{jm}}=o(1),
		\end{align*}	and	
		\begin{align*}
			\frac{1}{n}\E\sum_{j,m}\kappa^{(2)}_{jm}G_{jm}\frac{\partial Tr G(z_s)(\hat{H}-H)G(z_s)(\hat{H}-H)G(z_s)}{\partial H_{jm}}=o(1),
		\end{align*}	
		 the non-vanishing contribution of the terms to the covariance terms would be the same as in Section \ref{sec:gaussianitynonhat}.

		Easy to see that
		\begin{align*}
			\sum_{m,j}	G_{mj}\frac{\kappa^{(2)}_{mj}}{n}Tr(G(z)(E_{mj}+E_{jm})G(z)(\hat{H}-H)G(z))=O(\frac{\log(n)}{\sqrt{n}})
		\end{align*}
	is minor.	So are the other components generated by the reminder terms of the derivatives.
		
		Similar things happen when we consider the analog of  $I_{3,(1,2)}$
		\begin{align*}
			&J^{(1)}_{3,(1,2)}=-\frac{1}{n^2z}\sum_{j,m}\frac{\kappa^{(4)}_{jm}}{3!}3(G^2_{jm}+G_{jj}G_{mm})\frac{\partial^2\langle \hat{e}_q\rangle }{\partial H_{jm}^2},
		\end{align*}  
		the repetitive $O(\frac{\log(n)}{\sqrt{n}})$ factors introduced by $(\hat{H}-H)$ make the terms generated by the difference between $\hat{G}$ and $G$ minor.

		Then it remains to show that $J_3^{(2)}$ and $J_3^{(3)}$ are minor.
		\begin{align*}
			&J_{3}^{(2)}:=-\sum_{j=1}^n\E\{\langle \hat{e}_q\rangle(G(\hat{H}-H)G)_{jj}  \} 
%			=&-\sum_{j,m=1}^n\E \langle \hat{e}_q\rangle(\hat{H}-H)_{jm}(G^2)_{mj}\\ 
			=\sum_{j,m=1}^n\E \langle \hat{e}_q\rangle \sum_{\alpha\in C_{\sigma(m)},\beta\in C_{\sigma(j)}}\frac{H_{\alpha\beta}}{N_{\sigma(m)\sigma(j)}}(G^2)_{mj}\\ 
			=&\sum_{k,l=1}^K\sum_{m,\alpha\in C_k,j,\beta\in C_l}\E\frac{Q^{(2)}_{kl}}{N_{kl}n}\langle \hat{e}_q\rangle(G_{m\alpha}(G^2)_{\beta j}+G_{m\beta}(G^2)_{\alpha j}+(G^2)_{m\alpha}(G_{\beta j})+(G^2)_{m\alpha}G_{\beta j})\\ 
			&+O(\frac{\log(n)}{n})\\
			=&O(\frac{\log(n)}{n}).
		\end{align*}
		\begin{align*}
			&J_{3}^{(3)}:=\sum_{j=1}^n\E[\langle \hat{e}_q\rangle(G(\hat{H}-H)G(\hat{H}-H)G)_{jj}]=\sum_{j,m=1}^n\E \langle \hat{e}_q\rangle(\hat{H}-H)_{jm}(G(\hat{H}-H)G^2)_{mj}\\ 
			=&-\sum_{j,m=1}^n\E \langle \hat{e}_q\rangle \sum_{\alpha\in C_{\sigma(m)},\beta\in C_{\sigma(j)}}\frac{H_{\alpha\beta}}{N_{\sigma(m)\sigma(j)}}(G(\hat{H}-H)G^2)_{mj}\\ 
			=&\sum_{k,l=1}^K\sum_{m,\alpha\in C_k,j,\beta\in C_l}\E\frac{Q^{(2)}_{kl}}{N_{kl}n}\langle \hat{e}_q\rangle(G_{m\alpha}(G(\hat{H}-H)G^2)_{\beta j}+G_{m\beta}(G(\hat{H}-H)G^2)_{\alpha j}\\ 
			&+(G(\hat{H}-H)G)_{m\alpha}(G^2)_{\beta j}+(G(\hat{H}-H)G)_{m\beta}(G^2)_{\alpha j}\\ 
			&+(G(\hat{H}-H)G^2)_{m\alpha}G_{\beta j}+(G(\hat{H}-H)G^2)_{m\beta}G_{\alpha j}+\frac{1}{N_{kl}}(G(E_{B(k,l)}+E_{B(l,k)})G^2)_{mj}\\ 
			&+O(\frac{\log(n)}{n})\\
			=&O(\frac{\log(n)}{n}).
		\end{align*}
		Thus, we may also conclude that the covariance function would be the same as that of Theorem \ref{thm:mainBlockWigner} and the normality follows.

		\subsection{Tightness of the process $\langle Tr \hat{G}(z)\rangle$}\label{sec:tightness2}
		
		Similarly, after we establish the finite dimensional convergence, it left to show that the process $Tr\langle \hat{G}(z)\rangle, z\in \mathbb{C}\backslash B_{\varepsilon_0}(\sigma(\hat{H}))$ is tight. We will show that 
		\begin{equation}
			\E |Tr\langle \hat{G}(z_1)\rangle-Tr \langle \hat{G}(z_2)\rangle|^2=O(|z_1-z_2|^2).
		\end{equation}
		
		Again note that \begin{equation*}
			\begin{aligned}
				&\E |Tr\langle \hat{G}(z_1)\rangle-Tr \langle \hat{G}(z_2)\rangle|^2\\ 
				=&\E |Tr\langle \hat{G}(z_1)\hat{G}(z_2)\rangle|^2|z_1-z_2|^2.
			\end{aligned}
		\end{equation*}
		and we can break down the question to boundedness of $\E |Tr\langle T_l\hat{G}(z_1)T_m\hat{G}(z_2)\rangle|^2$ and adopt a similar approach to Section \ref{sec:boundtgtgvar}. The details are omitted here.

	\end{proof}		

%
%\newpage
%
%\large\textbf{Appendix}
%
%%%%%%%%%%%%%%%%%%%%%%%%%%%%%%%%%%%%%%%%%%%%%%
%% Example with single Appendix:            %%
%%%%%%%%%%%%%%%%%%%%%%%%%%%%%%%%%%%%%%%%%%%%%%

	% Authors must disclose all relationships or interests that 
	% could have direct or potential influence or impart bias on 
	% the work: 
	%
	% \section*{Conflict of interest}
	%
	% The authors declare that they have no conflict of interest.

%%%%%%%%%%%%%%%%%%%%%%%%%%%%%%%%%%%%%%%%%%%%%%
%% Support information, if any,             %%
%% should be provided in the                %%
%% Acknowledgements section.                %%
%%%%%%%%%%%%%%%%%%%%%%%%%%%%%%%%%%%%%%%%%%%%%%
\begin{acks}[Acknowledgments]
	The authors would like to thank Prof. Zhigang Bao at HKUST for his insightful suggestions and comments.
\end{acks}

\bibliographystyle{imsart-number} % Style BST file (imsart-number.bst or imsart-nameyear.bst)
%\bibliography{bibliography}       % Bibliography file (usually '*.bib')

%	
% BibTeX users please use one of
%\bibliographystyle{spbasic}      % basic style, author-year citations
%\bibliographystyle{spmpsci}      % mathematics and physical sciences
%\bibliographystyle{spphys}       % APS-like style for physics

\bibliography{library.bib}

\begin{thebibliography}{30}
% BibTex style file: imsart-number.bst, 2017-11-03
% Default style options (sort=1,type=number).
% Used options (sort=1,type=number).

\bibitem{Adhikari2019a}
\begin{barticle}[author]
\bauthor{\bsnm{Adhikari},~\bfnm{Kartick}\binits{K.}},
  \bauthor{\bsnm{Jana},~\bfnm{Indrajit}\binits{I.}} \AND
  \bauthor{\bsnm{Saha},~\bfnm{Koushik}\binits{K.}}
(\byear{2021}).
\btitle{Linear eigenvalue statistics of random matrices with a variance
  profile}.
\bjournal{Random Matrices: Theory and Applications}
\bvolume{10}
\bpages{2250004}.
\bdoi{10.1142/S2010326322500046}
\end{barticle}
\endbibitem

\bibitem{Airoldi2013}
\begin{barticle}[author]
\bauthor{\bsnm{Airoldi},~\bfnm{Edoardo}\binits{E.}},
  \bauthor{\bsnm{Costa},~\bfnm{Thiago}\binits{T.}} \AND
  \bauthor{\bsnm{Chan},~\bfnm{Stanley}\binits{S.}}
(\byear{2013}).
\btitle{Stochastic blockmodel approximation of a graphon: Theory and consistent
  estimation}.
\bjournal{Advances in Neural Information Processing Systems}.
\end{barticle}
\endbibitem

\bibitem{Ajanki}
\begin{barticle}[author]
\bauthor{\bsnm{Ajanki},~\bfnm{Oskari}\binits{O.}},
  \bauthor{\bsnm{Erd\H{o}s},~\bfnm{L\'aszl\'o}\binits{L.}} \AND
  \bauthor{\bsnm{Kr\"uger},~\bfnm{Torben}\binits{T.}}
(\byear{2015}).
\btitle{Quadratic Vector Equations On Complex Upper Half-Plane}.
\bjournal{Memoirs of the American Mathematical Society}
\bvolume{261}.
\bdoi{10.1090/memo/1261}
\end{barticle}
\endbibitem

\bibitem{Ajanki2017a}
\begin{barticle}[author]
\bauthor{\bsnm{Ajanki},~\bfnm{Oskari~H.}\binits{O.~H.}},
  \bauthor{\bsnm{Erd\H{o}s},~\bfnm{L\'aszl\'o}\binits{L.}} \AND
  \bauthor{\bsnm{Kr\"uger},~\bfnm{Torben}\binits{T.}}
(\byear{2017}).
\btitle{{Universality for general Wigner-type matrices}}.
\bjournal{Probability Theory and Related Fields}
\bvolume{169}
\bpages{667--727}.
\bdoi{10.1007/s00440-016-0740-2}
\end{barticle}
\endbibitem

\bibitem{Anderson2006}
\begin{barticle}[author]
\bauthor{\bsnm{Anderson},~\bfnm{Greg~W}\binits{G.~W.}} \AND
  \bauthor{\bsnm{Zeitouni},~\bfnm{Ofer}\binits{O.}}
(\byear{2006}).
\btitle{{A CLT for a band matrix model}}.
\bjournal{Probability Theory and Related Fields}
\bvolume{134}
\bpages{283--338}.
\bdoi{10.1007/s00440-004-0422-3}
\end{barticle}
\endbibitem

\bibitem{Bai2009}
\begin{bbook}[author]
\bauthor{\bsnm{Bai},~\bfnm{Z.}\binits{Z.}} \AND
  \bauthor{\bsnm{Silverstein},~\bfnm{Jack}\binits{J.}}
(\byear{2010}).
\btitle{{Spectral Analysis of Large Dimensional Random Matrices}}.
\bpublisher{Springer}.
\bdoi{10.1007/978-1-4419-0661-8}
\end{bbook}
\endbibitem

\bibitem{bai2005}
\begin{barticle}[author]
\bauthor{\bsnm{Bai},~\bfnm{Z.~D.}\binits{Z.~D.}} \AND
  \bauthor{\bsnm{Yao},~\bfnm{J.}\binits{J.}}
(\byear{2005}).
\btitle{On the convergence of the spectral empirical process of Wigner
  matrices}.
\bjournal{Bernoulli}
\bvolume{11}
\bpages{1059--1092}.
\bdoi{10.3150/bj/1137421640}
\end{barticle}
\endbibitem

\bibitem{Bai1999}
\begin{barticle}[author]
\bauthor{\bsnm{Bai},~\bfnm{Z~D}\binits{Z.~D.}}
(\byear{1999}).
\btitle{{Methodologies in spectral analysis of large dimensional random
  matrices, a review}}.
\bjournal{Statistica Sinica}
\bvolume{9}
\bpages{611--677}.
\end{barticle}
\endbibitem

\bibitem{Bai2004}
\begin{barticle}[author]
\bauthor{\bsnm{Bai},~\bfnm{Z~D}\binits{Z.~D.}} \AND
  \bauthor{\bsnm{Silverstein},~\bfnm{Jack~W}\binits{J.~W.}}
(\byear{2004}).
\btitle{{CLT for linear spectral statistics of large-dimensional sample
  covariance matrices}}.
\bjournal{Annals of Probability}
\bvolume{32}
\bpages{553--605}.
\end{barticle}
\endbibitem

\bibitem{Banerjee2017}
\begin{barticle}[author]
\bauthor{\bsnm{Banerjee},~\bfnm{Debapratim}\binits{D.}} \AND
  \bauthor{\bsnm{Ma},~\bfnm{Zongming}\binits{Z.}}
(\byear{2017}).
\btitle{{Optimal hypothesis testing for stochastic block models with growing
  degrees}}.
\bjournal{arXiv: 1705.05305}
\bpages{1--77}.
\end{barticle}
\endbibitem

\bibitem{bao2021quantitative}
\begin{barticle}[author]
\bauthor{\bsnm{Bao},~\bfnm{Zhigang}\binits{Z.}} \AND
  \bauthor{\bsnm{He},~\bfnm{Yukun}\binits{Y.}}
(\byear{2021}).
\btitle{Quantitative CLT for linear eigenvalue statistics of Wigner matrices}.
\bjournal{arXiv:2103.05402}.
\end{barticle}
\endbibitem

\bibitem{Male2013}
\begin{barticle}[author]
\bauthor{\bsnm{Benaych-Georges},~\bfnm{Florent}\binits{F.}},
  \bauthor{\bsnm{Guionnet},~\bfnm{Alice}\binits{A.}} \AND
  \bauthor{\bsnm{Male},~\bfnm{Camille}\binits{C.}}
(\byear{2014}).
\btitle{{Central limit theorems for linear statistics of heavy tailed random
  matrices}}.
\bjournal{{Communications in Mathematical Physics}}
\bvolume{239}
\bpages{641-686}.
\bdoi{10.1007/s00220-014-1975-3}
\end{barticle}
\endbibitem

\bibitem{Benaych-Georges2016}
\begin{barticle}[author]
\bauthor{\bsnm{Benaych-Georges},~\bfnm{Florent}\binits{F.}} \AND
  \bauthor{\bsnm{Knowles},~\bfnm{Antti}\binits{A.}}
(\byear{2018}).
\btitle{{Lectures on the local semicircle law for Wigner matrices}}.
\bjournal{Panoramas et Syntheses, Soci\'et\'e Math\'ematique de France}
\bvolume{53}
\bpages{1--90}.
\end{barticle}
\endbibitem

\bibitem{Chatterjee2006}
\begin{barticle}[author]
\bauthor{\bsnm{Chatterjee},~\bfnm{Sourav}\binits{S.}}
(\byear{2008}).
\btitle{A New Method of Normal Approximation}.
\bjournal{The Annals of Probability}
\bvolume{36}
\bpages{1584--1610}.
\end{barticle}
\endbibitem

\bibitem{Chatterjee2009}
\begin{barticle}[author]
\bauthor{\bsnm{Chatterjee},~\bfnm{Sourav}\binits{S.}}
(\byear{2009}).
\btitle{Fluctuations of eigenvalues and second order Poincar\'e inequalities}.
\bjournal{Probability Theory and Related Fields}
\bvolume{143}
\bpages{1-40}.
\bdoi{10.1007/s00440-007-0118-6}
\end{barticle}
\endbibitem

\bibitem{Cipolloni2020}
\begin{barticle}[author]
\bauthor{\bsnm{Cipolloni},~\bfnm{Giorgio}\binits{G.}},
  \bauthor{\bsnm{Erd\H{o}s},~\bfnm{L\'{a}szl\'{o}}\binits{L.}} \AND
  \bauthor{\bsnm{Schr\"{o}der},~\bfnm{Dominik}\binits{D.}}
(\byear{2020}).
\btitle{{Functional Central Limit Theorems for Wigner Matrices}}.
\bjournal{arXiv:2012.13218}.
\end{barticle}
\endbibitem

\bibitem{Costin1995}
\begin{barticle}[author]
\bauthor{\bsnm{Costin},~\bfnm{Ovidiu}\binits{O.}} \AND
  \bauthor{\bsnm{Lebowitz},~\bfnm{Joel~L}\binits{J.~L.}}
(\byear{1995}).
\btitle{{Gaussian fluctuation in random matrices}}.
\bjournal{Physical Review Letters}
\bvolume{75}
\bpages{69--72}.
\bdoi{10.1103/PhysRevLett.75.69}
\end{barticle}
\endbibitem

\bibitem{Erdoes2011}
\begin{barticle}[author]
\bauthor{\bsnm{Erd{\H{o}}s},~\bfnm{L\'aszl\'o}\binits{L.}}
(\byear{2011}).
\btitle{{Universality of Wigner random matrices: a survey of recent results}}.
\bjournal{Russian Mathematical Surveys}
\bvolume{66}
\bpages{507--626}.
\bdoi{10.1070/rm2011v066n03abeh004749}
\end{barticle}
\endbibitem

\bibitem{Erds2019MDE}
\begin{barticle}[author]
\bauthor{\bsnm{Erd\H{o}s},~\bfnm{L.}\binits{L.}}
(\byear{2019}).
\btitle{The matrix Dyson equation and its applications for random matrices}.
\bjournal{arXiv: 1903:10060}.
\end{barticle}
\endbibitem

\bibitem{Guionnet2002a}
\begin{barticle}[author]
\bauthor{\bsnm{Guionnet},~\bfnm{Alice}\binits{A.}}
(\byear{2002}).
\btitle{{Large deviations upper bounds and central limit theorems for
  non-commutative functionals of Gaussian large random matrices}}.
\bjournal{Annales de l'institut Henri Poincare (B) Probability and Statistics}
\bvolume{38}
\bpages{341--384}.
\bdoi{10.1016/S0246-0203(01)01093-7}
\end{barticle}
\endbibitem

\bibitem{He2017a}
\begin{barticle}[author]
\bauthor{\bsnm{He},~\bfnm{Yukun}\binits{Y.}} \AND
  \bauthor{\bsnm{Knowles},~\bfnm{Antti}\binits{A.}}
(\byear{2017}).
\btitle{{Mesoscopic eigenvalue statistics of wigner matrices}}.
\bjournal{Annals of Applied Probability}
\bvolume{27}
\bpages{1510--1550}.
\bdoi{10.1214/16-AAP1237}
\end{barticle}
\endbibitem

\bibitem{Khorunzhy1996}
\begin{barticle}[author]
\bauthor{\bsnm{Khorunzhy},~\bfnm{Alexei~M.}\binits{A.~M.}},
  \bauthor{\bsnm{Khoruzhenko},~\bfnm{Boris~A.}\binits{B.~A.}} \AND
  \bauthor{\bsnm{Pastur},~\bfnm{Leonid~A.}\binits{L.~A.}}
(\byear{1996}).
\btitle{{Asymptotic properties of large random matrices with independent
  entries}}.
\bjournal{Journal of Mathematical Physics}
\bvolume{37}
\bpages{5033--5060}.
\bdoi{10.1063/1.531589}
\end{barticle}
\endbibitem

\bibitem{landon2019applications}
\begin{barticle}[author]
\bauthor{\bsnm{Landon},~\bfnm{Benjamin}\binits{B.}} \AND
  \bauthor{\bsnm{Sosoe},~\bfnm{Philippe}\binits{P.}}
(\byear{2018}).
\btitle{Applications of mesoscopic CLTs in random matrix theory}.
\bjournal{arXiv: 1811.05915}
\end{barticle}
\endbibitem

\bibitem{Kevin2018}
\begin{barticle}[author]
\bauthor{\bsnm{Lee},~\bfnm{Ji}\binits{J.}} \AND
  \bauthor{\bsnm{Schnelli},~\bfnm{Kevin}\binits{K.}}
(\byear{2018}).
\btitle{Local law and Tracy-Widom limit for sparse random matrices}.
\bjournal{Probability Theory and Related Fields}
\bvolume{171}.
\bdoi{10.1007/s00440-017-0787-8}
\end{barticle}
\endbibitem

\bibitem{Lei2016}
\begin{barticle}[author]
\bauthor{\bsnm{Lei},~\bfnm{Jing}\binits{J.}}
(\byear{2016}).
\btitle{{A goodness-of-fit test for stochastic block models}}.
\bjournal{Annals of Statistics}
\bvolume{44}
\bpages{401--424}.
\bdoi{10.1214/15-AOS1370}
\end{barticle}
\endbibitem

\bibitem{Lytova2009a}
\begin{barticle}[author]
\bauthor{\bsnm{Lytova},~\bfnm{A.}\binits{A.}} \AND
  \bauthor{\bsnm{Pastur},~\bfnm{L.}\binits{L.}}
(\byear{2009}).
\btitle{{Central limit theorem for linear eigenvalue statistics of random
  matrices with independent entries}}.
\bjournal{Annals of Probability}
\bvolume{37}
\bpages{1778--1840}.
\bdoi{10.1214/09-AOP452}
\end{barticle}
\endbibitem

\bibitem{Mingo2006}
\begin{barticle}[author]
\bauthor{\bsnm{Mingo},~\bfnm{James~A.}\binits{J.~A.}} \AND
  \bauthor{\bsnm{Speicher},~\bfnm{Roland}\binits{R.}}
(\byear{2006}).
\btitle{Second order freeness and fluctuations of random matrices: I. Gaussian
  and Wishart matrices and cyclic Fock spaces}.
\bjournal{Journal of Functional Analysis}
\bvolume{235}
\bpages{226-270}.
\bdoi{https://doi.org/10.1016/j.jfa.2005.10.007}
\end{barticle}
\endbibitem

\bibitem{Sinai}
\begin{barticle}[author]
\bauthor{\bsnm{Sinai},~\bfnm{Ya.}\binits{Y.}} \AND
  \bauthor{\bsnm{Soshnikov},~\bfnm{A.}\binits{A.}}
(\byear{1998}).
\btitle{Central limit theorem for traces of large random symmetric matrices
  with independent matrix elements}.
\bjournal{Boletim da Sociedade Brasileira de Matem{\'a}tica -
  Bulletin/Brazilian Mathematical Society}
\bvolume{29}
\bpages{1-24}.
\bdoi{10.1007/BF01245866}
\end{barticle}
\endbibitem

\bibitem{Tikhomirov1981}
\begin{barticle}[author]
\bauthor{\bsnm{Tikhomirov},~\bfnm{A.~N.}\binits{A.~N.}}
(\byear{1980}).
\btitle{On the convergence rate in the central limit theorem for weakly
  dependent random variables}.
\bjournal{Theory of Probability and Its Applications}
\bvolume{XXV}
\bpages{790--809}.
\bdoi{10.1002/0471667196.ess2714.pub2}
\end{barticle}
\endbibitem

\bibitem{Yizhe2019}
\begin{barticle}[author]
\bauthor{\bsnm{Zhu},~\bfnm{Yizhe}\binits{Y.}}
(\byear{2020}).
\btitle{A graphon approach to limiting spectral distributions of Wigner-type
  matrices}.
\bjournal{Random Structures \& Algorithms}
\bvolume{56}
\bpages{251-279}.
\bdoi{https://doi.org/10.1002/rsa.20894}
\end{barticle}
\endbibitem

\end{thebibliography}

\end{document}